\documentclass[a4paper, 10pt]{amsart}
\usepackage{mathtools, microtype, mathrsfs, xfrac, amssymb, enumerate, xspace}
\usepackage[alphabetic]{amsrefs}
\usepackage{graphicx}
\usepackage{float}
\usepackage[export]{adjustbox}
\usepackage{subcaption}
\usepackage[latin1]{inputenc}
\usepackage{tikz-cd}
\usepackage[percent]{overpic}
\usepackage[all,cmtip]{xy}
\usepackage{hyperref}
\hypersetup{
	colorlinks=true,
	linkcolor=blue,
	filecolor=magenta,      
	urlcolor=cyan,
}
\usepackage{faktor}
\usepackage{cleveref}
\usepackage{comment}

\numberwithin{equation}{section}

\numberwithin{subsection}{section}

\allowdisplaybreaks[1]

\newenvironment{enumeratea}
{\begin{enumerate}[\upshape (a)]}
	{\end{enumerate}}

\newenvironment{enumerate1}
{\begin{enumerate}[\upshape (1)]}
	{\end{enumerate}}

\newtheorem*{namedtheorem}{\theoremname}
\newcommand{\theoremname}{testing}

\newtheorem*{maintheorem}{Theorem}

\newtheorem{theorem}{Theorem}[section]
\newtheorem{proposition}[theorem]{Proposition}
\newtheorem{proposition-definition}[theorem]
{Proposition-Definition}
\newtheorem{corollary}[theorem]{Corollary}
\newtheorem{lemma}[theorem]{Lemma}

\theoremstyle{definition}
\newtheorem{definition}[theorem]{Definition}

\newtheorem{remark}[theorem]{Remark}

\theoremstyle{remark}

\newcommand\nome{testing}
\newcommand\call[1]{\label{#1}\renewcommand\nome{#1}}
\newcommand\itemref[1]{\item\label{\nome;#1}}

\newcommand\refpart[2]{(\ref{#1;#2})}

\renewcommand{\mathcal}{\mathscr}

 \newcommand\cB{\mathcal{B}}
\newcommand\cC{\mathcal{C}} \newcommand\cD{\mathcal{D}}

 \newcommand\cL{\mathcal{L}}
\newcommand\cM{\mathcal{M}} \newcommand\cN{\mathcal{N}}
\newcommand\cO{\mathcal{O}} \newcommand\cP{\mathcal{P}}
 \newcommand\cR{\mathcal{R}}
 
\newcommand\cU{\mathcal{U}} \newcommand\cV{\mathcal{V}}
 \newcommand\cX{\mathcal{X}}
\newcommand\cY{\mathcal{Y}} \newcommand\cZ{\mathcal{Z}}

\renewcommand\AA{\mathbb{A}}

\newcommand\GG{\mathbb{G}} \newcommand\HH{\mathbb{H}}

 \newcommand\PP{\mathbb{P}}

 \newcommand\ZZ{\mathbb{Z}}

\newcommand\bC{\mathbf{C}}

 \newcommand\rB{\mathrm{B}}

\newcommand\rma{\mathrm{a}}

\newcommand\rmm{\mathrm{m}}

 \newcommand\bfp{\mathbf{p}}
\newcommand\bfq{\mathbf{q}} \newcommand\bfr{\mathbf{r}}

 \newcommand\frkm{\mathfrak{m}}

\newcommand\arr{\ifinner\to\else\longrightarrow\fi}

\newcommand\arrto{\ifinner\mapsto\else\longmapsto\fi}

\newcommand\larr{\longrightarrow}

\newcommand{\xarr}{\xrightarrow}

\newcommand{\hooklongrightarrow}{\lhook\joinrel\longrightarrow}

\renewcommand\H{\operatorname{H}}

\newcommand{\eqdef}{\mathrel{\smash{\overset{\mathrm{\scriptscriptstyle def}} =}}}

\newcommand\im[1]{\operatorname{im}(#1)}

\renewcommand\th{^\text{th}}

\def\displaytimes_#1{\mathrel{\mathop{\times}\limits_{#1}}}

\def\displayotimes_#1{\mathrel{\mathop{\bigotimes}\limits_{#1}}}

\newcommand\pic{\operatorname{Pic}}

\newcommand\spec{\operatorname{Spec}}

\newcommand\codim{\operatorname{codim}}

\newcommand\id{\mathrm{id}}

\newcommand\pr{\operatorname{pr}}

\newcommand{\underaut}{\mathop{\underline{\mathrm{Aut}}}\nolimits}

\newlength{\ignora}

\renewcommand{\setminus}{\smallsetminus}

\newcommand{\gm}{\GG_{\rmm}}

\newcommand{\GL}{\mathrm{GL}}

\newcommand{\ga}{\GG_{\rma}}

\newcommand{\ds}[1]{[\mspace{-2mu}[#1]\mspace{-2mu}]}

\DeclareFontFamily{U}{mathx}{\hyphenchar\font45}
\DeclareFontShape{U}{mathx}{m}{n}{
	<5> <6> <7> <8> <9> <10>
	<10.95> <12> <14.4> <17.28> <20.74> <24.88>
	mathx10
}{}
\DeclareSymbolFont{mathx}{U}{mathx}{m}{n}
\DeclareFontSubstitution{U}{mathx}{m}{n}
\DeclareMathAccent{\widecheck}{0}{mathx}{"71}
\DeclareMathAccent{\wideparen}{0}{mathx}{"75}

\renewcommand{\epsilon}{\varepsilon}

\newcommand{\cha}{\operatorname{char}}

\newcommand{\Mbar}{\overline{\cM}}

\newcommand{\Mtilde}{\widetilde{\mathcal M}}

\newcommand{\Cbar}{\overline{\cC}}

\newcommand{\Ctilde}{\widetilde{\mathcal C}}

\newcommand{\Dtilde}{\widetilde{\Delta}}

\newcommand{\ThTilde}{\widetilde{\Theta}}

\newcommand{\cu}{^{\rm c}}

\newcommand{\sing}{^{\rm sing}}

\newcommand{\aff}[1][k]{(\operatorname{Aff}/#1)}

\newcommand{\ag}{[\AA^{1}/\gm]}

\newcommand{\ch}[1][*]{\operatorname{CH}^{#1}}

\newcommand{\mbar}{\overline{\mathcal M}}

\newcommand{\mt}{\widetilde{\mathcal M}}

\newcommand{\dtil}{\widetilde{\Delta}}

\newcommand{\htil}{\widetilde{\HH}}

\newcommand{\modangelo}{\marginpar{}}

\begin{document}
	
	\title[Stable cuspidal curves and $\Mbar_{2,1}$]{Stable cuspidal curves \\ and the integral Chow ring of $\Mbar_{2,1}$}
	\author[A. Di Lorenzo]{Andrea Di Lorenzo}
	\address[A. Di Lorenzo]{Humboldt Universit\"{a}t zu Berlin, Germany}
	\email{andrea.dilorenzo@hu-berlin.de}
	\author[M. Pernice]{Michele Pernice}
	\address[M. Pernice]{Scuola Normale Superiore, Pisa, Italy}
	\email{michele.pernice@sns.it}
	\author[A. Vistoli]{Angelo Vistoli}
	\address[A. Vistoli]{Scuola Normale Superiore, Pisa, Italy}
	\email{angelo.vistoli@sns.it}
	\maketitle
	\begin{abstract}
		In this paper we introduce the moduli stack $\Mtilde_{g,n}$ of $n$-marked stable at most cuspidal curves of genus $g$ and we use it to determine the integral Chow ring of $\Mbar_{2,1}$. Along the way, we also determine the integral Chow ring of $\Mbar_{1,2}$.
	\end{abstract}
	\section*{Introduction}\label{sec:intro}
	Rational Chow rings of moduli spaces of curves have been the subject of intensive research in the last forty years, since Mumford's first investigation of the topic (\cite{Mum}). Nevertheless, rational Chow rings of moduli of \emph{stable} curves of genus larger than $1$ have proved to be hard to compute: the complete calculations have been carried out only for $\overline{M}_2$ (\cite{Mum}), $\overline{M}_{2,1}$ and $\overline{M}_3$ (\cites{Fab,Fab1}).
	
	Integral Chow rings of moduli \emph{stacks} of curves, introduced in \cite{EG}, are even harder to study, but they contain way more information: for instance, rational Chow rings of moduli of hyperelliptic curves are trivial, but their integral counterparts are not (see \cites{VisM2,EF,Dil,Per}); and similarly for $\mathcal{M}_3\smallsetminus\mathcal{H}_3$, the stack of smooth non-hyperelliptic curves of genus three (see \cite{DLFV}).
	
	For what concerns the moduli stack of stable curves of genus $>1$, the only result available in the integral case is the computation of the Chow ring of $\Mbar_2$ by Larson in \cite{Lars} (the first and the third authors subsequently reproved Larson's theorem using a different approach, see \cite{DLV}).
	
	The goal of this paper is to compute the integral Chow ring of $\Mbar_{2,1}$ using a new geometric approach, involving  the stack of what we call \emph{stable $A_2$-curves} (these are curves with only nodes or cusps as singularities).
	
	\begin{maintheorem}[\ref{thm:chow Mbar21}]
		Suppose that the ground field has characteristic $\neq 2,3$. Then we have
		\[ \ch(\Mbar_{2,1})\simeq \ZZ[\lambda_1,\psi_1,\vartheta_1,\lambda_2,\vartheta_2]/(\alpha_{2,1},\alpha_{2,2}, \alpha_{2,3}, \beta_{3,1}, \beta_{3,2}, \beta_{3,3}, \beta_{3,4}), \]
		where
		\begin{enumeratea}
			
			\item the $\lambda_i$ are the Chern classes of the Hodge bundle,
			
			\item the cycle $\psi_{1}$ is the first Chern class of the conormal bundle of the tautological section $\Mbar_{2, 1} \arr \Cbar_{2,1}$, where $\Cbar_{2,1} \arr \Mbar_{2,1}$ is the universal curve,
			
			\item the cycle $\vartheta_1$ is the fundamental class of the locus of marked curves with a separating node, 
			
			\item the cycle $\vartheta_2$ is the fundamental class of the locus of marked curves with a marked separating node,
			
		\end{enumeratea}
		
		and the relations are
		\begin{align*}
			\alpha_{2,1}&=\lambda_2-\vartheta_2-\psi_1(\lambda_1-\psi_1),\\
			\alpha_{2,2}&=24\lambda_1^2-48\lambda_2,\\
			\alpha_{2,3}&=\vartheta_1(\lambda_1+\vartheta_1),\\
			\beta_{3,1}&=20\lambda_1\lambda_2-4\lambda_2\vartheta_1,\\
			\beta_{3,2}&=2\psi_1\vartheta_2,\\
			\beta_{3,3}&=\vartheta_2(\vartheta_1+\lambda_1-\psi_1),\\
			\beta_{3,4}&=2\psi_1(\lambda_1 +\vartheta_1)(7\psi_1-\lambda_1) - 24\psi_1^3.\\
		\end{align*}
	\end{maintheorem}
	
	In the above we identify, as usual, $\Mbar_{2,1}$ with the universal family over $\Mbar_{2}$ via the stabilization map $\Mbar_{2,1} \arr \Mbar_{2}$. Viewed in this way, $\psi_{1}$ is the first Chern class of the dualizing sheaf $\omega_{\Mbar_{2,1}/\Mbar_{2}}$.
	
	In \Cref{cor:rational chow} we reprove Faber's result on the rational Chow ring of $\overline{M}_{2,1}$ and we further extend it over any base field of characteristic $\neq 2,3$. Moreover, to show how our strategy works in a simpler case, we also determine the integral Chow ring of $\Mbar_{1,2}$.
	\begin{maintheorem}[\ref{thm:chow Mbar12}]
		Suppose that the ground field has characteristic $\neq 2,3$. Then we have
		\[ \ch(\Mbar_{1,2}) \simeq \ZZ[\lambda_1,\mu_1]/(\mu_1(\lambda_1+\mu_1),24\lambda_1^2) \]
		where $\lambda_1$ is the first Chern class of the Hodge line bundle and $\mu_1:=\bfp_*[\Mbar_{1,1}]$ is the fundamental class of the universal section $\bfp:\Mbar_{1,1}\to\Mbar_{1,2}$ of $\Mbar_{1,2}\to\Mbar_{1,1}$.
	\end{maintheorem}
	
	\subsection*{Stable $A_2$-curves and strategy of proof}
	The calculations of integral Chow rings of stacks of stable curves that have been carried out so far have been for stacks with a presentation of the form $[X/G]$, where $G$ is an affine algebraic group, and $X$ is an open subscheme of a representation of $G$; we will call this a \emph{good presentation}.
	
	In other cases, like $\Mbar_{2,1}$, the stack of interest, let us call it $\cX$, is not known to have such a presentation, but contains a closed subscheme $\cY \subseteq \cX$ such that both $\cY$ and $\cX \setminus \cY$ have a good presentation, and we can compute both $\ch(\cY)$ and $\ch(\cX)$. One can try to use the localization sequence
	\[
	\operatorname{CH}^{*-c}(\cY) \arr \ch(\cX) \arr \ch(\cX \setminus \cY) \arr 0
	\]
	where $c$ is the codimension of the regular immersion $\cY \hookrightarrow \cX$ but this gives no information on the kernel of the pushforward $\operatorname{CH}^{*-c}(\cY) \arr \ch(\cX)$.

	Our approach, first introduced in \cite{fulghesu} in the context of intersection theory on stacks of curves nodal curves of genus~$0$, exploits a patching technique (see Lemma~\ref{lm:Atiyah}) that is at heart of the Borel-Atiyah-Segal-Quillen localization theorem, and has been used by many authors working on equivariant cohomology, equivariant Chow rings and equivariant K-theory (see the discussion in the introduction of \cite{DLV}). This works when the top Chern class of the normal bundle of $\cY$ in $\cX$ is not a zero-divisor in $\ch(\cY)$. Unfortunately when $\cY$ is Deligne--Mumford then $\ch[i](\cY)$ is torsion for $i > \dim \cY$, which means that the condition can never be satisfied. Thus to apply this to stacks of stable curves one needs to enlarge them so that the condition has a chance to be satisfied, compute the Chow ring of the enlarged stack, then use the localization sequence to compute the additional relations coming from the difference between the enlarged stack and the one we are interested in.
	
	In the first section of this paper we introduced the moduli stack $\Mtilde_{g,n}$ of \emph{stable $A_2$-curves}, obtained by adding to $\Mbar_{g,n}$ curves with cuspidal singularities; the stability condition is that the canonical is still required to be ample. This is neither Deligne-Mumford nor separated (thus the stability condition above is only a weak one, and does not ensure GIT stability), but it is still a smooth quotient stack, hence it has a well defined integral Chow ring. The same holds for the universal stable $A_2$-curve $\Ctilde_{g,n}\arr\Mtilde_{g,n}$. In particular, as $\Mbar_{2,1}$ is contained in $\Ctilde_2$ as an open substack, we can break down the computation of $\ch(\Mbar_{2,1})$ in two main steps:
	\begin{enumerate1}
		\item the determination of $\ch(\Ctilde_2)$, which is the content of \Cref{prop:chow Ctilde2};
		\item the determination of the image of $\ch(\Ctilde_2\smallsetminus\Mbar_{2,1})\arr \ch(\Ctilde_2)$, i.e. the cycles coming from the locus of stable $A_2$-curves with cuspidal singularities. This is first done abstractly in terms of fundamental classes in \Cref{thm:chow Mbar21 abs}; afterwards we proceed with the explicit computations.
	\end{enumerate1}
	
	For the first step, we consider the stratification of $\Ctilde_2$ by closed substacks given by
	\[ \ThTilde_2 \subset \ThTilde_1 \subset \Ctilde_2 \]
	where $\ThTilde_1$ is the locus of marked curves with a separating node and $\ThTilde_2$ is the stratum of marked curves with a marked separating node. We compute $\ch(\Ctilde_2\smallsetminus\ThTilde_1)$, $\ch(\ThTilde_1\smallsetminus\ThTilde_2)$ and $\ch(\ThTilde_2)$ separately (each of these stacks has a good presentation); then we use the fact that the top Chern classes of the normal bundles of $\ThTilde_2$ in $\Ctilde_2$, and of $\ThTilde_1\smallsetminus\ThTilde_2$ in $\Ctilde_2 \setminus \ThTilde_2$ are not zero-divisors, which allows us to use the patching technique to get $\ch(\Ctilde_2)$.

\subsection*{Relations with previous work} A very natural approach to the problem of computing Chow rings of smooth stacks with a stratification for which we know the Chow rings of the strata is to use higher Chow groups to get a handle on the kernel. For the stack $\Mbar_{2}$ this is carried out in \cite{Lars} (see also \cite{bae-schmitt-1, bae-schmitt-2}). Our approach, explained above, is different, and does not use higher Chow groups at all. It was first introduced in \cite{fulghesu} in intersection theory, and used to give an alternate proof of Larson's theorem on $\Mbar_{2}$ in \cite{DLV}.

The idea of defining alternate compactifications of the moduli spaces of smooth curves by considering curves with more complicated singularities than nodes is already in the literature, starting from \cite{schubert} (see also \cite{hassett-hyeon-1, hassett-hyeon-2}), and from a more stack-theoretic standpoint in the work of Smyth (see \cite{smyth-survey}), continued by Alper, Fedorchuk and Smyth \cite{alper-fedorchuk-smyth-1, alper-fedorchuk-smyth-2, alper-fedorchuk-smyth-3, alper-fedorchuk-smyth-4}. Their perspective, however, is very different, as they look for stacks that are proper and Deligne--Mumford, in order to study the birational geometry of moduli spaces of curves.

In \cite{fulghesu, bae-schmitt-1, bae-schmitt-1} the authors study Chow rings of stacks of curves with positive dimensional automorphism groups; unlike us, they impose no stability condition, and their curves are all nodal.

\subsection*{Future prospects} The approach of the present paper should be applicable, maybe after inverting a few primes, to other moduli stacks with a stratification such that the Chow rings of the strata are computable. For the relatively simple cases considered in this paper it is enough to consider at most cuspidal curves; in other cases one needs more complicated singularities. The second author's PhD thesis will include several results on stacks of stable $A_{n}$-curves, as well a presentation for the Chow ring of $\Mbar_{3}$ with coefficients in $\ZZ[1/6]$, obtained using $A_{3}$-curves, that is, curves with only nodes, cusps or tacnodes. Of course the presentation for the rational Chow ring of $\Mbar_{3}$ is an old theorem of Faber \cite{Fab1}; our technique gives a completely different approach, and a more refined result.

The general principle seems to be that the more singularities you allow, the more likely it is that the patching condition be satisfied, making it possible to compute the Chow ring of the larger stack. Then one needs to compute the classes of various loci of unwanted curves, something that can be difficult; allowing very complicated singularities makes it harder. So, it is not clear what the limits of the methods are; still, we think it worth investigating what it can give, for example, for $\Mbar_{3,1}$ and $\Mbar_{4}$, whose rational Chow rings are not known.

	\subsection*{Outline of the paper}
	In Section \ref{sec:cusp} we introduce stable $A_2$-curves (\Cref{def:A2 curves}) and the associated stack $\Mtilde_{g,n}$. We prove that $\Mtilde_{g,n}$ is a quotient stack (\Cref{thm:quotient-stack}) and we study the normalization of $\Mtilde_{g,n}\smallsetminus\Mbar_{g,n}$: in particular, we prove in \Cref{thm:cuspidal-description} that it is isomorphic to $\Mtilde_{g-1,n+1}\times\ag$ via a certain \emph{pinching} morphism, a result that will be essential for our computations.
	
	In Section \ref{sec:chow Mbar12} we show how our strategy works in a simple case, namely for computing $\ch(\Mbar_{1,2})$ (\Cref{thm:chow Mbar12}): we first determine $\ch(\Ctilde_{1,1})$ (\Cref{prop:chow Ctilde11}) and then we excise the cuspidal locus.
	
	Section \ref{sec:chow cusp} is devoted to the computation of $\ch(\Ctilde_2)$ (\Cref{prop:chow Ctilde2}). As explained before, this result is obtained by patching together the explicit presentations of $\ch(\Ctilde\smallsetminus\ThTilde_1)$ (\Cref{prop:chow C minus Theta1}), of $\ch(\ThTilde_1\smallsetminus\ThTilde_2)$ (\Cref{prop:chow ThTilde1 minus ThTilde2}) and $\ch(\ThTilde_2)$ (\Cref{prop:chow ThTilde2}).
	
	In Section \ref{sec:abstract} we give an abstract characterization of $\ch(\Mbar_{2,1})$. More precisely, we prove in \Cref{thm:chow Mbar21 abs} that
	\[\ch(\Mbar_{2,1})\simeq\ch(\Ctilde_2)/(J,[\Ctilde_2^{\rm c}],[\Ctilde_2^{\rm E}],c''_*\rho^*T) \]
	where $J$ is the ideal generated by the relations coming from $\Mbar_2$ and two of the other generators are the fundamental classes of certain loci in $\Ctilde_2$.
	
	Finally, in Section \ref{sec:concrete} we compute the integral Chow ring of $\Mbar_{2,1}$ (\Cref{thm:chow Mbar21}) by determining explicit expressions for $[\Ctilde_2^{\rm c}]$ and $[\Ctilde_2^{\rm E}]$ (\Cref{prop:class Ctilde2 c} and \Cref{prop:class Ctilde2 E}). Again, a key ingredient for this computation is the patching Lemma. In the last part of the Section we compare our presentation of $\ch(\overline{M}_{2,1})_{\mathbb{Q}}$ with the one obtained by Faber.
	
	\subsection*{Notation}
	For the convenience of the reader, we summarize here the notation for most of the stacks appearing in the paper.
	\begin{enumerate}
		\item $\Mtilde_{g,n}$ - the stack of $n$-marked stable $A_2$-curves of genus $g$.
		\item $\Mbar_{g,n}$ - the stack of $n$-marked stable curves of genus $g$.
		\item $\Ctilde_{g,n}$ - the universal $n$-marked stable $A_2$-curve of genus $g$.
		\item $\Dtilde_1$ - the stack of stable $A_2$-curves of genus two with a separating node.
		\item $\Delta_1$ - the stack of stable curves of genus two with a separating node.
		\item $\ThTilde_1$ - the stack of $1$-pointed stable $A_2$-curves of genus two with a separating node.
		\item $\Theta_1$ - the stack of $1$-pointed stable curves of genus two with a separating node.
		\item $\ThTilde_2$ - the stack of $1$-pointed stable $A_2$-curves of genus two with a marked separating node.
		\item $\Theta_2$ - the stack of $1$-pointed stable curves of genus two with a marked separating node.
	\end{enumerate}
	
	\subsection*{Acknowledgments}
	We warmly thank Carel Faber for sharing his results on the rational Chow ring of $\overline{M}_{2,1}$ with us. We also thank Martin Bishop for spotting a mistake in the proof of \Cref{thm:chow Mbar12} and for suggesting a correction. 
	
	Finally, we are very grateful to the referees, who did an absolutely outstanding job. Their comments and suggestions greatly improved the quality of the paper.
	
	\section*{Notations and conventions}
	
	We work over a base commutative ring $k$, such that $2$ and $3$ are invertible in $k$. We will almost exclusively use the case in which $k$ is a field; but in the proof of Theorem~\ref{thm:cuspidal-description} the added generality will be useful, as we will have to assume $k = \ZZ[1/6]$. From Section~\ref{sec:chow Mbar12}, $k$ will be a field of characteristic different from $2$ and $3$.
	
	All schemes, algebraic spaces and morphisms are going to be defined over $k$. All stacks will be over the \'{e}tale site $\aff$ of affine schemes over $k$ (we could take all schemes, it does not really make a difference). For algebraic stacks and algebraic spaces we will follow the conventions of \cite{knutson} and \cite{laumon-moret-bailly}; in particular, they will have a separated diagonal of finite type.
	
	A \emph{quotient stack} $\cX$ is a stack over $k$, such that there exists an algebraic space $X \arr \spec k$, with an action of an affine flat group scheme of finite type $G \arr \spec k$, such that $\cX$ is isomorphic to $[X/G]$.
	
	\section{The stack of stable $A_2$-curves of fixed genus} \label{sec:cusp}
	
	\begin{definition}\label{def:A2 curves}
		An \emph{$A_2$-curve} over a scheme $S$  is a proper flat finitely presented morphism of schemes $C \arr S$, whose geometric fibers are reduced connected curves with only nodes or cusps as singularities.
		
		If $S$ is a scheme, an $n$-marked stable $A_2$-curve of genus $g$ over $S$ is an $A_{2}$-curve $C \arr S$ with $n$ sections $s_{1}$, \dots, $s_{n}\colon S \arr C$, such that, if we denote by $\Sigma_{i}$ the image of $s_{i}$ in $C$,
		\begin{enumerate1}
			
			\item the $\Sigma_{i}$ are disjoint,
			
			\item they are contained in the smooth locus of $C \arr S$, and
			
			\item if $\omega_{C/S}$ is the relative dualizing sheaf, the invertible sheaf $\omega_{C/S}(\Sigma_{1} +  \dots + \Sigma_{n})$ is relatively ample on $S$. 
			
		\end{enumerate1}
		
		In the most general, and correct, definition of an $A_2$-curve $C \arr S$, one should not assume that $C$ is a scheme, but an algebraic space; for the purposes of this paper, this will not be needed, but can be done without making any essential changes.
		
	\end{definition}
	
	We have the following standard result.
	
	\begin{proposition}\label{prop:openness}
		Let $C \arr S$ a proper flat finitely presented morphism with $n$-sections $S \arr C$. There exists an open subscheme $S' \subseteq S$ with the property that a morphism $T \arr S$ factors through $S'$ if and only if the projection $T\times_{S} C \arr T$, with the sections induced by the $s_{i}$, is a stable $n$-pointed $A_{2}$-curve.
	\end{proposition}
	
	\begin{proof}
	\modangelo It is well known that a small deformation of a curve with cuspidal or nodal singularities still has cuspidal or nodal singularities. Hence, after restricting to an open subscheme of $S$ we can assume that $C \arr S$ is an $A_{2}$-curve. By further restricting $S$ we can assume that the sections land in the smooth locus of $C \arr S$, and are disjoint. Then the result follows from openness of ampleness for invertible sheaves. 
	\end{proof}
	
	If $S' \arr S$ is a morphism of schemes, and $C \arr S$ is an $n$-marked stable $A_2$-curve of genus~$g$, the projection $S'\times_{S}C \arr S'$ is an $n$-marked stable $A_2$-curve of genus~$g$. Thus, there is an obvious stack $\mt_{g,n, k}$ of stable $A_2$-curves of genus~$g$ over $\aff$, whose category of sections over an affine scheme $S$ is the groupoid of stable $A_2$-curves of genus~$g$ over $S$. Mostly we will omit the indication of the base field, and simply write $\mt_{g,n}$.
	
	As customary, we will denote $\mt_{g,0}$ by $\mt_{g}$.

	\begin{theorem}\label{thm:quotient-stack}
		The stack $\mt_{g,n}$ is a smooth algebraic stack of finite type over $k$. Furthermore, it is a quotient stack; more precisely, there exists a smooth quasi-projective scheme $X$ with an action of $\GL_{N}$ for some positive integer $N$, such that $\mt_{g,n} \simeq [X/\GL_{N}]$.
		
		If $k$ is a field, then $\mt_{g,n}$ is connected.
	\end{theorem}
	
	\begin{proof}
		It follows from Lemma~\ref{lem:boundedness} that there exists a positive integer $m$ with the property that if $\Omega$ is an algebraically closed field with a homomorphism $k \arr \Omega$, and $(C, p_{1}, \dots p_{n})$ is in $\mt_{g,n}(\Omega)$, then $\omega_{C/\Omega}(p_{1}+ \dots + p_{n})^{\otimes m}$ is very ample, and $\H^{1}\bigl(C, \omega_{C/\Omega}(p_{1}+ \dots + p_{n})^{\otimes m}\bigr) = 0$.
		
		The invertible sheaf $\omega_{C/\Omega}(p_{1}+ \dots + p_{n})^{\otimes m}$ has Hilbert polynomial $P(t) = g + m(2g-2 + n)t$, and defines an embedding $C \subseteq \PP_{\Omega}^{N-1}$, where $N \eqdef m(2g-2 + n) - g + 1$. If $\pi\colon C \arr S$ is a stable $n$-pointed $A_{2}$-curve, and $\Sigma_{1}$, \dots,~$\Sigma_{n} \subseteq C$ are the images of the sections, then $\pi_{*}\omega_{C/S}(\Sigma_{1}+ \dots + \Sigma_{n})^{\otimes m}$ is locally free sheaf of rank $N$, and its formation commutes with base change, because of Grothendieck's base change theorem.
		
		Call $X$ the stack over $k$, whose sections over a scheme $S$ consist of a stable $n$-pointed $A_{2}$-curve as above, and an isomorphism $\cO_{S}^{N} \simeq \pi_{*}\omega_{C/S}(\Sigma_{1}+ \dots + \Sigma_{n})^{\otimes m}$ of sheaves of $\cO_{S}$-modules. Since $\pi_{*}\omega_{C/S}(\Sigma_{1}+ \dots + \Sigma_{n})^{\otimes m}$ is very ample, the automorphism group of an object of $X$ is trivial, and $X$ is equivalent to its functor of isomorphism classes.
		
		Call $H$ the Hilbert scheme of subschemes of $\PP^{N-1}_{k}$ with Hilbert polynomial $P(t)$, and $D \arr H$ the universal family. Call $F$ the fiber product of $n$ copies of $D$ over $S$, and $C \arr F$ the pullback of $D \arr H$ to $F$; there are $n$ tautological sections $s_{1}$, \dots,~$s_{n}\colon F \arr C$. Consider the largest open subscheme $F'$ of $F$ such that the restriction $C'$ of $C$, with the restrictions of the $n$ tautological sections, is a stable $n$-pointed $A_{2}$-curve, as in Proposition~\ref{prop:openness}. Call $Y \subseteq F'$ the open subscheme whose points are those for which the corresponding curve is nondegenerate, $E \arr Y$ the restriction of the universal family, $\Sigma_{1}$, \dots,~$\Sigma_{n} \subseteq E$ the tautological sections. Call $\cO_{E}(1)$ the restriction of $\cO_{\PP^{N-1}_{Y}}(1)$ via the tautological embedding $E \subseteq \PP^{N-1}_{Y}$; there are two section of the projection $\pic_{E/Y}^{m(2g-2 + n)}\arr Y$ from the Picard scheme parametrizing invertible sheaves of degree $m(2g-2 + n)$, one defined by $\cO_{E}(1)$, the other by $\omega_{E/Y}(\Sigma_{1} + \dots + \Sigma_{n})^{\otimes m}$; let $Z \subseteq Y$ the equalizer of these two sections, which is a locally closed subscheme of $Y$.
		
		Then $Z$ is a quasi-projective scheme over $k$ representing the functor sending a scheme $S$ into the isomorphism class of tuples consisting of a stable $n$-pointed $A_{2}$-curve $\pi\colon C \arr S$, together with an isomorphism of $S$-schemes
		\[
		\PP^{N-1}_{S} \simeq \PP(\pi_{*}\omega_{C/S}(\Sigma_{1} + \dots + \Sigma_{n}))\,.
		\]
		There is an obvious functor $X \arr Z$, associating with an isomorphism $\cO_{S}^{N} \simeq \pi_{*}\omega_{C/S}(\Sigma_{1}+ \dots + \Sigma_{n})^{\otimes m}$ its projectivization. It is immediate to check that $X \arr Z$ is a $\gm$-torsor; hence it is representable and affine, and $X$ is a quasi-projective scheme over $\spec k$.
		
		On the other hand there is an obvious morphism $X \arr \mt_{g,n}$ which forgets the isomorphism $\cO_{S}^{N} \simeq \pi_{*}\omega_{C/S}(\Sigma_{1}+ \dots + \Sigma_{n})^{\otimes m}$; this is immediately seen to be a $\GL_{N}$ torsor. We deduce that $\mt_{g,n}$ is isomorphic to $[X/\GL_{N}]$. This shows that is a quotient stack, as in the last statement; this implies that $\mt_{g,n}$ is an algebraic stack of finite type over $k$.
		
		The fact that $\mt_{g,n}$ is smooth follows from the fact that $A_{2}$-curves are unobstructed.
		
		Finally, to check that $\mt_{g,n}$ is connected it is enough to check that the open embedding $\cM_{g,n} \subseteq \mt_{g,n}$ has a dense image, since $\cM_{g,n}$ is well known to be connected. This is equivalent to saying that every stable $n$-pointed $A_{2}$-curve over an algebraically closed extension $\Omega$ of $k$ has a small deformation that is stable and nodal. Let $(C, p_{1}, \dots, p_{n})$ be a stable $n$-pointed $A_{2}$-curve; the singularities of $C$ are unobstructed, so we can choose a lifting $\overline{C}\arr \spec \Omega\ds{t}$, with smooth generic fiber. The points $p_{i}$ lift to sections $\spec\Omega\ds{t} \arr \overline{C}$, and then the result follows from  Proposition~\ref{prop:openness}.
	\end{proof}
	
	Notice that there are many examples of stable $A_2$-curves over an algebraically closed field whose group of automorphisms is a positive-dimensional affine group (for example, let $C$ an irreducible rational curve whose only singularity is a single cusp, attached by a smooth point to a smooth curve of genus $g-1$; then the automorphisms group of $C$ is an extension of finite group by $\gm$). This means that $\mt_{g}$ is not Deligne--Mumford, and not separated.

	\begin{proposition}\label{prop:hodge-bundle}
		Let $\pi\colon C \arr S$ be a stable $A_2$-curve of genus $g$. Then $\pi_{*}{\omega}_{C/S}$ is a locally free sheaf of rank $g$ on $S$, and its formation commutes with base change.
	\end{proposition}
	
	\begin{proof}
		If $C$ is an $A_2$-curve over a field $k$, the dimension of $\H^{0}(C, \omega_{C/k})$ is $g$; so the result follows from Grauert's theorem when $S$ is reduced. But the versal deformation space of an $A_{2}$-curve over a field is smooth, so every $A_{2}$-curve comes, \'{e}tale-locally, from an $A_{2}$-curve over a reduced scheme, and this proves the result.
	\end{proof}
	
	As a consequence we obtain a locally free sheaf $\htil_{g}$ of rank~$g$ on $\mt_{g, n}$, the \emph{Hodge bundle}.
	
	We will need a classification of stable $A_2$-curves of genus~$2$, and $1$-marked stable $A_2$-curves of genus $1$. The following is straightforward.
	
	\begin{proposition}\call{prop:list} Assume that $k$ is an algebraically closed field. A $1$-marked stable $A_2$-curve of genus $1$ over $k$ is either stable, or an irreducible cuspidal rational cubic.
		
		A stable $A_2$-curve of genus~$2$ over $k$ is of one of the following types.
		
		\begin{enumerate1}
			
			\itemref{1} A stable curve of genus $2$.
			
			\itemref{2} An irreducible curve of geometric genus~$1$ with a node or a cusp.
			
			\itemref{3} An irreducible rational curve with a node and a cusp, or two cusps.
			
			\itemref{5} The union of two $1$-marked stable $A_2$-curves of genus $1$ meeting transversally at the marked points.
			
		\end{enumerate1}
	\end{proposition}
	
	The locus of singular curves in $\mt_{2}$ is a divisor with normal crossing, with two irreducible components $\dtil_{0}$ and $\dtil_{1}$. The component $\dtil_{0}$ is the closure of the locus of irreducible curves; the other component $\dtil_{1}$ is formed by the curves with a separating node, which are those of type \refpart{prop:list}{5}.
	
	\section{The locus of cusps in the universal curve}\modangelo
	
	The stack $\mt_{g}$ contains a closed substack of codimension~$2$, the cuspidal locus, whose points correspond to curves with at least a cusp. Its normalization, which we denote by $\Ctilde\cu_g$, is the stack of curves of genus $g$ with a distinguished cusp. This will play an important role in one of our calculations. Everything that we are going to say in this section generalizes to $\mt_{g,n}$, but this would complicate notation, and the added generality would not be useful to us.

	\subsection{The cuspidal locus in an $A_{2}$-curve}
	Let $C \arr S$ be an $A_2$-curve; we are interested in giving a scheme structure to the locus of cusps $C\cu \subseteq C$, which is closed. This is done as follows.
	
	Let $C\sing \subseteq C$ the singular locus of the map $C \arr S$, with the usual scheme structure given by the first Fitting ideal of the sheaf $\Omega_{C/S}$. Then $C\sing$ is finite over $S$; it is unramified at the nodal points, and it ramifies at the cusps. We define $C\cu$ as the closed subscheme of $C$ defined by the $0\th$ Fitting ideal of $\Omega_{C\sing/S}$. 
	
	\begin{proposition}\label{prop:cuspidal-projection}
		The geometric points of the subscheme $C\cu  \subseteq C$ correspond precisely to the cuspidal points of the geometric fibers of $C \arr S$. The restriction $C\cu \arr S$ is finite and unramified. 
		
		Furthermore, the formation of $C\cu \subseteq C$ commutes with base change on $S$. If $S$ is of finite type over $k$ and the family $C \arr S$ is versal at each point of $S$, then $C\cu$ is a smooth scheme over $k$. 
	\end{proposition}
	
	\begin{proof} \modangelo
	Let us check the first statement. This can be done after restricting to the geometric fibers; so assume that $S = \spec\ell$, where $\ell$ is an algebraically closed field. Furthermore, one can pass to an \'{e}tale cover of $C$, without assuming that $C$ is proper.
		
Let $p \in C(\ell)$ be a singular point of $C$. If $p$ is a node, than it is \'{e}tale-locally of the form $\spec\ell[x,y]/(xy)$; a straightforward calculation reveals that the first Fitting ideal of $\Omega_{C/\ell}$ is $(x,y) \subseteq \cO_{C}$. Hence $\Omega_{C\sing/\ell}$ is zero at $p$, and $p$ is not in $C\cu$.

If $p$ is a cusp, then $C$ is \'{e}tale-locally of the form $\spec\ell[x,y]/(y^{2} - x^{3})$. Then the first Fitting ideal of $\Omega_{C/\ell}$ is $(2y,3x^{3}) = (y, x^{2})$ (recall that we are assuming $\cha k \neq 2$, $3$), so \'{e}tale locally $\cO_{C\sing}$ is $\spec\ell[x,y]/(y, x^{2})$. Hence \'{e}tale-locally $0\th$ Fitting ideal of $\Omega_{C\sing/\ell}$ is $(x)\subseteq \cO_{C\sing}$, and $\cO_{C\cu} = \spec \ell$. This proves the first statement.
		
		Formation of sheaves of differentials, and of Fitting ideals, commute with base change on $S$; hence formation of $C\cu$ also commutes with base change. Therefore, to prove the remaining statements we can assume that $S = \spec R$ is affine and of finite type over $k$. 
		
		The projection $C\cu \arr S$ is proper and representable; hence to check that it is finite it is enough to show that it has finite fibers, which follows from the first statement. 
		
		For the rest of the statements, take a geometric point $p\colon \spec\ell \arr C\cu$. We can base change from $k$ to $\ell$, and assume that $p$ is a rational point, and $k$ is algebraically closed; call $q$ the image of $p$ in $S$. Notice that the definition of $C\cu$ does not require $C \arr S$ to be proper; furthermore, if $U \arr C$ is an \'{e}tale map, then $U\cu$ is the inverse image of $C\cu$ in $U$. 
		
		By \cite[Example~6.46]{mattia-vistoli-deformation}, after passing to an \'{e}tale neighborhood of the image of $q$ in $S$ there exist two elements $a$ and $b$ of $R$, vanishing at $q$, and an \'{e}tale neighborhood $p \arr U \arr C$ of $p$, such that, if we denote by $V \subseteq \AA^{2}_{R}$ the subscheme defined by the equation $y^{2} = x^{3} + ax + b$, there is an \'{e}tale map $U \arr V$ of $S$-schemes sending $p \in U$ into the point over $q$ defined by $x = y = 0$. Then $U\cu$ is the inverse image of $V\cu$; a straightforward calculation shows that $V\cu$ is the subscheme defined by $x = y = a = b$, which proves that $C\cu$ is unramified over $S$. Furthermore, if $C \arr S$ is versal then $S$ is smooth, and $a$ and $b$ are part of a system of parameters around $q$, which shows that $C\cu$ is smooth, as claimed.
	\end{proof}

	\subsection{The pinching construction}\label{sub:pinching}\modangelo
	
	Suppose that we have an $A_{2}$-curve $D$ over an algebraically closed field, with a cuspidal point $q \in C(k)$. Call $C$ the normalization of $D$ at the point $q$; then the inverse image of $q$ consists of a single point $p$. Then it is standard that the pair $(D,q)$ can be recovered from $(C, p)$: as a topological space $D = C$, while the structure sheaf $\cO_{D}$ is the subsheaf of $\cO_{C}$ consisting of functions whose Taylor expansions at $p$ has first order term equal to $0$. This construction, which we call \emph{pinching}, works in families, and plays a fundamental role in this paper; a depiction of it can be found in \Cref{fig:pinch}.
	
   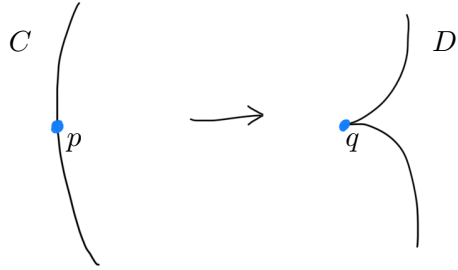
\begin{figure}[h]
 \centering 
 
 \begin{overpic}[width=1\textwidth]{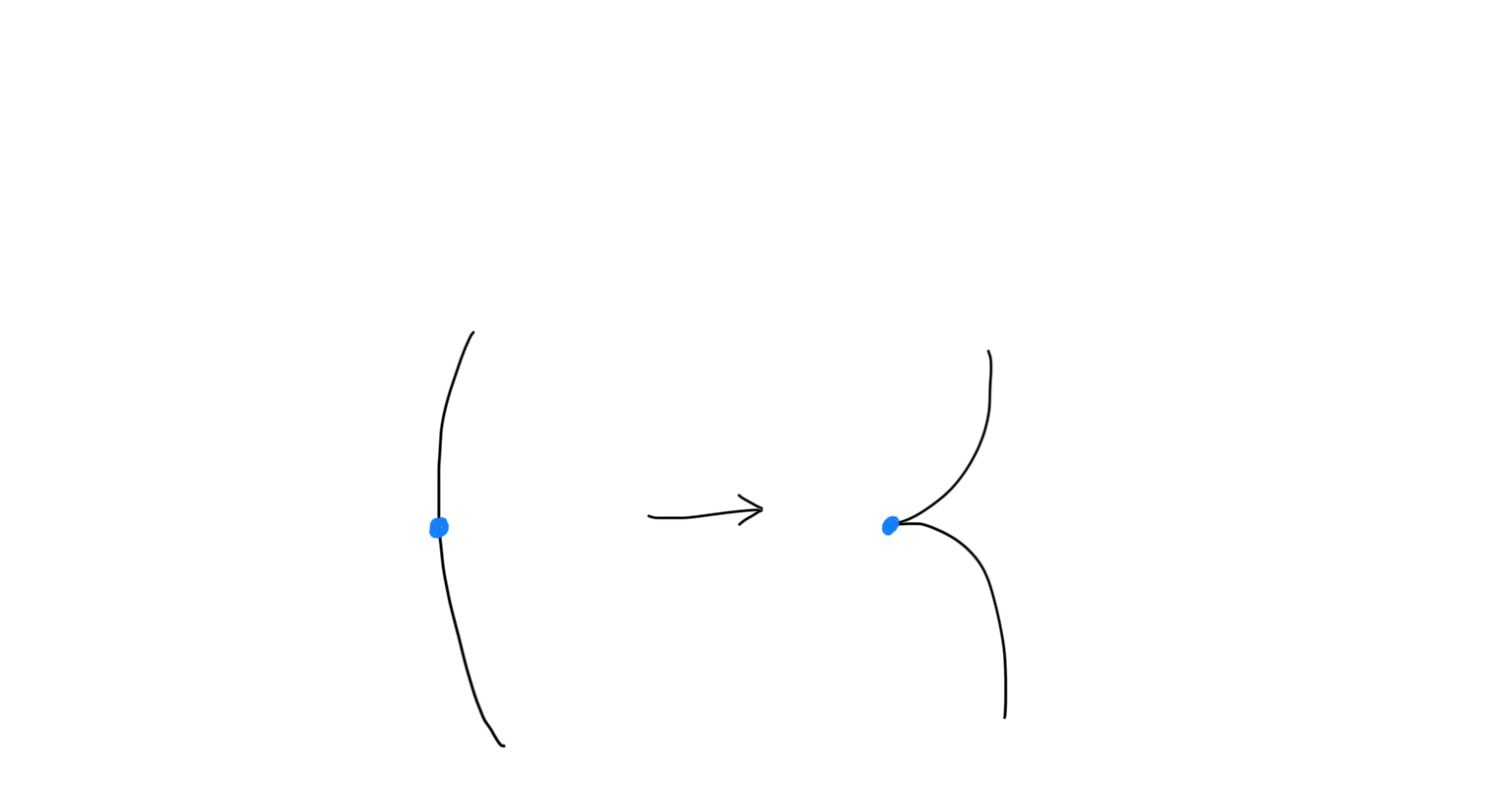}
 \put (24,25) {\large$\displaystyle C$}
 \put (30,15) {\large$\displaystyle p$}
 \put (68,25) {\large$\displaystyle D$}
 \put (59,15) {\large$\displaystyle q$}
 \end{overpic}
 \caption{The pinching construction}
 \label{fig:pinch} 
\end{figure}
	
	Let $\pi\colon C \arr S$ be an $A_2$-curve with a section $\sigma\colon S \arr C$ landing into the smooth locus of $C \arr S$. We will assume that $C$ is a scheme (the pinching construction will also work for algebraic spaces, using the small \'{e}tale site, but we will not need this). We define $\widehat{C} \arr S$ as follows. Denote by $\Sigma \subseteq C$ the image of $\sigma$; call $I_{\Sigma} \subseteq \cO_{C}$ the sheaf of ideals of $\Sigma \subseteq C$. Sending a section $f$ of $\cO_{C}$ into $f - \pi^{*}\sigma^{*}f$ gives a $\cO_{S}$-linear splitting $\cO_{C} \arr I_{\Sigma}$ of the embedding $I_{\Sigma} \subseteq \cO_{C}$.
	
	As a topological space we set $\widehat{C} = C$; the structure sheaf $\cO_{\widehat{C}} \subseteq \cO_{C}$ is defined as the subring of sections $f$ of $\cO_{C}$ with the property that the class $[f - \pi^{*}\sigma^{*}f] \in I_{\Sigma}/I^{2}_{\Sigma}$ is zero. Then $\cO_{\widehat{C}}$ is a sheaf of rings with local stalks. Furthermore, $\pi^{\sharp}\colon \cO_{S} \arr \cO_{C}$ has image contained in $\cO_{\widehat{C}}$, so it defines a morphism of locally ringed spaces $\widehat{C} \arr S$.
	
	\begin{proposition}\label{prop:pinching}
		The morphism $\widehat{C} \arr S$ is an $A_2$-curve. The morphism $C \arr \widehat{C}$ induced by the embedding $\cO_{\widehat{C}} \subseteq \cO_{C}$ is an isomorphism outside of the image of the section $\sigma$. Furthermore, the composite $\Sigma \subseteq C \xarr{\pi} \widehat{C}$ gives a closed and open embedding of $\Sigma$ into the cuspidal locus $\widehat{C}\cu$.
	\end{proposition}
	
	\begin{proof}
		There is an exact sequence
		\[
		0 \arr \cO_{\widehat{C}} \arr \cO_{C} \arr I_{\Sigma}/I^{2}_{\Sigma} \arr 0
		\]
		of sheaves of $\cO_{S}$-modules; since $I_{\Sigma}/I^{2}_{\Sigma}$ is flat over $S$, it is clear that formation of $\widehat{C}$ commutes with base change on $S$, and that $\widehat{C}$ is flat over $\cO_{S}$.
		
		Let us check that $\widehat{C}$ is finitely presented over $S$. Since formation of $\widehat{C}$ commutes with base change on $S$, and $C$ is finitely presented, we can assume that $S$ is noetherian. But the extension $\cO_{\widehat{C}} \subseteq \cO_{C}$ is finite, while $\cO_{C}$ is of finite type over $\cO_{S}$; from this is follows that $\cO_{\widehat{C}}$ is also of finite type over $\cO_{S}$.
		
		Since $C \arr \widehat{C}$ is proper and surjective, it follows that $C\arr S$ is also proper. 
		
		The fact that $\widehat{C}$ is an $A_{2}$-curve can now be checked when $S$ is the spectrum of an algebraically closed field, and it is straightforward.

		The second statement is evident from the construction.
		
		The last statement is proved with a straightforward local calculation.
	\end{proof}
	
The pinching construction is functorial, in the sense that an isomorphism $C \simeq C'$ of pointed $A_{2}$-curves over $S$ induces an isomorphism $\widehat{C} \simeq \widehat{C'}$ of $A_{2}$-curves over $S$.

Suppose that $C$ is a geometrically connected $A_{2}$-curve over $\spec k$, $p \in C(k)$ a smooth rational point. Let $U = \spec R$ be a miniversal deformation space for the pair $(C,p)$ as a $1$-pointed $A_{2}$-curve. Here $R$ is a complete local $k$-algebra; since $A_{2}$-curves are unobstructed, $R$ is a power series algebra $k\ds{x_{1}, \dots, x_{r}}$. Denote by $C_{U} \arr U$ the corresponding $1$-pointed $A_{2}$-curve; this $C_{U}$ is a priori a formal scheme; but the closed fiber $C$ is projective, and an ample line bundle on $C$ extends to $C_{U}$, because $C$ is $1$-dimensional, so it follows from Grothendieck's existence theorem that $C_{U} \arr U$ is a projective scheme.

Analogously, let $V = \spec S$ be a miniversal deformation space for the pinched $A_{2}$-curve $\widehat{C} \arr \spec k$; like in the previous case, $V$ is the spectrum of a power series algebra over $k$. Call $D_{V} \arr V$ the universal $A_{2}$-curve. Its closed fiber $\widehat{C}$ has a distinguished cusp $\widehat{p}$, the image of $p \in C(k)$; call $\Delta \subseteq D_{V}\cu$ the connected component of $D_{V}\cu$ containing $\widehat{p}$; by Proposition~\ref{prop:cuspidal-projection} the $\Delta \arr V$ is a embedding, and $\Delta$ is once again the spectrum of a power series algebra.

By pinching the tautological section $U \arr C_{U}$ we get a family of $A_{2}$-curves $\widehat{C}_{U} \arr U$ with a distinguished section $U \arr \widehat{C}_{U}\cu$; this induces a (non-unique) morphism $U \arr V$, which factors through $\Delta$.

\begin{lemma}\label{lem:cuspidal-isomorphism}
The morphism $U \arr \Delta$ described above is an isomorphism.
\end{lemma}

\begin{proof}
Call $\phi\colon U \arr \Delta$ the morphism above. To produce a morphism in the other direction $\psi\colon \Delta \arr U$, consider the pullback $D_{\Delta} \eqdef \Delta\times_{V}D_{V}$, with its tautological section $\Delta \arr D_{\Delta}\cu$. A straightforward local calculation reveals that the normalization $\overline{D}_{\Delta}$ along $\Delta \subseteq$ is an $A_{2}$-curve, the reduced inverse image of $\Delta \subseteq D_{\Delta}$ in $\overline{D}_{\Delta}$ maps isomorphically onto $\Delta$, thus giving a section $\Delta\arr \overline{D}_{\Delta}$. Hence we get a pointed $A_{2}$-curve $\overline{D}_{\Delta}$, whose closed fiber is $C$; so, there exists a morphism $\psi\colon \Delta\arr U$, with an isomorphism $\overline{D}_{\Delta} \simeq \Delta\times_{U}C_{U}$, such that the section $\Delta \arr \overline{D}_{\Delta}$ corresponds to the section $\Delta\times_{U}C_{U}$ coming from the tautological section $U \arr C_{U}$. 

Consider the composite $\phi\psi\colon U \arr U$. The pullback of $C_{U}$ along $\phi\psi$ is the normalization of $\widehat{C}_{U}$ along the canonical section $U \arr \widehat{C}\cu$; hence it is isomorphic to $C_{U}$. This implies that this pullback is $C_{U}$, which, in turn, together with the fact that $C_{U} \arr U$ is miniversal implies that $\psi\phi$ is an isomorphism. An analogous argument shows that the pullback of $D_{\Delta}$ along $\phi\psi$ is isomorphic to $D_{\Delta}$, so that $\phi\psi$ is also an isomorphism. It follows that $\phi$ is an isomorphism, as claimed.
\end{proof}

	As an application of the pinching construction, let us prove the following lemma, which is needed in the proof of Theorem~\ref{thm:quotient-stack}.
	
	\begin{lemma}\label{lem:boundedness}
		There exists a scheme of finite type $U$ over $k$, and a morphism $U \arr\mt_{g,n}$, which is surjective on geometric points.
	\end{lemma}
	
	\begin{proof}
		Let $\Omega$ be an algebraically closed extension of $k$. Let $(C, p_{1}, \dots, p_{n})$ be a stable $n$-pointed $A_{2}$-curve over $\Omega$, and call $q_{1}$, \dots,~$q_{s}$ the cuspidal points of $C$. Denote by $\overline{C}$ the normalization of $C$ at the $q_{i}$, $\overline{p}_{i}$ the point of $\overline{C}$ lying over $p_{i}$ and $\overline{q}_{i}$ the point of $\overline{C}$ lying over $q_{i}$. Then $(\overline{C}, \overline{p}_{1}, \dots, \overline{p}_{n}, \overline{q}_{1}, \dots, \overline{q}_{s})$ is a $(n+s)$-pointed nodal curve of genus~$g-s$. Hence $s \leq g$; it is enough to produce for each $s$ with $1 \leq s \leq g$ a scheme of finite type and a morphism $U_{s} \arr \mt_{g,n}$, whose image contains all $n$-pointed stable $A_{2}$-curves with exactly $s$ cuspidal points.
		
		Notice that the $n$-pointed curve $(\overline{C}, \overline{p}_{1}, \dots, \overline{p}_{n}, \overline{q}_{1}, \dots, \overline{q}_{s})$ above is not necessarily stable, because the pullback of $\omega_{C/\Omega}(p_{1}+ \dots + p_{n})$ is
		\[
		\omega_{\overline{C}/\Omega}(\overline{p}_{1}+ \dots + \overline{p}_{n} + 2\overline{q}_{1} + \dots + 2\overline{q}_{s})\,;
		\]
		but this instability can only occur when $C$ has a rational component whose only singularity is a single cusp, and intersects the rest of the curve in only one point. For each $i = 1$, \dots,~$s$ let $r_{i}$ be a smooth point on the component of $\overline{C}$ containing $\overline{q}_{i}$, but different from $\overline{q}_{i}$; then the $(n+2s)$-pointed curve
		\[
		(\overline{C}, \overline{p}_{1}+ \dots + \overline{p}_{n} + 2\overline{q}_{1} + \dots + 2\overline{q}_{s}, r_{1}, \dots, r_{s})
		\]
		is stable. 
		
		Now, starting from a stable $(n+2s)$-curve of genus $g-s$ $(D, t_{1}, \dots, t_{n+2s})$ over a scheme $S$, let us construct an $n$-pointed $A_{2}$-curve $(\widetilde{D}, \widetilde{t}_{1}, \dots, \widetilde{t}_{n})$ of genus $g$, over the same $S$, as follows.
		
		\begin{enumerate1}
			
			\item The curve $\widetilde{D}$ is obtained from $D$ by pinching down the images of $t_{n+1}$, \dots,~$t_{n+s}$.
			
			\item For each $i = 1$, \dots,~$i= n$, the section $\widetilde{t}_{i}\colon S \arr$ is the composite of $t_{i}\colon  S \arr D$ with the projection $D \arr \widetilde{D}$.
			
			\item The markings $t_{n+s+1}$, \dots,~$t_{n+2s}$ are forgotten.
			
		\end{enumerate1} 
		
		The curve $(\widetilde{D}, \widetilde{t}_{1}, \dots, \widetilde{t}_{n})$ is not necessarily stable; however, by Proposition~\ref{prop:openness} there exists an open substack $\cU_{s} \subseteq \mbar_{g-s, n+2s}$ whose objects are curves $$(D, t_{1}, \dots, t_{n+2s})$$ whose associate curve $(\widetilde{D}, \widetilde{t}_{1}, \dots, \widetilde{t}_{n})$ is stable. By sending $(D, t_{1}, \dots, t_{n+2s})$ into $(\widetilde{D}, \widetilde{t}_{1}, \dots, \widetilde{t}_{n})$ we get a morphism $\cU_{s} \arr \mt_{g,n}$. For each $\Omega$, the stable $n$-pointed curves with exactly $s$ cuspidal points are precisely those coming from $\cU_{s}$; since $\cU_{s}$ is of finite type we have a smooth surjective morphism $U_{s} \arr \cU_{s}$, in which $U_{s}$ is a scheme of finite type over $k$. The composite $U_{s} \arr \cU_{s} \arr \mt_{g, n}$ is the desired morphism.
	\end{proof}

	\subsection{The description of the cuspidal locus}\label{sub:description cusp} 
	Consider an $A_{2}$-curve $C \arr S$. The fact that formation of $C\cu$ commutes with base changes allows us define $C\cu$ even when $S$ is an Artin stack. In particular, let $\Ctilde_{g} \arr \mt_{g}$ be the universal stable $A_2$-curve of genus~$g$; define $\Ctilde\cu_g$ to be the cuspidal locus inside $\cC_{g}$. Then $\Ctilde\cu_{g}$ is the stack of stable $A_2$-curves of genus~$g$ with a marked cusp. More precisely, one can describe $\Ctilde\cu_{g}$ as the stack of stable $A_2$-curves $C \arr S$ of genus $g$, with a section $S \arr C\cu$.
	
	By Proposition~\ref{prop:cuspidal-projection}, the projection $\Ctilde\cu_{g} \arr \mt_{g}$ is finite and unramified; its image is precisely the locus of curves with at least one cusp. Furthermore $\Ctilde\cu_{g}$ is a smooth stack; hence, it is in fact the normalization of the locus of curves in $\mt_{g}$ with at least one cusp. We aim to give a description of $\Ctilde\cu_{g}$. This is a little subtle, as there are two distinct possibilities for a stable $A_{2}$ curve, revealed by the following.
	
	\begin{proposition}\label{prop:description-punctual}
		Assume that $k$ is algebraically closed. Let $C$ be a stable $A_2$-curve of genus $g$ over $k$, and let $p \in C(k)$ be a cuspidal point. Call $\overline{C}$ the normalization of $C$ at $p$, and $\overline{p}$ the point of $\overline{C}$ lying over $p$. Furthermore, call $D$ the component of $C$ containing $p$ and $\overline{D}$ its inverse image in $\overline{C}$. Then there are two possibilities:
		
		\begin{enumerate1}
			
			\item $(\overline{C}, \overline{p})$ is a $1$-marked stable $A_2$-curve of genus $g-1$, or
			
			\item $\overline{D}$ is smooth of genus~$0$, and meets the rest of $\overline{C}$ transversally in one smooth point.
			
		\end{enumerate1}
		
		In case (1), the morphism $\overline{C} \arr C$ induces an isomorphism of group schemes $\underaut_{k}(\overline{C}, \overline{p}) \simeq \underaut_{k}(C, p)$.
		
		In case (2), call $C_{1}$ the union of the irreducible components of $C$ different from $D$, and $q$ the point of intersection of $D$ with $C_{1}$. Then $\underaut_{k}(D, p, q) = \gm$; there is an isomorphism of group schemes $\underaut_{k}(C_{1}, q) \times \gm \simeq \underaut_{k}(C,p)$, in which $\underaut_{k}(C_{1}, q)$ acts trivially on $D$, and $\gm$ acts trivially on $C_{1}$.
		
	\end{proposition}
	
	\begin{proof}
		Then the degree of the restriction $\omega_{\overline{C}}(\overline{p}) \mid \overline{D}$ equals the degree of the restriction $\omega_{C} \mid D$ minus one. This means that the are two possibilities: either $\omega_{\overline{C}}(\overline{p}) \mid \overline{D}$ has positive degree, which gives case (1), or it has degree $0$, which gives case (2).
		
		The statement about automorphism group schemes is straightforward.\end{proof}

	This shows that there are two possibilities for the pair $(C, p)$. It can be obtained from a stable $1$-marked $A_2$-curve $(\overline{C}, \overline{p})$ by either pinching down $\overline{p}$, or by attaching to $\overline{p}$ a rational curve with a cusp.

	Let us give a stack-theoretic version of this, by showing that $\Ctilde\cu_{g}$ is isomorphic to $\mt_{g-1,1}\times\ag$. The stack $\ag$ has a closed substack $\cB\gm = [0/\gm]$, whose complement is $[(\AA^{1} \setminus \{0\})/\gm] = \spec k$. Thus a geometric point of $\mt_{g-1,1}\times\ag$ is of two types.
	
\begin{enumerate1}

\item It can be in the open substack $\mt_{g-1,1} = \mt_{g-1,1}\times\spec k$; in this case it corresponds to a stable $1$-marked $A_{2}$-curve $(\overline{C}, \overline{p})$ of genus $g-1$. In this case the associated $A_{2}$ curve is of type~(1), and is obtained by pinching $\overline{C}$ at $\overline{p}$.
\begin{figure}[h]
\centering
\begin{overpic}[width=0.5\textwidth]{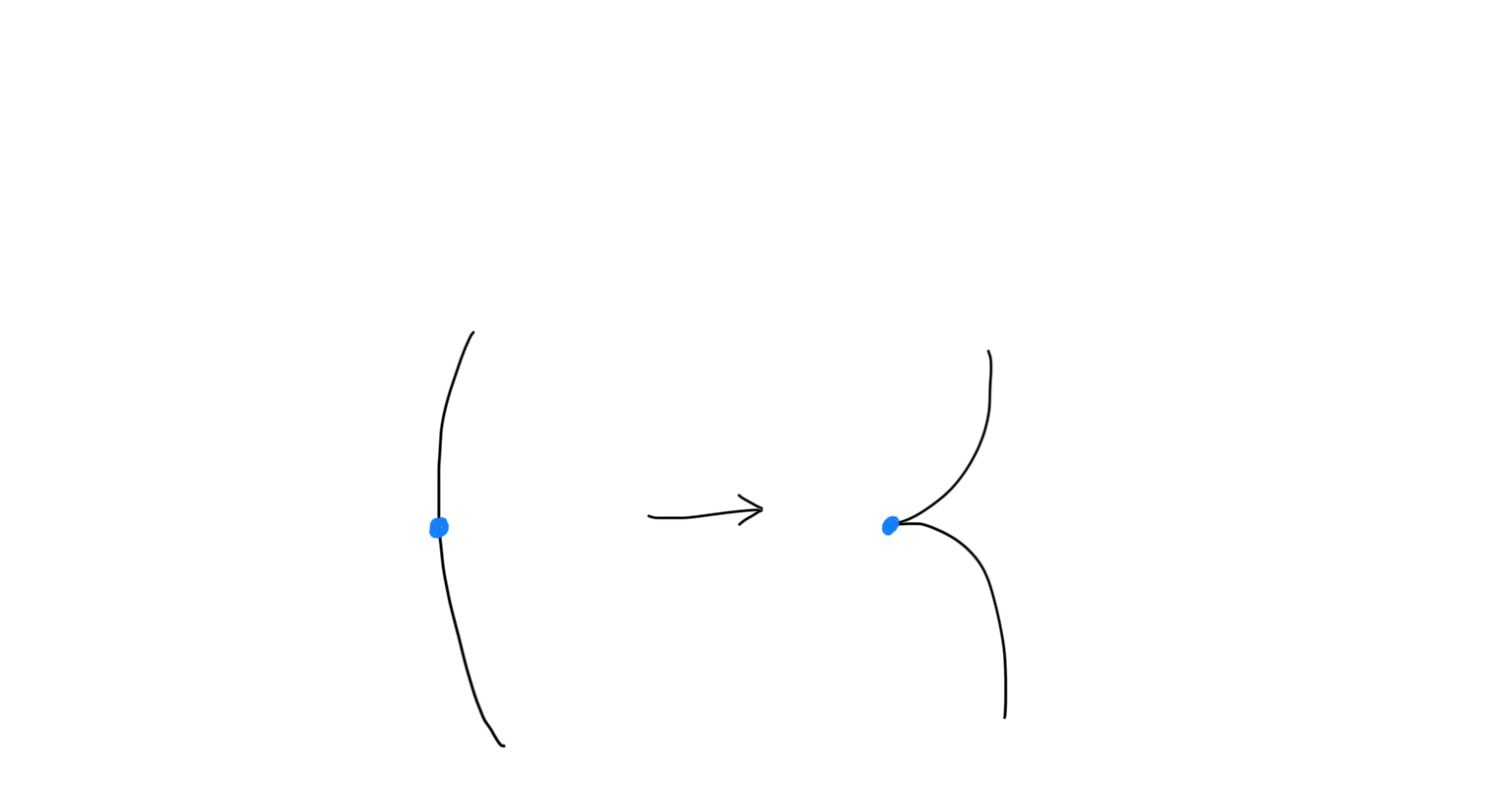} 
\put (24,25) {\large$\displaystyle \overline{C}$}
\put (32,13) {\large$\displaystyle \overline{p}$}
\put (68,25) {\large$\displaystyle C$}
\put (59,13) {\large$\displaystyle p$}
\end{overpic}
\caption{Cuspidal curve of type (1)}
\label{fig:subim1}
\end{figure}

\item Alternatively, it can be in the closed substack $\mt_{g-1,1}\times\cB\gm \subseteq \mt_{g-1,1}\times\ag$; then it corresponds, up to isomorphism, to a stable $1$-marked $A_{2}$-curve $(\overline{C}, \overline{p})$ of genus $g-1$; but the automorphism group of the object of the geometric point of $\mt_{g-1,1}\times\ag$ is $\underaut(\overline{C}, \overline{p}) \times\gm$. In the case we associate with this object the curve of type~(2) obtained by attaching to $\overline{p}$ a rational curve with a cusp.

\begin{figure}[h]
\centering
\begin{overpic}[width=1\textwidth]{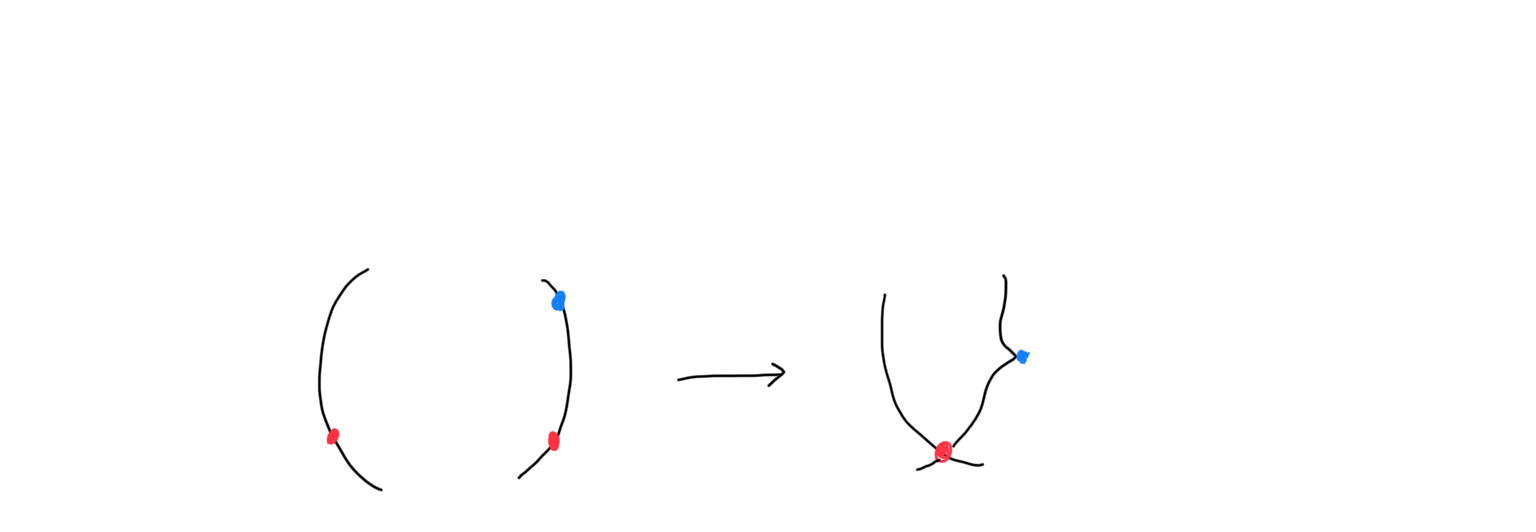}
\put (19,16) {\large$\displaystyle \overline{C}$}
\put (38,16) {\large$\displaystyle E$}
\put (23,6) {\large$\displaystyle \overline{p}$}
\put (68,16) {\large$\displaystyle C$}
\put (61,2) {\large$\displaystyle p$}
\end{overpic}
\caption{Cuspidal curve of type (2)}
\label{fig:subim2}
\centering
\end{figure}

\end{enumerate1}

The way these two set-theoretic descriptions glue together to a morphism $\mt_{g-1, 1} \times \ag \arr \Ctilde\cu_{g}$ is a little subtle. We construct the morphism as follows.
	
	Consider the universal curve $\Ctilde_{g-1,1} \arr \mt_{g-1,1}$, with its tautological section $\sigma\colon \mt_{g-1,1} \arr \cC_{g-1, 1}$. Call $\Sigma$ the image of $\sigma$; then $\Sigma$ is a closed smooth subscheme of $\Ctilde_{g-1, 1} \times \ag$ of codimension~$2$. Denote by $\cD_{g-1,1}$ the blowup of $\Ctilde_{g-1, 1} \times \ag$ along $\Sigma \times \cB\gm$, and call $\Sigma'$ the proper transform of $\Sigma$ in $\cD_{g-1, 1}$.
	
	\begin{proposition}
		The morphism $\cD_{g-1,1} \arr \mt_{g-1,1}\times \ag$ is an $A_{2}$-curve of genus $g-1$. The composite $\Sigma' \subseteq \cD_{g-1,1} \arr \mt_{g-1,1}\times \ag$ is an isomorphism. The corresponding section $\mt_{g-1,1}\times \ag \arr \cD_{g-1,1}$ has image contained in the smooth locus of the morphism $\cD_{g-1,1} \arr \mt_{g-1,1}\times \ag$.
	\end{proposition}
	
	\begin{proof}
		Obviously, $\cD_{g-1,1}$ is proper and representable over $\mt_{g-1,1}\times \ag$. Also, it is smooth over $k$, since both $\mt_{g-1,1}\times \ag$ and $\Sigma\times\gm$ are smooth; since the morphism $\cD_{g-1,1} \arr \mt_{g-1,1}\times \ag$ has equidimensional fibers, it is also flat.
		
		The fact that $\Sigma' \arr \mt_{g-1,1}\times \ag$ is an isomorphism follows immediately from the fact that $\Sigma \times \cB\gm$ is a Cartier divisor in $\Sigma \times \ag$.
		
		Now we need to check that the singularities of the geometric fibers of $\cD_{g-1,1}$ over are at most $A_{2}$, and that the intersection of $\Sigma'$ with the geometric fibers is contained in the smooth locus. This is a particular case of the following.
	\end{proof}
	
	\begin{lemma}\label{lem:fibers-D}
		Let $s\colon \spec\Omega \arr \mt_{g-1,1}\times \ag$ be a geometric point. Call $(C, p) \in \mt_{g-1, 1}$ the marked $A_{2}$-curve corresponding to the composite
		\[
		\spec\Omega \xarr{s} \mt_{g-1,1}\times \ag \xarr{\pr_{1}} \mt_{g-1,1}\,.
		\]
		
		If the image of $s$ in $\ag$ is in $\ag \setminus \cB\gm \simeq \spec k$, then the fiber of $\cD_{g-1,1} \arr \mt_{g-1,1}\times \ag$ over $s$ is $C$, and the inverse image of $\Sigma'$ is $p \in C$.
		
		If the image of $s$ in $\ag$ is in $\cB\gm$, then the fiber is $C$, with a copy of $\PP^{1}$ glued to $p$ by the point $0 \in \PP^{1}$. The inverse image of $\Sigma'$ is the point at infinity.
	\end{lemma}
	
	Thus we can apply the pinching construction above and get an $A_2$-curve $\widehat{\cD}_{g-1,1} \arr \mt_{g-1,1}\times \ag$.
	
	\begin{proposition}
		The morphism $\widehat{\cD}_{g-1,1} \arr \mt_{g-1,1}\times \ag$ is a stable $A_2$-curve.
	\end{proposition}
	
	\begin{proof}
		We need to check that the geometric fibers of $\widehat{\cD}_{g-1,1} \arr \mt_{g-1,1}\times \ag$ are $A_2$-stable; this is clear from Lemma~\ref{lem:fibers-D}.
	\end{proof}
	
	By construction $\widehat{\cD}_{g-1,1} \arr \mt_{g-1,1}\times \ag$ comes from a stable $1$-pointed $A_{2}$-curve $\cD_{g-1,1} \arr \mt_{g-1,1}\times \ag$ by pinching down a section $\mt_{g-1,1}\times \ag \arr \cD_{g-1,1}$ with image $\Sigma' \subseteq\cD_{g-1,1}$; from the last statement of Proposition~\ref{prop:pinching} we get a factorization $\mt_{g-1,1}\times \ag \arr \Ctilde_{g}\cu \arr \mt_{g}$.
	
	\begin{theorem}\label{thm:cuspidal-description}
		The morphism $\Pi\colon \mt_{g-1,1}\times \ag \arr \Ctilde_{g}\cu$ described above is an isomorphism.
	\end{theorem}
	
	We will use the following.
	
	\begin{lemma}\label{lem:isom}
		Let $f\colon \cX \arr \cY$ an \'{e}tale morphism of algebraic stacks over $k$. Assume that the following conditions are satisfied for every algebraically closed extension $k \subseteq\Omega$.
		
		\begin{enumerate1}
			
			\item The morphism $f$ induces a bijection between isomorphism classes in $\cX(\Omega)$ and in $\cY(\Omega)$.
			
			\item If $x \in \cX(\Omega)$, the morphism $\underaut_{\Omega}(x) \arr \underaut_{\Omega}\bigl(f(x)\bigr)$ induced by $f$ is an isomorphism.
			
		\end{enumerate1}
		
		Then $f\colon \cX \arr \cY$ is an isomorphism.
	\end{lemma}
	
	\begin{proof}
		Let $V \arr \cY$ be a smooth surjective morphism, where $V$ is a scheme. Set $U := \cX\times_{\cY} V$; the conditions above are easily seen to imply that the automorphism group schemes of the geometric points of $U$ are trivial, so that $U$ is an algebraic space. The morphism $U \arr V$ is \'{e}tale, and for each $\Omega$ the induced function $U(\Omega) \arr V(\Omega)$ is a bijection. Since $V \arr U$ is \'{e}tale and surjective, it is an epimorphism of \'{e}tale sheaves; hence to show that it is an isomorphism it is enough to prove that the diagonal $V \arr V\times_{U}V$ is an isomorphism. But $V \arr U$ is unramified, so $V$ is an open subspace of $V\times_{U}V$; since it has the same geometric points, the result follows. So $U \arr V$ is an isomorphism, hence $f$ is an isomorphism, as claimed.
	\end{proof}
	
	\begin{proof}[Proof of Theorem~\ref{thm:cuspidal-description}]
	By Lemma~\ref{lem:isom}, Proposition~\ref{prop:description-punctual} and Lemma~\ref{lem:fibers-D} it is enough to prove that $\Pi$ is \'{e}tale. 

Let $\mt'_{g-1, 1}$ be the stack whose sections are $1$-pointed $A_{2}$-curves $C \arr S$ such that, if we denote by $\Sigma \subseteq$ the image of the section, the invertible sheaf $\omega_{C/S}(2\Sigma)$ is relatively ample on $S$. The $A_{2}$-curve $\cD_{g-1,1} \arr \mt_{g-1,1}\times \ag$ satisfies this condition; hence we get a morphism $\Pi'\colon \mt_{g-1,1}\times \ag \arr \mt'_{g-1, 1}$. 

Also, if $C \arr S$ is in $\mt'_{g-1, 1}$, the pinched curve $\widehat{C} \arr S$ is in $\mt_{g}$; hence the pinching construction gives a morphism $\Pi''\colon \mt'_{g-1, 1} \arr \Ctilde_{g}\cu$. The composite
   \[
   \mt_{g-1,1}\times \ag \xarr{\Pi'} \mt'_{g-1,1} \xarr{\Pi''} \Ctilde_{g}\cu
   \]
is $\Pi$, by construction; so it is enough to show that $\Pi'$ and $\Pi''$ are \'{e}tale.

For $\Pi''$ this follows from Lemma~\ref{lem:cuspidal-isomorphism}. For $\Pi'$, the locus of $\mt_{g-1,1}\times \ag$ where $\Pi'$ is \'{e}tale is open, and the only open substack of $\mt_{g-1,1}\times \ag$ containing $\mt_{g-1,1}\times \cB\gm$ is $\mt_{g-1,1}\times \ag$ itself; hence it is enough show that $\Pi'$ is \'{e}tale along $\mt_{g-1,1}\times \cB\gm$.

Denote by $\mt^{o}_{g-1,1}$ the complement $\mt'_{g-1,1} \setminus \mt_{g-1,1}$, with its reduced scheme structure; this is the scheme-theoretic image of $\mt_{g-1,1}\times \cB\gm$ into $\mt'_{g-1,1}$.

The morphism $\mt_{g-1,1}\times \cB\gm \arr \mt^{o}_{g-1,1}$ can be described as follows. We can interpret $\cB\gm$ as the stack $\cM_{0,2}$ of smooth curves $P \arr S$ of genus~$0$ with two disjoint sections $s_{1}$, $s_{2}\colon S \arr P$; the $\gm$-torsor corresponding to $(P \arr S, s_{1}, s_{2})$ is that associated with the normal bundle to $s_{1}$. Given an object $(C \to S, s)$ of $\mt_{g-1,1}$ and an object $(P \to S, s_{1}, s_{2})$ of $\cM_{0,2}$, we obtain an object $C \sqcup P$ of $\mt'_{g-1,1}$ by gluing $C$ with $P$ by identifying $s$ and $s_{1}$, and using $s_{2}$ to give the marking. It is easy to check that $\mt_{g-1,1}\times \cB\gm \arr \mt^{o}_{g-1,1}$ is an isomorphism. 

So, it is enough to show that the scheme theoretic inverse image of $\mt^{o}_{g-1,1} \subseteq\mt_{g-1,1}$  in $\mt_{g-1,1}\times \ag$ is $\mt_{g-1,1}\times \cB\gm$. We can assume that $k$ is algebraically closed; let $c\colon \spec k \arr \mt_{g-1,1}\times \cB\gm$ be a morphism, and let us show that there is a smooth morphism $U \arr \mt_{g-1,1}\times \ag$ with a lifting $\spec k \arr U$ of $c$, such that the pullback of $\mt_{g-1,1}\times \cB\gm$ to $U$ coincides with the pullback of $\mt^{o}_{g-1,1}$. Here is a picture of the curve corresponding to the composite $c'\colon \spec k \arr \mt_{g-1,1}\times \cB\gm \arr \mt^{o}_{g-1,1}$; the blue point is the marked point, while the red point, which we call $q$, is the node where the component containing the marked point, which is isomorphic to $\PP^{1}$, meets the rest of the curve.

\centerline{\includegraphics[scale=.4]{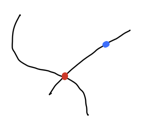}}

Let $(V_{1},q_{1})$ be a versal deformation space of the node $q$ (here $q_{1} \in V_{1}(k)$ is the marked point). Since the universal family over $\mt'_{g-1,1}$ is obviously versal, there exists a smooth morphism $V \arr \mt'_{g-1,1}$ with a lifting of $c$ and a smooth morphism $V \arr V_{1}$ such that the inverse image of $q_{1}\in V_{1}$ in $V$ coincides with the inverse image of $\mt^{o}_{g-1,1}$. Set $U \eqdef (\mt_{g-1,1}\times\ga)\times_{\mt'_{g-1,1}}V$; for the thesis to hold it is enough to show the the inverse image of $q_{1}$ along the morphism $V \arr V_{1}$ is reduced. But this follows immediately from the fact that the curve $\cD_{g-1,1}$ is smooth over $\spec k$.

This completes the proof.

	\end{proof}

	
	\section{The Chow ring of $\Mbar_{1,2}$} \label{sec:chow Mbar12}
	From now on, the base ring $k$ will be a field of characteristic $\neq 2,3$.
	
	In this Section we compute the integral Chow ring of $\Mbar_{1,2}$, the moduli stack of stable genus $1$ curves with two markings (\Cref{thm:chow Mbar12}), and consequently also the rational Chow ring of the coarse moduli space $\overline{M}_{1,2}$ (\Cref{cor:chow Mbar12}).
	
	These results are achieved in the following way: we first compute the integral Chow ring of $\Ctilde_{1,1}$, the universal elliptic stable $A_2$-curve (\Cref{prop:chow Ctilde11}), using the patching lemma (\Cref{lm:Atiyah}). Then we conclude our computations leveraging the localization exact sequence induced by the open embedding $\Mbar_{1,2}\hookrightarrow\Ctilde_{1,1}$. 
	
	In this Section the reader can already see all the main tools that will be used in the upcoming Sections to determine $\ch(\Mbar_{2,1})$
	\subsection{The Chow ring of $\Ctilde_{1,1}$}
	First let us recall one of our key tools, the \emph{patching lemma}.
	\begin{lemma}[\cite{DLV}*{Lemma 3.4}]\label{lm:Atiyah}
		Let $X$ be a smooth variety endowed with the action of a group $G$, and $Y\xhookrightarrow{i} X$ a smooth, closed and $G$-invariant subvariety, with normal bundle $\cN$. Suppose that $c_{\rm top}^G(\cN)$ is not a zero-divisor in $\ch_G(Y)$. Then the following diagram of rings is cartesian:
		\[\xymatrix{
			\ch_G(X) \ar[r]^{i^*} \ar[d]^{j^*} & \ch_G(Y) \ar[d]^{q} \\
			\ch_G(X\smallsetminus Y) \ar[r]^{p} & \ch_G(Y)/(c_{\rm top}^G(\cN))
		}\]
		where the bottom horizontal arrow $p$ sends the class of a variety $V$ to the equivalence class of $i^*\xi$, where $\xi$ is any element in the set $(j^*)^{-1}([V])$.
	\end{lemma}
	
	We will call $V_i$ the rank one representation of $\gm$ on which the latter acts with weight $i$. The notation $V_{i_1,\dots,i_r}$ will stand for the rank $r$ representation $V_{i_1}\oplus\cdots\oplus V_{i_r}$.
	
	Let $\bfp:\Mtilde_{1,1}\to \Ctilde_{1,1}$ be the universal section.
	\begin{lemma}\label{lm:Ctilde11 minus sigma affine bundle}
		The stack $\Ctilde_{1,1}\smallsetminus\im{\bfp}$ is an affine bundle over $[V_{-2,-3}/\gm]$ and
		\[\ch(\Ctilde_{1,1}\smallsetminus\im{\bfp})\simeq \ZZ[\lambda_1],\]
		where $\lambda_1$ is the pullback of the first Chern class of the Hodge line bundle along $\Ctilde_{1,1}\arr\Mtilde_{1,1}$.
	\end{lemma}
	\begin{proof}
		Recall that $\Mtilde_{1,1}$ is isomorphic to the quotient stack $[V_{-4,-6}/\gm]$: indeed, the $\gm$-torsor associated with the Hodge line bundle is isomorphic to $V_{-4,-6}$ (see \cite{EG}*{5.4}). Consider the universal affine Weierstrass curve $W$ inside $V_{-4,-6}\times V_{-2,-3}$ defined by the equation
		\[ y^2=x^3+ax+b, \]
		where $(a,b)\in V_{-4,-6}$ and $(x,y)\in V_{-2,-3}$. Observe that this subscheme is $\gm$-invariant.
		The quotient stack $[W/\gm]$ is isomorphic to $\Ctilde_{1,1}\smallsetminus\im{\bfp}$, and it is immediate to check that $W\arr V_{-2,-3}$ is an equivariant affine bundle. Therefore
		\[ \ch(\Ctilde_{1,1}\smallsetminus\im{\bfp}) \simeq \ch([V_{-2,-3}/\gm]) \simeq \ch(\cB\gm) \simeq \ZZ[T] \]
		where the isomorphisms are all given by pullback homomorphism. To conclude, observe that we have a commutative diagram of $\gm$-quotient stacks
		\[ \xymatrix{
			[W/\gm] \ar[r] \ar[d] & [V_{-2,-3}/\gm] \ar[d] \\
			[V_{-4,-6}/\gm] \ar[r] & [\spec{k}/\gm].
		}  \]
		If we take the induced pullbacks, we obtain a commutative diagram of Chow rings: in particular, the pullback of $T$ along the right vertical and the top horizontal maps is equal to the pullback of $T$ along the other two maps. As the pullback of $T$ along the bottom horizontal arrow is $-\lambda_1$, we get the desired conclusion.
	\end{proof}
	\begin{proposition}\label{prop:chow Ctilde11}
		We have
		\[\ch(\Ctilde_{1,1})\simeq \ZZ[\lambda_1,\mu_1]/(\mu_1(\lambda_1+\mu_1)),\]
		where $\lambda_1$ is the first Chern class of the Hodge line bundle and $\mu_1:=\bfp_*[\Mtilde_{1,1}]$ is the fundamental class of the universal section.
	\end{proposition}
	\begin{proof}
		By definition the normal bundle of the universal section is equal to $-\psi_1=-\lambda_1$. We can apply \Cref{lm:Atiyah}, which tells us that we have a cartesian diagram of rings
		\[
		\begin{tikzcd}
			\ch(\Ctilde_{1,1}) \ar[r, "\bfp^*"] \ar[d, "j^*"] & \ch(\Mtilde_{1,1})\simeq\ZZ[\lambda_1] \ar[d] \\
			\ch(\Ctilde_{1,1}\smallsetminus\im\bfp)\simeq\ZZ[\lambda_1] \ar[r] & \ZZ.
		\end{tikzcd}
		\]
		As the restriction of $\lambda_1$, regarded as an element in $\ch(\Ctilde_{1,1})$, to $\ch(\Mtilde_{1,1})$ is equal to $\lambda_1$, we deduce that $\ch(\Ctilde_{1,1})$ is generated by $\lambda_1$ and $\bfp_*[\Mtilde_{1,1}]=:\mu_1$.
		
		The cartesianity of the diagram above implies that the ideal of relations is formed by those polynomials $p(\lambda_1,\mu_1)$ which belong to the kernel of both $j^*$ and $\bfp^*$.
		
		If $j^*p(\lambda_1,\mu_1)=0$ then it must be of the form $\mu_1q(\lambda_1,\mu_1)$. We have
		\[ \bfp^*(\mu_1q(\lambda_1,\mu_1)) = -\lambda_1(q(\lambda_1,-\lambda_1)). \]
		If this is zero then $q(\lambda_1,\mu_1)$ must be divisible by $\lambda_1+\mu_1$, thus the ideal of relations is generated by $\mu_1(\lambda_1+\mu_1)$.
	\end{proof}
	
	\subsection{The Chow ring of $\Mbar_{1,2}$}
	In several points in what follows we will use the following notion (see \cite[Definition 4.1]{Per})
	
	\begin{definition}
		If $f \colon \cX \arr \cY$ is a morphism of algebraic stacks, we say that $f$ is a \emph{Chow envelope} if it is proper and representable, and for every extension $K/k$ the induced functor $f_{K}\colon \cX(K) \arr \cY(K)$ is essentially surjective.
	\end{definition}
	
	\begin{proposition}\hfill
		\begin{enumerate1}
			
			\item Being a Chow envelope is a property that is stable under composition and under base change.
			
			\item If $f\colon \cX \arr \cY$ is a Chow envelope, and $\cY$ is a quotient stack of finite type over $k$, then $f_{*}\colon \operatorname{CH}^{*-c}(\cX) \arr \ch(\cY)$ is surjective where $c$ is the relative dimension of $f$ (namely $\dim \cY - \dim \cX$).
			
		\end{enumerate1}
	\end{proposition}
	
	\begin{proof}
	
		Part (1) is straightforward.
		
\modangelo For Part (2), write $\cY = [Y/G]$, where $Y$ is an algebraic space of finite type, and $G \arr \spec k$ an affine algebraic group acting on $Y$. Recall the basic construction of \cite{EG}. Let $N$ be a positive integer, $G \arr \GL(V)$ be a finite dimensional representation with an open subscheme $U \subseteq V$ with $\codim_{V}(V \setminus U) > N$; then the fppf quotient $(Y\times U)/G$ is an algebraic space, and by definition we have $\ch[i](\cY) = \ch[i]\bigl((Y\times U)/G\bigr)$ for $i \leq N$. 

Since the morphism $\cX \arr \cY$ is proper and representable, we have a proper $G$-equivariant proper morphism of algebraic spaces $X \arr Y$ with $[X/G] = \cX$, and $\ch[i](\cX) = \ch[i]\bigl((X\times U)/G\bigr)$ for $i \leq N$, and we have a cartesian diagram
   \[
   \begin{tikzcd}
   (X\times U)/G \rar\dar & \cX \dar\\
   (Y\times U)/G \rar & \cY\,.
   \end{tikzcd}
   \]
This means that we can assume that $\cX$ and $\cY$ are algebraic spaces.

We refer to \cite[\S6.1]{EG} for basic facts about intersection theory on algebraic spaces. The Chow group of $\cY$ is generated by classes of integral closed subspaces $V \subseteq \cY$; we need to show that, given such a subspace $V \subseteq \cY$, there exists an integral closed subspace $W \subseteq \cX$ mapping birationally onto $V$. The algebraic space $V$ has a dense open subscheme $V' \subseteq V$; set $k(V) \eqdef k(V')$. By hypothesis, the composite  $\spec k(V) \arr V' \subseteq V \subseteq \cY$ lifts to $\spec k(V) \arr \cX$. But $\cX$ is of finite type, so this means that there exists a nonempty open subscheme $V_{1} \subseteq V'$ that lifts to an embedding $V_{1} \subseteq \cX$. Then we take $W$ to be the closure of $V_{1}$ in $\cX$.
	\end{proof}
	
	Here is the main result of the Section. We identify $\Mbar_{1,2}$ with the universal curve over $\Mbar_{1,1}$.
	\begin{theorem}\label{thm:chow Mbar12}
		Suppose that the ground field has characteristic $\neq 2,3$. Then
		\[ \ch(\Mbar_{1,2}) \simeq \ZZ[\lambda_1,\mu_1]/(\mu_1(\lambda_1+\mu_1),24\lambda_1^2) \]
		where $\lambda_1$ is the first Chern class of the Hodge line bundle and $\mu_1:=\bfp_*[\Mbar_{1,1}]$ is the fundamental class of the universal section.
	\end{theorem}

	\begin{proof}
	    Let $\cM_{0,2}$ be the stack of genus zero smooth curves with two marked points. Up to scalar multiplication, there is only one isomorphism class of $2$-pointed smooth genus zero curve, i.e. a projective line $\PP^1$ with $0$ and $\infty$ as marked points. This implies that $\cM_{0,2}\simeq\cB\gm$ and that the universal marked curve is isomorphic to $[\PP^1/\gm]$, where $\gm$ acts on $\PP^1$ by scalar multiplication.
		
		Consider the morphism of stacks
		\[ \cM_{0,2} \simeq \cB\gm \arr \Mtilde_{1,1} \]
		that sends a $2$-marked genus zero smooth curve $(C\arr S,p,\sigma)$ to the cuspidal elliptic curve $(\widehat{C} \arr S,p)$, where $\widehat{C}$ is obtained from $C$ by pinching $\sigma$. 
		
		In particular, let $\widehat{\PP^1}$ be the cuspidal curve obtained by pinching $\PP^1$ at $0$. The scalar action of $\gm$ on $\PP^1$ descends to an action on $\widehat{\PP^1}$ and the normalization map $\PP^1\to\widehat{\PP^1}$ is equivariant with respect to this action, hence it induces a morphism of quotient stacks $c''':[\PP^1/\gm]\to [\widehat{\PP^1}/\gm]$.
		
		Then we have a commutative diagram
		\[ 
		\begin{tikzcd}
			\left[ \PP^1/\gm \right] \ar[r, "c'''"] \ar[dr, "\rho"] \ar[rr, bend left, "c'"]  & {\left[\widehat{\PP^1}/\gm\right]} \ar[r, "c''"] \ar[d, "\pi'"] & \Ctilde_{1,1} \ar[d, "\pi"] \\
			& \cM_{0,2}\simeq\cB\gm \ar[r, "c"] & \Mtilde_{1,1}.
		\end{tikzcd}
		\]
		where the square in the diagram is cartesian.
		Observe that $c'$ is a Chow envelope for the locus of cuspidal curves in $\Ctilde_{1,1}$, hence if we excise the image of
		\begin{equation}\label{eq:c} 
		c'_*\colon\operatorname{CH}^{*-2}([\PP^1/\gm]) \arr \ch(\Ctilde_{1,1})\simeq \ZZ[\lambda_1,\mu_1]
		 \end{equation}
		we obtain the Chow ring of $\Mbar_{1,2}$.

		From the projective bundle formula, we see that $\ch([\PP^1/\gm])$ is generated as a $\ch(\cB\gm)$-module by $1$ and by the hyperplane section $h$, hence the image of $c'_*$ is generated as an ideal by $c'_*(\rho^*t^i\cdot h)$, where $t$ is the generator of $\ch(\cB\gm)\simeq \ZZ[t]$ and it coincides with $c^*\lambda_1$.
		
		We have that $\rho^*(t^i)=\rho^*c^*(\lambda_1^i) = (c')^*(\pi^*\lambda_1^i)$. This implies that $c'_*\rho^*(t^ih)=c'_*(h\cdot(c')^*(\pi^*\lambda_1^i))=\pi^*\lambda_1^i\cdot c'_*(h)$, hence the image of $c'_*$ is generated as an ideal by $c'_*(1)$ and $c'_*(h)$. 
		
		The computation of $c'_*(1)$ is straightforward, as we have 
		\[c'_*(1)=c''_*(\pi')^*(1)=\pi^*c_*(1)=24\lambda_1^2.\]
		Above we are using the fact that $c'''$ is birational, hence $c'''_*(1)=1$, and that $c:\cB\gm\to\Mtilde_{1,1}\simeq [V_{-4,-6}/\gm]$ is the zero section of the vector bundle $[V_{-4,-6}/\gm]\to\cB\gm$, therefore its class coincides with the top Chern class of the $\gm$-representation $V_{-4,-6}$, which is equal to $(-4\lambda_1)(-6\lambda_1)=24\lambda_1^2$.
		
		Let $\bfp':\cB\gm\to [\PP^1/\gm]$ be the universal section given by the first marking, i.e. the $\gm$-quotient of the $\gm$-equivariant map $\spec{k}\to\PP^1$ corresponding to the point at infinity. Then $c'\circ \bfp ' = \bfp \circ c$ and $\bfp'_*(1)=h$, which readily implies $c'_*(h)=\bfp_*(24\lambda_1^2)=24\lambda_1^2\mu_1$. This concludes the proof.
	\end{proof}
	The rational Chow ring of the coarse moduli space $\overline{M}_{1,2}$ is isomorphic to the rational Chow ring of $\Mbar_{1,2}$ (see \cite{VisInt}*{Proposition 6.1}). Thus, from the Theorem above we also get the following.
	\begin{corollary}\label{cor:chow Mbar12}
		Let $\overline{M}_{1,2}$ be the coarse moduli space of $\Mbar_{1,2}$. Then over fields of characteristic $\neq 2,3$ we have
		\[ \ch(\overline{M}_{1,2})_{\mathbb{Q}} \simeq \mathbb{Q}[\lambda_1,\mu_1]/(\mu_1(\lambda_1+\mu_1),\lambda_1^2). \]
	\end{corollary}
	
	\section{The Chow ring of $\Ctilde_2$}\label{sec:chow cusp}
	In this Section we determine the integral Chow ring of $\Ctilde_2$, the universal stable $A_2$-curve of genus two (\Cref{prop:chow Ctilde2}). Our strategy resembles the one previously used for computing $\ch(\Ctilde_{1,1})$.
	
	Notice that a stable $A_{2}$-curve of genus $2$ cannot have more than one separating node; furthermore, a small deformation of a cusp cannot be a separating node. Hence the locus of separating nodes is a closed subset $\ThTilde_{2} \subseteq\Ctilde_{2}$, which is a connected component of the singular locus $\Ctilde_{2}\sing \subseteq \Ctilde_{2}$ of the map $\Ctilde_2\to\Mtilde_2$, with the usual scheme structure given by the first Fitting ideal of $\Omega_{\Ctilde_{2}/\Mtilde_{2}}$. With this structure $\ThTilde_{2}$ is a closed smooth connected substack of $\Ctilde_{2}$.
	
	The composite $\ThTilde_{2} \subseteq \Ctilde_{2} \arr \Mtilde_{2}$ is proper, representable, unramified, and injective on geometric points, hence it is a closed embedding. We will denote by $\Dtilde_{1} \subseteq \Mtilde_{2}$ its image; it is the closure in $\Mtilde_{2}$ of the divisor $\Delta_{1} \subseteq\Mbar_{2}$.
	
	We will denote by $\ThTilde_{1} \subseteq \Ctilde_{2}$ the inverse image of $\Dtilde_{1}$; this is an integral divisor, which is smooth outside of $\ThTilde_{2} \subseteq \ThTilde_{1}$.

	We will use the following cycle classes.
	
	\begin{enumerate1}
		
		\item $\lambda_{1} \in \ch[1](\Mtilde_{2})$ and $\lambda_{2} \in \ch[2](\Mtilde_{2})$ are, as usual, the Chern classes of the Hodge bundle $\htil_{2}$; we will use the same notation for their pullbacks to $\ch(\Ctilde_{2})$.
		
		\item $\psi_{1} \in \ch[1](\Ctilde_{2})$ is the first Chern class of the dualizing sheaf $\omega_{\Ctilde_{2}/\Mtilde_{2}}$.
		
		\item $\theta_{1} \in \ch[1](\Ctilde_{2})$ and $\theta_{2} \in \ch[2](\Ctilde_{2})$ are the classes of $\ThTilde_{1}$ and $\ThTilde_{2}$ respectively.
		
	\end{enumerate1}
	
	If $\cX \subseteq \Ctilde_{2}$ is any smooth substack, we will use the same symbols for the pullbacks of these classes to $\ch(\cX)$.
	
	Then we first compute $\ch(\Ctilde_2\smallsetminus\ThTilde_1)$ and $\ch(\ThTilde_1\smallsetminus\ThTilde_2)$: we put together these descriptions using the patching lemma (\Cref{lm:Atiyah}) to get the Chow ring of $\Ctilde_2\smallsetminus\ThTilde_2$. 
	
	We repeat this process once more: we compute $\ch(\ThTilde_2)$ and we apply the patching lemma to finally obtain $\ch(\Ctilde_2)$.
	\subsection{The Chow ring of $\Ctilde_2\smallsetminus\ThTilde_1$}\label{c2-d1}
	In this Subsection we compute the integral Chow ring of $\Ctilde_2\smallsetminus\ThTilde_1$, using its description as a quotient stack and basic equivariant intersection theory.
	
	Let ${\rm B}_2$ be the Borel subgroup of lower triangular matrices inside $\GL_2$. We denote $\AA(6)$ the $\rm B_2$-representation on degree 6 binary forms, where the action is defined as
	\[A\cdot h(x,z)=\det(A)^2 h(A^{-1}(x,z)). \]
	Let $\widetilde{\AA}(6)$ be the vector space formed by pairs $(h(x,z),s)$ such that $h(0,1)=s^2$. Then $\widetilde{\mathbb{A}}(6)$ can be regarded as a $\rm B_2$-representation, where the action is given by
	\[A\cdot (h(x,z),s):=(\det(A)^2 h(A^{-1}(x,z), \det(A)a_{22}^{-3}\cdot s). \]
	Observe that the natural map $\pi_6\colon\widetilde{\AA}(6)\to\AA(6)$ is a $\rm B_2$-equivariant ramified double cover.

	Let $D_4$ be the closed subscheme of $\AA(6)$ parametrizing homogeneous binary forms of degree $6$ with a root of multiplicity greater than $3$ in some field extension, and call $\widetilde{D}_4$ the preimage of $D_4$ in $\widetilde{\AA}(6)$. Let us set $U:=\widetilde{\AA}(6) \setminus \widetilde{D}_4$; this is a $\rB_{2}$ invariant open subscheme of $\widetilde{\AA}(6)$.
	
	Consider the two characters $\rB_{2} \arr \gm$ defined by $(a_{ij}) \mapsto a_{11}$ and $(a_{ij}) \mapsto a_{22}$; let us call $\xi$ and $\eta \in \ch[1](\cB\rB_{2})$ respectively the first Chern classes of these two characters, and by the same symbols their pullbacks to $[U/\rB_{2}]$.
	
	Furthermore, if $V$ is the tautological rank two representation of $\rB_{2} \subseteq \GL_{2}$, and $c_{1} \in \ch[1](\cB\rB_{2})$ and $c_{2} \in \ch[2](\cB\rB_{2})$ its Chern classes, then we have $\xi+\eta = c_{1}$ and $\xi\eta = c_{2}$.
	
	\begin{proposition}\label{prop:open-strata}
		We have an equivalence of stacks $\Ctilde_2\smallsetminus\ThTilde_1 \simeq [U/{\rm B}_2]$. Furthermore, under this equivalence the class $\eta \in \ch[1][U/{\rm B}_2]$ corresponds to $\psi_{1} \in \ch[1](\Ctilde_{2})$, while $c_{i} \in \ch[i][U/{\rm B}_2]$ corresponds to $\lambda_{i} \in \ch(\Ctilde_{2})$.
	\end{proposition}
	
	The proof of the Proposition above is essentially contained in \cite{Per} by the second author: there only smooth curves are considered, but the arguments are virtually identical.
	
	As explained in \cite{Per}*{Remark 3.1}, the ${\rm B}_2$-equivariant Chow ring of a scheme $X$ is isomorphic to its $T_2$-equivariant one, where $T_2 \subset {\rm B}_2$ is the maximal torus. Therefore, what we need is an explicit presentation of $\ch_{T_2}(U)$.
	\begin{remark}
		It is easy to check that
		$$ \Mtilde_2\setminus \Dtilde_1 \simeq [\AA(6)\setminus D_4/\GL_2]. $$
		The morphism $U = \widetilde{\AA}(6) \setminus \widetilde{D}_4 \rightarrow \AA(6)\setminus D_4$ is $\rB_{2}$-equivariant and the induced map 
		$$ \pi\colon\Ctilde_2 \setminus \ThTilde_1 \simeq [\widetilde{\AA}(6) \setminus \widetilde{D}_4/{\rm B}_2] \longrightarrow [\AA(6) \setminus D_4/\GL_2] \simeq \Mtilde_2\setminus \Dtilde_1$$ 
		is in fact the restriction of the morphism $\Ctilde_2 \rightarrow \Mtilde_2$.
	\end{remark}
	
	The localization sequence applied to the closed subscheme $\widetilde{D}_4$ shows that all we need to do is to compute the ideal $\widetilde{I}$ given by the image of 
	\[{\rm{CH}}^{\ast -3}_{T_2}(\widetilde{D}_4) \longrightarrow \ch_{T_2}(\widetilde{\AA}(6)), \]
	the pushforward along the equivariant closed embedding $\widetilde{D}_4\hookrightarrow \widetilde{\AA}(6)$. 
	
	If we denote by $I$ the ideal generated by the image of the pushforward along the closed embedding $D_4 \hookrightarrow \AA(6)$, clearly we have that $\pi_6^*(I)\subset \widetilde{I}$, where $\pi_6\colon\widetilde{\AA}(6)\to\AA(6)$ is the aforementioned ramified double cover. 
	
	We will prove that the ideal generated by $\pi_6^*(I)$ is in fact equal to $\widetilde{I}$. This works essentially in the same way as in \cite{Per}*{Theorem 5.7}, where instead of $\widetilde{\AA}(6)\smallsetminus \widetilde{D}_4$, there the author considers the complement of the preimage of the discriminant locus.

	We can define a $\gm$-torsor 
	$$ \widetilde{\AA}(6)\setminus 0 \longrightarrow \PP(2^6,1) $$
	where the $\gm$-action on $\widetilde{\AA}(6)$ is described as follows:
	$$ \lambda\cdot (h,s)=(\lambda^2 h,\lambda s)$$ 
	for every point $(h,s)$ in $\widetilde{\AA}(6)$ and for every $\lambda \in \gm$. (Here, as in what follows, we denote by $\PP(2^6,1)$ the weighted projective stack, that is, the stack quotient of $\AA^{7}\smallsetminus\{0\}$ by the action of $\gm$ with weights $(2^{6},1)$). The $T_2$-action descends to an action on $\PP(2^6,1)$. We consider the following cartesian diagram
	$$ 
	\begin{tikzcd}
		\widetilde{\Delta}_4 \arrow[rr, hook] \arrow[d, "\phi_6\vert_{\widetilde{\Delta}_4}"] &  & {\PP(2^6,1)} \arrow[d, "\phi_6"] \\
		\Delta_4 \arrow[rr, hook]                                                          &  & \PP^6                        
	\end{tikzcd}
	$$
	which is the projectivization of
	$$ 
	\begin{tikzcd}
		\widetilde{D}_4 \arrow[rr, hook] \arrow[d, "\pi_6\vert_{\widetilde{D}_4}"] &  & \widetilde{\AA}(6) \arrow[d, "\pi_6"] \\
		D_4 \arrow[rr, hook]                                                          &  & \AA(6)                             
	\end{tikzcd}
	$$
	The action of $\gm$ on the top row is the one described above, while the action on bottom row is the usual one. 
	
	Recall that if $X \rightarrow Y$ is a $\gm$-torsor with $Y$ a quotient stack and $X$ the complement of the zero section in a line bundle $\cL$ on $Y$, then
	$$ \ch(X) \simeq \ch(Y)/(c_1(\cL)).$$ 
	Let $I'$ be the image of the pushforward along the closed embedding $\Delta_4\hookrightarrow \PP^6$ and let $\widetilde{I}'$ be the ideal generated by the image of the pushforward along the closed embedding $\widetilde{\Delta_4}\hookrightarrow \PP(2^6,1)$.
	
	Then from the formula for the Chow ring of $\gm$-torsors we get that if $\widetilde{I}'$ is equal to the ideal generated by $\phi_6^*(I')$ then the same statement holds in the non-projective case, i.e. $\widetilde{I}$ is equal to the ideal generated by $\pi_6^*(I)$. 
	
	A Chow envelope of $\Delta_4$ is given by
	$$\rho: \PP^1 \times \PP^2 \longrightarrow \PP^6$$
	where $\rho(f,g)=f^4g$ (as usual $\PP^n$ must be regarded as the projective space of binary forms of degree $n$). The morphism $\rho$ is a Chow envelope because of the following fact: if $K$ is a field and $f \in \Delta_4(K)$, then $f$ has a root of multiplicity $4$ and the root must defined over $K$ because $f$ has degree $6$.
	
	To find a Chow envelope of $\widetilde{\Delta}_4\subset \PP(2^6,1)$, one could just consider the cartesian diagram
	$$
	\begin{tikzcd}
		\cP \arrow[rr, "a"] \arrow[d, "q"]    &  & {\PP(2^6,1)} \arrow[d, "\phi_6"] \\
		\PP^1 \times \PP^2 \arrow[rr, "\rho"] &  & \PP^6.                        
	\end{tikzcd}
	$$
	The morphism $a$ is in fact a Chow envelope. Nevertheless, we do not know how to describe a Chow envelope for $\cP$. To construct a Chow envelope of $\widetilde{\Delta}_4 \subset \PP(2^6,1)$, we define the morphism 
	$$ a: \PP^1 \times \PP(2^2,1) \longrightarrow \PP(2^6,1)$$
	as $a([f],[g,s])=[(f^4g,f(0,1)^2s)]$ for every $[f]\in \PP^1$ and $[g,s] \in \PP(2^2,1).$
	A straightforward computation gives us the following diagram:
	$$
	\begin{tikzcd}
		{\PP^1 \times \PP(2^2,1)} \arrow[rd, "\alpha"] \arrow[rdd, "\id \times \phi_2"'] \arrow[rrrd, "a"] &                                       &  &                               \\
		& \cP \arrow[rr, "g"] \arrow[d, "q"]    &  & {\PP(2^6,1)} \arrow[d, "\phi_6"] \\
		& \PP^1 \times \PP^2 \arrow[rr, "\rho"] &  & \PP^6                        
	\end{tikzcd}
	$$
	where the square is cartesian (here $\phi_n: \PP(2^n,1)\rightarrow \PP^n$ sends $(a_0:\cdots:a_n)$ to $(a_0:\cdots:a_n^2)$). One would hope that $\alpha$ is a Chow envelope, but this is not the case. In fact, we have to define the morphism 
	$b: \PP(2^3) \longrightarrow \PP(2^6,1)$
	with the formula $b([g])=[(x_0^4g,0)]$ for every $[g]\in \PP(2^3)$. This induces another commutative diagram 
	$$
	\begin{tikzcd}
		\PP(2^3) \arrow[rrrrd, "b"] \arrow[rdd, "\varphi_3"'] \arrow[rd, "\beta"] &                                    &                                       &  &                               \\
		& \cR \arrow[d, "p"] \arrow[r, hook] & \cP \arrow[rr, "g"] \arrow[d, "q"]    &  & {\PP(2^6,1)} \arrow[d, "\phi_6"] \\
		& \PP^2 \arrow[r, hook]              & \PP^1 \times \PP^2 \arrow[rr, "\rho"] &  & \PP^6                        
	\end{tikzcd}
	$$ 
	where the morphism $\varphi_3\colon\PP(2^3)\rightarrow \PP^2$ is the natural $\mu_2$-gerbe and the inclusion $\PP^2 \hookrightarrow \PP^1 \times \PP^2$ is induced by the rational point $[0:1] \in \PP^1$.
	
	\begin{lemma}
		In the setting above, we get that $a$ and $b$ are proper representable morphisms of algebraic stacks and the morphism $\alpha \coprod \beta:(\PP^1 \times \PP(2^2,1)) \coprod \PP(2^3) \longrightarrow \cP$ is a Chow envelope.
	\end{lemma}
	
	\begin{proof}
		Properness is clear from the properness of the stacks involved. A straightforward computation shows that $a$ is not only representable (being faithful on every point) but is in fact fully faithful restricted to $U_0 \times \PP(2^2,1)$ where $U_0$ is the affine open subset of $\PP^1$ where $x_0\neq 0$. In the same way one can prove that $b$ is fully faithful. 
		
		Let us fix a field extension $K/k$: we want to prove that $(\alpha \coprod \beta) (K)$ is essentially surjective. Let $([f],[g],[h,s]) \in \cP(K)$, i.e. $h=f^4g$ for some representatives and suppose $f(0,1)\neq 0$. Thus $f(0,1)^4g(0,1)=s^2$ and therefore $\alpha(K)([f],[g,s/f(0,1)^2])=([f],[g],[h,s])$. Suppose instead $f(0,1)=0$, i.e. $h=x_0^4g$ and $s=0$. Then clearly $b([g])=[h,0]$ as desired. 
	\end{proof}
	
	\begin{remark}
		The previous proof shows us a bit more: it gives us that $\alpha(K)$ is an equivalence on the open $\cP \setminus \cR$ while $\beta(K)$ is actually an equivalence between $\PP(2^3)(K)$ and $\cR(K)$ for every $K/k$ field extension. This gives a description of $\cP_{\rm red}$ as the fibered product of the two stacks $\PP^1 \times \PP(2^2,1)$ and $\PP(2^3)$ over $\PP(2^2,1)$.
	\end{remark}
	
	The previous remark implies $\alpha_*(1)=1$ and $\beta_*(1)=1$ at the level of Chow groups.
	\begin{proposition}\label{prop:pullback relations}
		The ideal generated by the image of the group homomorphism $a_*\oplus b_*$ inside $\ch_{T_2}(\PP(2^6,1))$ is equal the ideal generated by $\phi_6^*(\im {\rho})$.
	\end{proposition}
	\begin{proof}
		First of all, using the projection formula and the explicit description of the Chow rings involved, one can prove that the ideal generated by $a_*\oplus b_*$ is in fact generated by the three cycles: $a_*(1),a_*(c_1^{T_2}(\cO_{\PP^1}(1)\boxtimes \cO_{\PP(2^2,1)})),b_*(1)$. For a more detailed discussion about this see \cite{Per}*{Section 5}.
		Consider now the following diagram:
		$$
		\begin{tikzcd}
			{\PP^1 \times \PP(2^2,1)} \arrow[rd, "\alpha"] \arrow[rdd, "\id \times \phi_2"'] \arrow[rrrd, "a"] &                                       &  &                               \\
			& \cP \arrow[rr, "g"] \arrow[d, "q"]    &  & {\PP(2^6,1)} \arrow[d, "\phi_6"] \\
			& \PP^1 \times \PP^2 \arrow[rr, "\rho"] &  & \PP^6;                        
		\end{tikzcd}
		$$
		therefore $a_*(1)=g_*\alpha_*(1)=g_*(1)=g_*q^*(1)=\phi_6^*\rho_*(1)$ because the square diagram is cartesian by construction and $\alpha_*(1)=1$. In the same way we get the result from $b_*(1)$. As far as the last generator is concerned, we get 
		$$ a_*(c_1^{T_2}(\cO_{\PP^1}(1)\boxtimes \cO_{\PP(2^2,1)}))= g_*\alpha_*(\id \times \phi_2)^*(c_1^{T_2}(\cO_{\PP^1}(1) \boxtimes \cO_{\PP^2}))= g_*\alpha_*c_1^{T_2}(\alpha^*\cL)$$
		where $\cL$ is the line bundle $q^*(\cO_{\PP^1}(1) \boxtimes \cO_{\PP^2}))$. Using the projection formula we get $\alpha_*c_1^{T_2}(\alpha^*\cL)=c_1^{T_2}(\cL)$ and therefore 
		$$  a_*(c_1^{T_2}(\cO_{\PP^1}(1)\boxtimes \cO_{\PP(2^2,1)}))=  g_*q^*(c_1^{T_2}(\cO_{\PP^1}(1) \boxtimes \cO_{\PP^2}))=\phi_6^*\rho_*(c_1^{T_2}(\cO_{\PP^1}(1)\boxtimes \cO_{\PP(2^2,1)}))$$
		and this concludes the proof.
	\end{proof}
	
	As a Corollary, we finally get the description of $\ch(\Ctilde_2 \setminus \ThTilde_1)$.
	\begin{proposition}\label{prop:chow C minus Theta1}
		We have
		\[ \ch(\Ctilde_2\smallsetminus\ThTilde_1)\simeq \ZZ[\psi_1,\lambda_1,\lambda_2]/(\lambda_2-\psi_1(\lambda_1-\psi_1),\lambda_1(24\lambda_1^2-48\lambda_2),20\lambda_1^2\lambda_2). \]
	\end{proposition}
	\begin{proof}
	    The relation $\lambda_2=\psi_1(\lambda_1-\psi_1)$ follows from \Cref{prop:open-strata} and the computations in the paragraph above it.
		From \Cref{prop:pullback relations} we deduce that the image of
		\[ {\rm{CH}}^{*-3}_{\rm B_2}(\widetilde{D}_4) \longrightarrow \ch_{\rm B_2}(\widetilde{\AA}(6)) \]
		is equal to the ideal generated by the pullback of the image of 
		\[ {\rm{CH}}^{*-3}_{\GL_2}(D_4) \longrightarrow \ch_{\GL_2}(\AA(6)). \]
		The latter can be determined with a standard argument of equivariant intersection theory. Consider the projectivization $\PP^6$ of $\AA(6)$. If we compute the generators of the image of
		\begin{equation}\label{eq:image} {\rm{CH}}^{*-3}_{\GL_2}(\Delta_4) \longrightarrow \ch_{\GL_2}(\PP^6) \end{equation}
		then we are done, because the ideal generated by the pullback to $\ch_{\GL_2}(\AA(6))$ of this image is easy to compute: we only have to substitute the hyperplane class $h$ with $2\lambda_1$, as in  \cite{VisM2}*{pg. 638}.
		
		The image of (\ref{eq:image}) is equal to the image of
		\[ {\rm{CH}}_{\GL_2}^{*-3}(\PP^1\times\PP^2)\longrightarrow \ch_{\GL_2}(\PP^6),\]
		the pushforward along the proper map $(f,g)\mapsto f^4g$. The latter can be easily computed, either as in \cite{VisM2}*{pg. 640-643} or via localization formulas. This concludes the proof.
	\end{proof}
	
	\subsection{The Chow ring of $\Ctilde_2\smallsetminus\ThTilde_2$}
	In this Subsection our main goal is to compute the integral Chow ring of $\Ctilde_2\smallsetminus\ThTilde_2$ (\Cref{prop:chow C minus ThTilde2}) by applying the so called patching lemma. For this, we are only missing one ingredient, namely an explicit description of the integral Chow ring of $\ThTilde_1\smallsetminus\ThTilde_2$.
	\begin{proposition}\label{prop:chow ThTilde1 minus ThTilde2}
		We have
		\[ \ch(\ThTilde_1\smallsetminus\ThTilde_2)\simeq \ZZ[\psi_1,\lambda_1,\lambda_2]/(\lambda_2-\psi_1(\lambda_1-\psi_1)).\]
	\end{proposition}
	
	\begin{proof}
		The objects of the stack $\ThTilde_1\smallsetminus\ThTilde_2$ are of the form $(C\to S,\sigma)$, where:
		\begin{itemize}
			\item $C\to S$ is a family of stable $A_2$-curves of genus $2$ whose geometric fibers are \'{e}tale-locally obtained by gluing two elliptic stable $A_2$-curves at their marked points;
			\item the image of the section $\sigma:S\to C$ does not contain the separating node in any of the geometric fibers of $C\to S$.
		\end{itemize}
		We can define a morphism
		\[ (\Ctilde_{1,1}\smallsetminus \im{\bfp})\times\Mtilde_{1,1}\arr \ThTilde_1\smallsetminus\ThTilde_2 \]
		by sending a pair $(C\to S,p,\sigma),(C'\to S,p')$ to the family of curves determined by gluing $C$ and $C'$ along $p$ and $p'$, marked with $\sigma$. The morphism misses $\ThTilde_2$ because $\Ctilde_{1,1}\setminus \im\bfp$ parametrizes curves in $\Mtilde_{1,1}$ together with an extra section which does not coincide with the previous one. This is easily checked to be an isomorphism, with Lemma~\ref{lem:isom}.

		\Cref{lm:Ctilde11 minus sigma affine bundle} implies that that $(\Ctilde_{1,1}\smallsetminus\im{\bfp})\times\Mtilde_{1,1}$ is an affine bundle over $[V_{-2,-3}/\gm]\times\Mtilde_{1,1}$. As $\Mtilde_{1,1}\simeq [V_{-4,-6}/\gm]$, we deduce
		\[ \ch((\Ctilde_{1,1}\smallsetminus\im{\bfp})\times\Mtilde_{1,1})\simeq \ch_{\gm\times\gm}(V_{-2,-3}\times V_{-4,-6}). \]
		This shows that the Chow ring of $(\Ctilde_{1,1}\smallsetminus\im{\bfp})\times\Mtilde_{1,1}$ is a ring with two generators, the first coming from $\Ctilde_{1,1}\smallsetminus\im{\bfp}$ and the second from $\Mtilde_{1,1}$.
		
		A first generator is then the pullback of the first Chern class of the invertible sheaf on $\Ctilde_{1,1}\smallsetminus\im{\bfp}$ defined by
		\[ (C\to S,p,\sigma)\longmapsto \sigma^*\omega_{C/S}. \]
		The pullback of the line bundle defined above to $(\Ctilde_{1,1}\smallsetminus\im{\bfp})\times\Mtilde_{1,1}\simeq \ThTilde_1\smallsetminus\ThTilde_2$ is isomorphic to the cotangent bundle of the universal section. This implies that a first generator for $\ch(\ThTilde_1\smallsetminus\ThTilde_2)$ is $\psi_1$, the first Chern class of the cotangent bundle of the universal section.
		A second generator is the pullback of the first Chern class of the Hodge line bundle of $\Mtilde_{1,1}$, i.e.
		\[ (\pi':C'\to S,p')\longmapsto \pi'_*\omega_{C'/S}. \]
		After identifying $(\Ctilde_{1,1}\smallsetminus\im{\bfp})\times\Mtilde_{1,1}$ with $\ThTilde_1\smallsetminus\ThTilde_2$, the pullback of the line bundle above gets identified with the line bundle on $\ThTilde_{1}\smallsetminus\ThTilde_2$ defined by
		\[ \cL:(\pi:C\to S,\sigma)\longmapsto \pi_*(\omega_{C/S}(-\im{\sigma})). \] 
		Indeed, we have the restriction homomorphism 
		\[ \omega_{C/S}(-\im{\sigma})\longrightarrow \omega_{C/S}(-\im{\sigma})|_{C'}\simeq i_{C'*}\omega_{C'/S}(p)\simeq i_{C'*}\omega_{C'/S} \]
		that when pushed forward along $\pi$ gives us the morphism $\cL_S\to\pi'_*\omega_{C'/S}$: for this to be an isomorphism, it is enough to check that it induces an isomorphism of the fibers, i.e. that for every geometric point $s\in S$ we have $H^0(C_{s},\omega_{C_s}(-\sigma(s)))\simeq H^0(C'_{s},\omega_{C'_s})$. Any element in the first group vanishes on the marked component because it vanishes on $\sigma(s)$, from which it follows that the restriction map on the fibers is both injective and surjective, hence we get the claimed isomorphism $\cL_s\simeq \pi'_*\omega_{C'/S}$.
		
		Call $s$ the first Chern class of the line bundle above. Observe that for every object $(C\to S,\sigma)$ in $\ThTilde_1\smallsetminus\ThTilde_2$ we have a short exact sequence
		\[ 0 \arr \omega_{C/S}(-\im{\sigma}) \arr \omega_{C/S} \arr \omega_{C/S}|_{\im{\sigma}} \arr 0. \]
		A straightforward application of the Cohomology and Base Change Theorem shows that the following sequence, obtained by pushing forward the sequence above along $\pi$, is actually a short exact sequence of locally free sheaves:
		\[ 0 \arr \pi_*(\omega_{C/S}(-\im{\sigma})) \arr \pi_*\omega_{C/S} \arr \sigma^*\omega_{C/S} \arr 0. \]
		Observe that the vector bundle in the middle is the Hodge vector bundle, the one on the right is the cotangent bundle to the universal section and the one on the left is the line bundle $\cL$ that we just defined. This implies that
		\[ s+\psi_1=\lambda_1,\quad s\psi_1 =\lambda_2. \]
		Putting all together, we deduce that
		\[ \ch(\ThTilde_1\smallsetminus\ThTilde_2)\simeq \ZZ[\lambda_1,\psi_1] \]
		and that $\lambda_2=\psi_1(\lambda_1-\psi_1)$.
	\end{proof}
	Let $\Dtilde_1$ be the divisor inside $\Mtilde_2$ parametrizing curves with one separating node.
	\begin{lemma}\label{lem:delta1}
	We have
		\[ \ch(\Dtilde_1)\simeq \ZZ[\xi_1,\lambda_1,\lambda_2]/(2\xi_1,\xi_1(\lambda_1-\xi_1)). \]
	\end{lemma}
	\begin{proof}
		 The objects of the stack $\Dtilde_1$ are curves $C\arr S$ obtained by gluing two elliptic stable $A_2$-curves $(C'\arr S,p')$ and $(C''\arr S,p'')$ at their marked points. We have an isomorphism
		\[ \Dtilde_1 \simeq [(\Mtilde_{1,1}\times \Mtilde_{1,1})/\bC_2], \]
		where $\bC_2$ acts by sending $((C'\arr S,p'),(C''\arr S,p''))$ to $((C''\arr S,p''),(C'\arr S,p'))$ (here the quotient stack is in the sense of \cite{Rom}).
		If $V_{-4,-6}$ denotes the rank two representation of $\gm$ of weights $-4$ and $-6$, then from \cite{EG}*{5.4} we have
		\[ \left(\Mtilde_{1,1}\times \Mtilde_{1,1}\right)/\bC_2 \simeq [V_{-4,-6}\times V_{-4,-6}/\gm^2\rtimes\bC_2]. \]
		The Chow ring of the stack on the left is isomorphic to $\ch(\cB(\gm^2\rtimes\bC_2))$, which can be extracted from the proof of \cite{DLV}*{Proposition 3.1}:
		\[ \ch(\cB(\gm^2\rtimes\bC_2)) \simeq \ZZ[\xi_1,c_1,c_2]/(2\xi_1,\xi_1c_1), \]
		where $\xi_1$ is the first Chern class of the sign representation of $\bC_2$ and $c_1$ and $c_2$ are the Chern classes of the standard rank two representation of $\gm^2\rtimes\bC_2$, i.e. $V_1\times V_1$ with $\bC_2$ that acts by switching the two factors.
		
		In the proof of \cite{DLV}*{Proposition 3.1} it is also shown that, after identifying $\Dtilde_1$ with $[V_{-4,-6}\times V_{-4,-6}/\gm^2\rtimes\bC_2]$, we have
		\[ \lambda_1=\xi_1-c_1, \quad\lambda_2=c_2. \]
		Putting all together, we obtain the claimed expression for the Chow ring of $\Dtilde_1$.
	\end{proof}	   
	\begin{proposition}\label{prop:chow C minus ThTilde2}
		We have
		\[ \ch(\Ctilde_2\smallsetminus\ThTilde_2)\simeq \ZZ[\psi_1,\vartheta_1,\lambda_1,\lambda_2]/I\]
		where $I$ is the ideal of relations generated by
		\begin{align*}
			\lambda_2-\psi_1(\lambda_1-\psi_1),\\
			(\lambda_1+\vartheta_1)(24\lambda_1^2-48\lambda_2),\\
			20(\lambda_1+\vartheta_1)\lambda_1\lambda_2,\\
			\vartheta_1(\lambda_1+\vartheta_1). 
		\end{align*}
	\end{proposition}
	
	\begin{proof}
		The proof relies on \Cref{lm:Atiyah}. Indeed, we have already computed $\ch(\Ctilde_2\smallsetminus\ThTilde_1)$ in \Cref{prop:chow C minus Theta1} and $\ch(\ThTilde_1\smallsetminus\ThTilde_2)$ in \Cref{prop:chow ThTilde1 minus ThTilde2}. Moreover, the normal bundle of $\ThTilde_1\smallsetminus\ThTilde_2$ in $\Ctilde_2\smallsetminus\ThTilde_2$ is equal to the pullback of the normal bundle of $\Dtilde_1$ in $\Mtilde_2$. The first Chern class of the latter is equal to $\xi_1-\lambda_1$ (see for instance \cite{Lars}*{Lemma 6.1}), and the pullback morphism
		\[ \ch(\Dtilde_1)\arr\ch(\ThTilde_1)\arr\ch(\ThTilde_1\smallsetminus\ThTilde_2)\simeq \ZZ[\lambda_1,\psi_1] \]
		sends $\lambda_1$ to $\lambda_1$ and $\xi_1$ to $0$. Therefore, the first Chern class of the normal bundle of $\ThTilde_1\smallsetminus\ThTilde_2$ is $-\lambda_1$ and in particular it is not a zero divisor.
		
		We can apply \Cref{lm:Atiyah}, which tells us that the following diagram of rings is cartesian:
		\begin{equation}\label{eq:atiyah diagram} \xymatrix{
				\ch(\Ctilde_2\smallsetminus\ThTilde_2) \ar[r]^{j^*} \ar[d]^{i^*} & \ch(\Ctilde_2\smallsetminus\ThTilde_1) \ar[d] \\
				\ch(\ThTilde_1\smallsetminus\ThTilde_2) \ar[r] & \ch(\ThTilde_1\smallsetminus\ThTilde_2)/(\lambda_1).
		}\end{equation}
		We have seen in \Cref{prop:chow C minus Theta1} that $\ch(\Ctilde_2\smallsetminus\ThTilde_1)$ is generated as a ring by $\psi_1$, the first Chern class of the cotangent bundle to the universal section, and $\lambda_1$, the first Chern class of the Hodge line bundle. Moreover, the Chow ring of $\ThTilde_1\smallsetminus\ThTilde_2$ is generated by $\psi_1$ and $\lambda_1$, intended as the first Chern classes of the restrictions of the aforementioned vector bundles. This implies that $\ch(\Ctilde_2\smallsetminus\ThTilde_2)$ is generated by $\lambda_1$, $\psi_1$ and $\vartheta_1$, the fundamental class of $\ThTilde_1\smallsetminus\ThTilde_2$ in $\Ctilde_2\smallsetminus\ThTilde_2$.
		
		Let $f(\lambda_1,\psi_1,\vartheta_1)$ be a relation in $\ch(\Ctilde_2\smallsetminus\ThTilde_2)$. As it is sent to zero by $j^*$ in (\ref{eq:atiyah diagram}), it must be of the form $f'(\lambda_1,\psi_1)+\vartheta_1g(\lambda_1,\psi_1,\vartheta_1)$, where $f'$ belongs to the ideal of relations of $\ch(\Ctilde_2\smallsetminus\ThTilde_1)$. Therefore we have
		\[ 0=i^*f(\lambda_1,\psi_1,\vartheta_1)=f'(\lambda_1,\psi_1) - \lambda_1g(\lambda_1,\psi_1,-\lambda_1). \]
		As $\lambda_1$ is not a zero divisor, we can reconstruct $g(\lambda_1,\psi_1,-\lambda_1)$ from $f'$ by dividing the latter by $\lambda_1$. This implies that
		\[ f(\lambda_1,\psi_1,\vartheta_1) = f'(\lambda_1,\psi_1) + \vartheta_1\frac{f'(\lambda_1,\psi_1)}{\lambda_1} + \vartheta_1 f''(\lambda_1,\psi_1,\vartheta_1),\]
		where $f''(\lambda_1,\psi_1,\vartheta_1)$ belongs to the kernel of $i^*$.
		
		Observe now that, after substituting $\lambda_2$ with $\psi_1(\lambda_1-\psi_1)$, the generators of the ideal of relations of $\ch(\Ctilde_2\smallsetminus\ThTilde_1)$ are
		\[ \quad \lambda_1(24\lambda_1^2-48\psi_1(\lambda_1-\psi_1)),\quad 20\lambda_1^2\psi_1(\lambda_1-\psi_1), \]
		hence the generators of the ideal of relations of $\ch(\Ctilde_2\smallsetminus\ThTilde_2)$ are
		\[ \quad (\lambda_1+\vartheta_1)(24\lambda_1^2-48\psi_1(\lambda_1-\psi_1)),\quad 20(\lambda_1+\vartheta_1)\lambda_1\psi_1(\lambda_1-\psi_1) \]
		together with the generators of $\ker(i^*)$ times $\vartheta_1$. An element $f''(\lambda_1,\psi_1,\vartheta_1)$ belongs to $\ker(i^*)$ if and only if $f''(\lambda_1,\psi_1,-\lambda_1)=0$. This means that $f''$ must be a multiple of $\lambda_1+\vartheta_1$.
		Finally, as the exact sequence 
		\[ 0 \arr \pi_*(\omega_{C/S}(-\im{\sigma})) \arr \pi_*\omega_{C/S} \arr \sigma^*\omega_{C/S} \arr 0 \]
		still holds on $\Ctilde_2\smallsetminus\ThTilde_2$, also the relation $\lambda_2=\psi_1(\lambda_1-\psi_1)$ holds true, thus concluding the computation.
	\end{proof}
	\subsection{The Chow ring of $\Ctilde_2$}
	To apply again the patching lemma and finally obtain a description of $\ch(\Ctilde_2)$ we need to compute the integral Chow ring of $\ThTilde_2$.
	\begin{proposition}\label{prop:chow ThTilde2}
		We have
		\[ \ch(\ThTilde_2)\simeq \ZZ[\xi_1,\lambda_1,\lambda_2]/(2\xi_1,\xi_1(\lambda_1-\xi_1)). \]
	\end{proposition}
	\begin{proof}
		The map $\Ctilde_2\arr\Mtilde_2$ induces an isomorphism $\ThTilde_2\arr\Dtilde_1$. 
	  Using \Cref{lem:delta1} we obtain the claimed expression for the Chow ring of $\ThTilde_2$.
	\end{proof}
	\begin{proposition}\label{prop:chow Ctilde2}
		We have
		\[ \ch(\Ctilde_2)\simeq \ZZ[\psi_1,\vartheta_1,\lambda_1,\lambda_2,\vartheta_2]/I\]
		where $I$ is the ideal generated by
		\begin{align*}
			\lambda_2-\vartheta_2-\psi_1(\lambda_1-\psi_1),\\
			(\lambda_1+\vartheta_1)(24\lambda_1^2-48\lambda_2),\\
			20(\lambda_1+\vartheta_1)\lambda_1\lambda_2,\\
			\vartheta_1(\lambda_1+\vartheta_1),\\
			2\psi_1\vartheta_2,\\
			\vartheta_2(\vartheta_1+\lambda_1-\psi_1),\\
			\psi_1\vartheta_1\vartheta_2.
		\end{align*}
	\end{proposition}
	\begin{proof}
		From the localization exact sequence of Chow groups we see that $\ch(\Ctilde_2)$ is generated by the preimages of the generators of $\ch(\Ctilde_2\smallsetminus\ThTilde_2)$ as a ring (which we computed in \Cref{prop:chow C minus ThTilde2}) together with the pushforward of the generators of $\ch(\ThTilde_2)$ as a $\ZZ$-module (which we computed in \Cref{prop:chow ThTilde2}). The former set of generators is given by $\lambda_1$, $\lambda_2$, $\psi_1$ and $\vartheta_1$: the first two cycles are the Chern classes of the Hodge bundle, which is actually defined on the whole $\Ctilde_2$, the fourth one is the fundamental class of $\ThTilde_1$.
		
		The extension of $\psi_1$ requires some further explanation.
		Recall that on $\Ctilde_2\smallsetminus\ThTilde_2$ the cycle $\psi_1$ is the first Chern class of the cotangent bundle of the universal section, that is
		\[ (C\to S,\sigma)\longmapsto \sigma^*\omega_{C/S}. \]
		The line bundle above can be extended to a line bundle on the whole $\Ctilde_2$ as follows: given a family $C\to S$ with section $\sigma$, let $C'\to S$ be the blow up of $C$ along the separating node, and let $\sigma'$ be the proper transform of $\sigma$, i.e. the proper transform of its image. Then $\sigma'$ is still a section of $C'\to S$, and $(\sigma')^*\omega_{C'/S}$ is a line bundle on $S$, which obviously coincides with the definition above when $C\to S$ belongs to $\Ctilde_2\smallsetminus\ThTilde_2$. The cycle $\psi_1$ is the first Chern class of this line bundle that we just defined.
		
		The second part of the set of generators is formed by the fundamental class $\vartheta_2$ of $\ThTilde_2$ together with the pushforward of the powers of $\lambda_1$, $\lambda_2$ and $\xi_1$. The projection formula tells us that
		\[ i_*\lambda_1^j\lambda_2^k = \lambda_1^j\lambda_2^k\cdot i_*1 = \lambda_1^j\lambda_2^k\vartheta_2,\quad i_*(\lambda_1^j\lambda_2^k\xi_1^\ell)=\lambda_1^j\lambda_2^k\cdot i_*(\xi_1^\ell)  \]
		Moreover, thanks to the relation $\xi_1^2=\lambda_1\xi_1$ in $\ch(\ThTilde_2)$, we have
		\[ i_*(\xi_1^j)=i_*(\xi_1\lambda_1^{j-1})=\lambda_1^{j-1}\cdot i_*(\xi_1). \]
		Putting all together we deduce that $\ch(\Ctilde_2)$ is generated as a ring by
		\[ \lambda_1,\lambda_2,\psi_1,\vartheta_1,\vartheta_2,\eta_3:=i_*\xi_1. \]
		
		The next step is to find the relations among the generators. For this, we leverage \Cref{lm:Atiyah} in order to compute the Chow ring of $\Ctilde_2$. We need to determine the top Chern class of the normal bundle of $\ThTilde_2$ in $\Ctilde_2$. By \cite{DLV}*{Proposition A.4}, we have
		\[ \cN_{\ThTilde_2}(C'\cup_{\sigma} C''\to S,\sigma)= TC'_{\sigma}\oplus TC''_{\sigma}\]
		Recall that
		\[ \ThTilde_2 \simeq (\Mtilde_{1,1}\times\Mtilde_{1,1})/\bC_2\simeq [V_{-4,-6}\times V_{-4,-6}/\gm^2\rtimes\bC_2]. \]
		
		If $\bfp\colon\Mtilde_{1,1}\arr\Ctilde_{1,1}$ denotes the universal section, the conormal bundle of $\im{\bfp}$ is
		\[ (C\to S,p)\longmapsto TC_p \]
		and its total space is isomorphic to $[V_{-4,-6}\times V_1/\gm]$. We deduce that the total space of $\cN_{\ThTilde_2}$ is isomorphic to
		\[ [ (V_{-4,-6}\times V_{-4,-6})\times (V_1\times V_1)/\gm^2\rtimes\bC_2 ],  \]
		where $\bC_2$ acts on $V_1\times V_1$ by switching the two factors. In other terms, the normal bundle is isomorphic to the pullback from $\cB(\gm^2\rtimes\bC_2)$ of the rank two vector bundle associated with the standard representation of $\gm^2\rtimes\bC_2$, i.e. the standard representation of $\GL_2$, regarded as a representation of $\gm^2\rtimes\bC_2$ via the natural map that sends $\gm^2$ to the subtorus of diagonal matrices and the generator of $\bC_2$ to the transposition. As already said in the proof of \Cref{prop:chow ThTilde2}, the second Chern class of this vector bundle is equal to $\lambda_2$, which is not a zero divisor in $\ch(\ThTilde_2)$.
		
		To proceed, we need to know how the pullback morphism $i^*\colon\ch(\Ctilde_2)\arr\ch(\ThTilde_2)$ works. The first formula below follows from the fact that the forgetful morphism $\ThTilde_2\to\Dtilde_1$ is an isomorphism, the other ones are a straightforward consequence of the projection formula:
		\begin{equation}\label{eq:pullback theta2} i^*\lambda_i=\lambda_i,\quad i^*\eta_3=\lambda_2\xi_1, \quad i^*\vartheta_2=\lambda_2. \end{equation}
		For the pullback of $\psi_1$, let $(C\to S,\sigma)$ belong to $\ThTilde_2$. Denote $E$ the exceptional divisor of the blow up of $C$ along the separating node, and consider the family of genus $0$ curves $E\to S$: this family is marked by the proper transform of $\sigma$ and by a divisor of degree $2$ given by the intersection points of $E$ with the other components. Therefore, there exists an \'{e}tale double cover $f:S'\to S$ such that $E'=E\times_S S'\simeq \PP^1_{S'}$. In particular $(\sigma')^* \omega_{E'/S'} \simeq \cO_{S'}$, hence $f^*(\psi_1|_S)=0$ and $2(\psi_1|_S)=f_*f^*(\psi_1|_S)=0$. Then we must necessarily have 
		\begin{equation}\label{eq:pullback psi1} i^*\psi_1=\xi_1.\end{equation}
		To compute $i^*\vartheta_1$, observe that we have a commutative diagram
		\[\begin{tikzcd}        
			\ThTilde_2 \arrow[hook]{r}{i} \arrow{d}{\rho} & \Ctilde_2 \arrow{d}{\pi}  \\
			\Dtilde_1  \arrow[hook]{r}{f} & \Mtilde_2.
		\end{tikzcd}\]
		This implies that 
		\begin{align*}
			i^*\vartheta_1=i^*\pi^*\delta_1=\rho^*f^*\delta_1=\rho^*(\xi_1-\lambda_1),
		\end{align*} 
		and identifying $\ch(\ThTilde_2)$ with $\ch(\Dtilde_1)$ via $\rho^*$ as in \Cref{prop:chow ThTilde2}, we deduce that
		\[ i^*\vartheta_1=\xi_1-\lambda_1. \]
		Observe that 
		\[\eta_3=i_*\xi_1=i_*i^*\psi_1=\psi_1\vartheta_2\]
		hence we do not need $\eta_3$ in order to have generators of $\ch(\Ctilde_2)$ as a ring.
		
		We can apply \Cref{lm:Atiyah}, which tells us that the following commutative diagram of rings is cartesian:
		\begin{equation}
			\xymatrix{
				\ch(\Ctilde_2) \ar[r]^{j^*}\ar[d]^{i^*} & \ch(\Ctilde_2\smallsetminus\ThTilde_2) \ar[d] \\
				\ch(\ThTilde_2) \ar[r] & \ch(\ThTilde_2)/(\lambda_2)}
		\end{equation}
		Therefore, the ideal of relations is given by those polynomials in the generators that are sent to zero both in $\ch(\Ctilde_2\smallsetminus\ThTilde_2)$ and in $\ch(\ThTilde_2)$. The generators of this ideal are therefore of the form $f=f'+\vartheta_2 g$, where by \Cref{prop:chow C minus ThTilde2} we know that $f'$ is $0$ or one of the following
		\begin{align}\label{eq:four rels}
			\lambda_2-\psi_1(\lambda_1-\psi_1),\nonumber\\
			(\lambda_1+\vartheta_1)(24\lambda_1^2-48\lambda_2), \nonumber\\
			20(\lambda_1+\vartheta_1)\lambda_1\lambda_2, \\
			\vartheta_1(\lambda_1+\vartheta_1). \nonumber
		\end{align}
		 Moreover we have
		\[ 0= i^*f = i^*f' + \lambda_2 (i^*g) \Rightarrow i^*g = \frac{i^*(-f')}{\lambda_2}. \]
		If $f'$ is one of the last three elements of \eqref{eq:four rels} or zero, then $i^*f'=0$, hence $g$ belongs to $\ker(i^*)$. If $f'=\lambda_2-\psi_1(\lambda_1-\psi_1)$, then $i^*g=(-1)$. Hence the ideal of relations is generated by $\ker(i^*)$ times $\vartheta_2$ and
		\begin{align*}
			\lambda_2-\psi_1(\lambda_1-\psi_1)-\vartheta_2,\\
			(\lambda_1+\vartheta_1)(24\lambda_1^2-48\lambda_2),\\
			20(\lambda_1+\vartheta_1)\lambda_1\lambda_2,\\
			\vartheta_1(\lambda_1+\vartheta_1).
		\end{align*}
		The determination of the generators of $\ker(i^*)$ is straightforward. Multiplying them by $\vartheta_2$, we get the following set of cycles:
		\[ 2\psi_1\vartheta_2,\quad \vartheta_2(\vartheta_1+\lambda_1-\psi_1),\quad \psi_1\vartheta_1\vartheta_2. \]
		Putting all together, we reach the desired conclusion.
	\end{proof}
	
	\section{The Chow ring of $\Mbar_{2,1}$, abstract characterization}\label{sec:abstract}
	The goal of this Section is to find a minimum number of generators for the kernel of the restriction map
	\[\ch(\Ctilde_2)\longrightarrow\ch(\Mbar_{2,1})\]
	In \Cref{thm:chow Mbar21 abs} we show that, up to relations coming from the integral Chow ring of $\Mbar_2$, this kernel is generated by the fundamental classes of two substacks, together with an additional cycle: the first fundamental class, denoted $\Ctilde_2^{\rm c}$, is the cuspidal singularity locus of $\Ctilde_2$; the second one, dubbed $\Ctilde_2^{\rm E}$, is the substack of cuspidal curves of genus two with a separating node and a marked rational component, with its reduced scheme structure. The additional cycle comes also from this second locus.
	
	\subsection{Relations from the locus of curves with two cusps}
	Notice that an $A_{2}$ curve of genus $2$ cannot have more than two cusps. Let $\Mtilde_2^{\rm c, \rm c}$ denotes the closed substack of $\Mtilde_2$ parametrizing stable $A_2$-curves of genus two with two cusps, with its reduced scheme structure. In other words, a section of $\Mtilde_2^{\rm c, \rm c}(S)$ is a family of stable $A_{2}$-curves $C \arr S$, which is trivial \'{e}tale-locally around each node. This is easily seen to be a smooth closed substack of codimension $4$.
	
	We set
	\[ \cZ^2 := \pi^{-1} \left( \Mtilde_2^{\rm{c,c}} \right). \]
	We think of $\cZ^2$ is the stack of pairs $(C\arr S,\sigma)$ where the geometric fibers of $C\arr S$ have two cusps. The pushforward along the closed embedding of $\cZ^2$ in $\Ctilde_2$ induces
	\[ {\rm CH}_*(\cZ^2) \arr {\rm CH}_*(\Ctilde_2). \]
	Our goal is to prove the following.
	
	\begin{proposition}\label{prop:relations two cusps}
		Let $J$ denote the ideal generated by the pullback along $\pi\colon\Ctilde_2\arr\Mtilde_2$ of the cuspidal relations. Then the image of 
		${\rm CH}_*(\cZ^2) \arr {\rm CH}_*(\Ctilde_2)$ is contained in $J$.
	\end{proposition}
	Let $\Mbar_0$ be the stack of genus $0$ curves with at most one node. In \cite{EFRat}*{Proposition 6} the authors give the following presentation of $\Mbar_0$: let $E$ be the standard representation of $\GL_3$, and set
	\[ V:={\rm Sym}^2E^{\vee}\otimes\det(E). \]
	We can regard $V$ as the vector space of quadratic ternary forms, endowed with a twisted $\GL_3$-action. If $U\subset V$ is the invariant open subscheme of quadratic forms of rank $>1$, then
	\[ \Mbar_0 \simeq [U/\GL_3]. \]
	Let $X\subset V\times \PP(E)$ be the closed subscheme of equation
	\[ \alpha_{00}X^2 + \alpha_{01} XY + \alpha_{02} XZ + \alpha_{11} Y^2 + \alpha_{12} YZ + \alpha_{22} Z^2 = 0 .  \]
	Set $C:= X\cap (U\times\PP(E^{\vee}))$, so that $C\arr U$ is a family of genus $0$ curve with at most one node.
	If $\pi\colon\Cbar_0\arr \Mbar_0$ is the universal curve of genus $0$, then we have
	\[ \Cbar_0 \simeq [C/\GL_3]. \]
	\begin{lemma}\label{lm:chow Cbar0}
		Let $h:=c_1(\cO(1))$ be the hyperplane class of $\PP(E)$. Then the pullback homomorphism of $\ch(\cB\GL_3)$-modules
		\[ \ch(\cB \GL_3)\langle 1, h, h^2\rangle \simeq \ch(\PP(E)) \arr \ch(\Cbar_0) \]
		is surjective.
	\end{lemma}
	\begin{proof}
		Observe that $X \arr \PP(E)$ is an equivariant vector subbundle of $V\times\PP(E) \arr \PP(E)$, because the equation defining $X$ is linear in the coefficients $\alpha_{ij}$. Therefore, the induced pullback at the level of equivariant Chow rings (which are the Chow rings of the quotients) is an isomorphism.
		The pullback along open embeddings is always surjective, hence so it is the pullback along $[X/\GL_3]\arr \Cbar_0$. Composing these two homomorphisms, we obtain the claimed surjection: the characterization of $\ch(\PP(E))$ follows from the usual projective bundle formula.
	\end{proof}
	Recall from \cite{EFRat}*{Lemma 8} that the sheaf $F:=\pi_*(\omega_{\pi}^{\vee})$ is a vector bundle of rank three on $\Mbar_0$. In particular, we can consider the stack $\PP(F)$: its objects are pairs $(C\arr S, D)$ where $C\arr S$ is a family of genus $0$ curves with at most one node, and $D\subset C$ is a divisor such that the induced map $D\arr S$ is finite of degree $2$. Moreover the restriction of $D$ to a singular curve must have degree $1$ on each component: in other words, the restriction of $D$ is supported either on two distinct points lying on two distinct components, or on the node.
	
	Consider the cartesian diagram 
	\[ \xymatrix{
		\Cbar_0 \times_{\Mbar_0} \PP(F) \ar[r] \ar[d] & \PP(F) \ar[d] \\
		\Cbar_0 \ar[r] & \Mbar_0.
	} \]
	Then the left vertical map is a projective bundle and we have
	\[ \ch(\Cbar_0 \times_{\Mbar_0} \PP(F)) \simeq \ch(\Cbar_0)\langle 1,t,t^2 \rangle \]
	where $t$ is the pullback of the hyperplane section of $\PP(F)$. Together with \Cref{lm:chow Cbar0} this tells us that $\ch(\Cbar_0 \times_{\Mbar_0} \PP(F))$ is generated as a $\ch(\cB\GL_3)$-module by the pullbacks of $1$, $h$, $t$, $h^2$, $ht$, $t^2$, $h^2t$, $ht^2$, $h^2t^2$.
	We have proved the following.
	\begin{lemma}\label{lm:chow Cbar0 times P(F)}
		The Chow ring of $\Cbar_0 \times_{\Mbar_0} \PP(F)$ is generated as a $\ch(\PP(F))$-module by the pullbacks of $1$, $h$ and $h^2$.
	\end{lemma}
	
	Let $\PP(F)^{\rm \acute{e}t}$ be the open substack of $\PP(F)$ formed by the pairs $(C\arr S, D)$ where $D\arr S$ is \'{e}tale.
	The family of genus zero curves $\Cbar_0 \times_{\Mbar_0} \PP(F)^{\rm \acute{e}t} \arr \PP(F)^{\rm \acute{e}t}$ has by construction a divisor $\cD$ whose projection onto $\PP(F)^{\rm \acute{e}t}$ is \'{e}tale of degree $2$.
	
	We can pinch this family of curves along the divisor $\cD$: the resulting scheme is a family of stable $A_2$-curves of genus two over $\PP(F)^{\rm \acute{e}t}$, with each fiber having two cusps. This in turn induces a map $f\colon\PP(F)^{\rm \acute{e}t}\arr\Mtilde_2$ which fits into the following commutative diagram:
	\begin{equation}\label{eq:diag envelope Z2}
		\begin{tikzcd}
			\Cbar_0 \times_{\Mbar_0} \PP(F)^{\rm \acute{e}t} \ar[rd] \ar[r, "\cP"] & f^*\Ctilde_2 \ar[r] \ar[d] & \Ctilde_2 \ar[d, "\pi"] \\
			& \PP(F)^{\rm \acute{e}t} \ar[r, "f"] & \Mtilde_2
		\end{tikzcd}
	\end{equation}
	Here $\cP$ denotes the pinching morphism, and $f^*\Ctilde_2$ is the resulting family of curves. The morphism $f$ sends a pair $(C\arr S, D)$ to the family of curves $\widehat{C}\arr S$ obtained by pinching $C$ along the support of $D$.
	
	Observe that by construction the morphism $f$ is an isomorphism onto $\Mtilde_2^{\rm{c,c}}$, the locus of curves with two cusps, hence $f^*\Ctilde_2$ is isomorphic to $\cZ^2:=\pi^{-1}(\Mtilde_2^{\rm{c,c}})$. As the pinching morphism induces an isomorphism at the level of Chow rings, we deduce the following:
	\begin{lemma}\label{lm:chow envelope Z2}
		The induced morphism
		\[ \Cbar_0 \times_{\Mbar_0} \PP(F)^{\rm \acute{e}t} \arr \Ctilde_2 \]
		is a Chow envelope for $\cZ^2:=\pi^*(\Mtilde_2^{\rm{c,c}})$.
	\end{lemma}
	By definition, the cycle $\psi_1$ in $\ch(\Ctilde_2)$ is the first Chern class of the line bundle
	\[ (C\arr S,\sigma) \longmapsto \sigma^*\omega_{C/S} \]
	Hence its pullback to $\Cbar_0 \times_{\Mbar_0} \PP(F)^{\rm \acute{e}t}$ is the invertible sheaf
	\[ (C'\arr S,\sigma,D)\longmapsto \sigma^*\omega_{C'/S}, \]
	where $C'$ is the partial normalization of $C$ at the cusps.
	It is clear that the line bundle above comes from a line bundle $\cL$ on $\Cbar_0$, which can be characterized as follows: consider the cartesian diagram
	\[
	\begin{tikzcd}
		\Cbar_0 \times_{\Mbar_0} \Cbar_0 \ar[r, "pr_2"] \ar[d, "pr_1"] & \Cbar_0 \ar[d, "\pi"] \\
		\Cbar_0 \ar[r, "\pi"] \ar[u, bend left, "\delta"] & \Mbar_0
	\end{tikzcd}
	\]
	where $\delta$ is the diagonal embedding. The pair $(\Cbar_0\times_{\Mbar_0}\Cbar_0\arr \Cbar_0,\delta)$ is the universal marked curve, and the line bundle $\cL$ is
	\[ \delta^*pr_2^*\omega_{\pi} = \omega_{\pi}. \]
	\Cref{lm:chow Cbar0} tells us that $\ch(\Cbar_0)$ is generated as a $\ch(\Mbar_0)$-module by $1$, $h$ and $h^2$. We claim that $h=-c_1(\cL)$.
	
	For this, recall that there is a commutative diagram
	\[
	\begin{tikzcd}
		\left[ C/\GL_3 \right] \ar[r,hook] & \left[ U\times\PP(E)/\GL_3 \right] \ar[r] \ar[d] & \left[\PP(E)/\GL_3\right] \\
		& \left[U/\GL_3\right] \ar[r] & \left[\spec{k}/\GL_3\right]
	\end{tikzcd}
	\]
	and that $h$ is obtained by pulling back the hyperplane section of $[\PP(E)/\GL_3]$, which coincides with the restriction to $[C/\GL_3]$ of the hyperplane section of $[U\times\PP(E)/\GL_3]$, regarded as a projective bundle on $[U/\GL_3]$. In more intrinsic terms, this projective bundle is $\PP((\pi_*\omega_{\pi})^{\vee})$, and the embedding of $\Cbar_0$ in $\PP(\pi_*\omega_{\pi}^{\vee})$ is induced by the surjection
	\[ \pi^*\pi_*(\omega_{\pi}^{\vee}) \arr \omega_{\pi}^{\vee}. \]
	Therefore, by the universal property of projective bundles, the restriction of $\cO(1)$ to $\Cbar_0$ is equal to $\omega_{\pi}^\vee$, from which we deduce that $\cL^\vee \simeq \cO(1)|_{\Cbar_0}$. Putting all together, we have proved the following.
	\begin{lemma}\label{lm:pullback h}
		The pullback of $\psi_1$ along $\Cbar_0\times_{\Mbar_0}\PP(F)^{\rm \acute{e}t} \rightarrow \Ctilde_2$ is equal to $-h$.
	\end{lemma}
	We are ready to prove the result stated at the beginning of the Subsection.
	\begin{proof}[Proof of \Cref{prop:relations two cusps}]
		It follows from \Cref{lm:chow envelope Z2} that the image of the pushforward
		\[ \ch(\cZ^2) \arr \ch(\Ctilde_2) \]
		coincides with the image of
		\[ \ch(\Cbar_0\times_{\Mbar_0}\PP(F)^{\rm \acute{e}t}) \arr \ch(\Ctilde_2). \]
		The diagram (\ref{eq:diag envelope Z2}) induces the following commutative diagram of Chow groups:
		\[
		\begin{tikzcd}
			\ch(\Cbar_0\times_{\Mbar_0}\PP(F)^{\rm \acute{e}t}) \ar[r, "\simeq"] \ar[rr, bend left, "g_*"] & \ch(f^*\Ctilde_2) \ar[r] & \ch(\Ctilde_2) \\
			& \ch(\PP(F)^{\rm \acute{e}t}) \ar[lu] \ar[u] \ar[r, "f_*"] & \ch(\Mtilde_2) \ar[u, "\pi^*"]
		\end{tikzcd}
		\]
		On one hand, we know from \Cref{lm:chow Cbar0 times P(F)} that the Chow ring of $\Cbar_0\times_{\Mbar_0}\PP(F)^{\rm \acute{e}t}$ is generated as a $\ch(\PP(F)^{\rm \acute{e}t})$-module by $1$, $h$, $h^2$. On the other hand \Cref{lm:pullback h} says that $g^*\psi_1=-h$, hence by the projection formula the image of $g_*$ is generated as an ideal by $\pi^*(\im{f_*})$.
		
		By construction $\im{f_*}$ is contained in the ideal of cuspidal relations, hence its pullback along $\pi$ is contained in $J$. This concludes the proof.
	\end{proof}
	
	\subsection{Relations from the locus of curves with one cusp}
	Consider the stack $\Mbar_{1,1}\times [\AA^1/\gm]$, where the action of $\gm$ on $\AA^1$ has weight $1$. In Subsection \ref{sub:description cusp} we defined a morphism of stacks
	\[ c\colon\Mtilde_{1,1}\times [\AA^1/\gm] \arr \Mtilde_2 \]
	that does the following: on one hand, when restricted to the open substack $\Mtilde_{1,1}\simeq\Mtilde_{1,1}\times [\AA^1\smallsetminus\{0\}/\gm]$, it sends an elliptic stable $A_2$-curve $(C\arr S,p)$ to the stable $A_2$-curve of genus two $\widehat{C}\arr S$ obtained by pinching the section $p$ in $C$. 
	
	On the other hand, we can regard the closed substack $\Mtilde_{1,1}\times \cB\gm$ as the stack of pairs $((C\arr S,p),(C'\arr S,p'))$ where $(C\arr S,p)$ is an elliptic stable $A_2$-curve and $(C'\arr S,p')$ is a family of elliptic stable $A_2$-curves with a cusp. Then the morphism $c$ sends a pair to the stable $A_2$-curve of genus two obtained by gluing $C$ and $C'$ at the marked points.
	
	Let us recall how to construct this morphism: let $\rm{Bl}(\Ctilde_{1,1}\times [\AA^1/\gm])$ be the blow-up of $\Ctilde_{1,1}\times [\AA^1/\gm]$ along the closed substack
	\[ \bfp\times 0\colon\Mtilde_{1,1}\times\cB\gm \hooklongrightarrow \Ctilde_{1,1}\times [\AA^1/\gm]. \]
	The fibers of $\rm{Bl}(\Ctilde_{1,1}\times [\AA^1/\gm])\arr \Mtilde_{1,1}\times [\AA^1/\gm]$ over the closed substack $\Mtilde_{1,1}\times\cB\gm$ have two irreducible components, meeting in one node: the first component is given by a stable $A_2$-curve of genus one, the second is a rational curve.
	
	Moreover, the proper transform of $\im{\bfp\times\rm{id}}$, once restricted to the fibers over $\Mtilde_{1,1}\times\cB\gm$, determines a section whose image lands in the rational curve.
	
	We can apply the pinching construction (Subsection \ref{sub:pinching}) to the curve $\rm{Bl}(\Ctilde_{1,1}\times [\AA^1/\gm])\arr \Mtilde_{1,1}\times [\AA^1/\gm]$ along the proper transform of $\im{\bfp\times\rm{id}}$, so that we get
	\[ \cP\colon\rm{Bl}(\Ctilde_{1,1}\times [\AA^1/\gm])\arr \widehat{\cD}_{1,1}.  \]
	By construction the curve $\widehat{\cD}_{1,1}\arr\Mtilde_{1,1}\times\cB\gm$ is a stable $A_2$-curve of genus two, hence we obtain the following fundamental diagram
	\begin{equation}\label{eq:fund diag Z1}
		\begin{tikzcd}
			\rm{Bl}(\Ctilde_{1,1}\times [\AA^1/\gm]) \ar[r, "\cP"] \ar[d, "\rho"]  & \widehat{\cD}_{1,1}\simeq c^*\Ctilde_{2} \ar[r, "c'"] \ar[d, "\pi'"] & \Ctilde_2 \ar[d, "\pi"] \\
			\Ctilde_{1,1}\times \left[\AA^1/\gm\right] \ar[r, "q"] & \Mtilde_{1,1}\times\left[\AA^1/\gm\right] \ar[r, "c"]  \ar[l, bend right, "\bfp\times\rm{id}"]  & \Mtilde_2
		\end{tikzcd}
	\end{equation}
	\begin{remark}
		The morphism $c\colon\Mtilde_{1,1}\times [\AA^1/\gm]\arr\Mtilde_2$ is proper and surjects onto $\Mtilde_2^{\rm c}$, the substack of stable $A_2$-curves with at least a cusp. Moreover, it is injective away from $\Mtilde_2^{\rm{c,c}}$, the substack of curves with two cusps.
		
		Consequently, the morphism $c':c^*\Ctilde_2 \arr \Ctilde_2$ is proper and surjects onto $\cZ^1=\pi^{-1}(\Mtilde_2^{\rm c})$ and it is injective away from $\cZ^2=\pi^{-1}(\Mtilde_2^{\rm{c,c}})$.
	\end{remark}
	Recall that $J\subset\ch(\Ctilde_2)$ is the ideal generated by the pullback along $\pi\colon\Ctilde_2\to\Mtilde_2$ of the ideal of cuspidal relations in $\ch(\Mtilde_2)$.
	
	We also introduced the cycles $[\Ctilde_2^{\rm c}]$ and $[\Ctilde_2^{\rm E}]$: the first one is the fundamental class of the cuspidal locus of the universal curve $\Ctilde_2\arr \Mtilde_2$ (to not be confused with the locus of cuspidal curves!). The second one is the fundamental class of the locus of curves with a separating node such that the marked component is a cuspidal curve of geometric genus zero.
	
	Let $E$ be the exceptional divisor of $\rm{Bl}(\Ctilde_{1,1}\times [\AA^1/\gm])$, so that we have a morphism $\rho:E\to \Mtilde_{1,1}\times\cB\gm$ and a proper morphism $c'':E\to\Ctilde_2$ ((there is a little abuse of notation here, as for the sake of simplicity we denoted $\rho$ what should be denoted $\rho|_E$). Call $T$ the first Chern class of the Hodge line bundle on $\Mtilde_{1,1}\times\cB\gm$. Then the main result of this Subsection is the following:
	\begin{proposition}\label{prop:1-cusp whole}
		The image of 
		\begin{equation}\label{eq:pushforward blowup} \ch(\rm{Bl}(\Ctilde_{1,1}\times [\AA^1/\gm])) \arr \ch(\Ctilde_2) \end{equation}
		is contained in the ideal $(J,[\Ctilde_2^{\rm c}],[\Ctilde_2^{\rm E}],c''_*\rho^*T).$
	\end{proposition}
	The Chow ring of a blow-up is generated by the pullback of the cycles from the base plus the cycles coming from the exceptional divisor.
	Let $i:E\hookrightarrow \rm{Bl}(\Ctilde_{1,1}\times [\AA^1/\gm])$ be the exceptional divisor, then the image of (\ref{eq:pushforward blowup}) coincides with the image of
	\begin{equation}\label{eq:image blow-up}
		\begin{tikzcd}
			\ch(\Ctilde_{1,1}\times \left[\AA^1/\gm\right])\oplus\ch(E) \ar[rrr, "c'_*\cP_*\rho^* + c'_*\cP_*i_*"] & & & \ch(\Ctilde_2).
		\end{tikzcd}
	\end{equation}
	Let us separately study the two factors of the homomorphism above. 
	\begin{proposition}\label{prop:1-cusp}
		The image of $\ch(\Ctilde_{1,1}\times \left[\AA^1/\gm\right])\arr \ch(\Ctilde_2)$ is contained in the ideal $(J,[\Ctilde_2^{\rm c}],[\Ctilde_2^{\rm E}])$.
	\end{proposition}
	
	The same argument of \Cref{prop:chow Ctilde11} implies that
	\[ (j\times \id)^*q^*\colon\ch(\Mtilde_{1,1}\times [\AA^1/\gm])\arr\ch((\Ctilde_{1,1}\smallsetminus\im{\bfp})\times [\AA^1/\gm]) \]
	is surjective. We deduce that Chow ring of $\Ctilde_{1,1}\times[\AA^1/\gm]$ is generated by the cycles coming from $\Mtilde_{1,1}\times [\AA^1/\gm]$ together with the ones coming from $\Mtilde_{1,1}\times[\AA^1/\gm]$ via
	\[ (\bfp\times\rm{id})_*\colon\ch(\Mtilde_{1,1}\times[\AA^1/\gm])\arr\ch(\Ctilde_{1,1}\times[\AA^1/\gm]). \]
	If a cycle is of the form $q^*\zeta$, the diagram (\ref{eq:fund diag Z1}) tells us that
	\begin{align*} c'_*\cP_*\rho^*(q^*\zeta ) = c'_*(\pi')^*\zeta = \pi^*c_*\zeta,
	\end{align*}
	hence it is contained in $J$.
	
	To prove \Cref{prop:1-cusp}, it remains to study the image in $\ch(\Ctilde_2)$ of the cycles of the form $(\bfp\times\rm{id})_*\zeta$. This corresponds to the image of the map $f$ in the commutative diagram below:
	\begin{equation*}
		\begin{tikzcd}
			& \ch(\Ctilde_2) \ar[d, "\pi"] \\
			\ch(\Mtilde_{1,1}\times[\AA^1/\gm]) \ar[ur,"f"] \ar[r, "c"] & \ch(\Mtilde_2)
		\end{tikzcd}
	\end{equation*}
	When restricted to the open substack $[\Mtilde_{1,1}\times [\AA^1\smallsetminus\{0\}/\gm]\simeq \Mtilde_{1,1}$, the map $f$ sends an elliptic stable $A_2$-curve $(C\to S,p)$ to the marked genus two curve obtained by creating a cusp along $\im{p}$ and with the marking given by the cusp itself.
	
	Recall that
	\[ \ch(\Mtilde_{1,1}\times[\AA^1/\gm]) \simeq \ZZ[T,S] \]
	where $T$ is the dual of the Hodge line bundle on $\Mtilde_{1,1}$ and $S$ is the first Chern class of the universal line bundle on $\cB\gm$.
	\begin{lemma}\label{prop:cusp rel M2tilde}
		In the setting above, we have $c^*(\lambda_1)= -S$ and $c^*(\lambda_2)=ST-T^2$, hence also $f^*(\lambda_1)= -S$ and $f^*(\lambda_2)=ST-T^2$. Moreover $f^*\psi_1=T+nS$ for some integer $n$.
	\end{lemma}
	\begin{proof}
		The pullback along the closed embedding
		\[ [\{0\}/\gm]\times [\{0\}/\gm] \hooklongrightarrow \Mtilde_{1,1}\times [\AA^1/\gm]\]
		is an isomorphism of Chow rings, hence to determine $c^*(\lambda_1)$ and $c^*(\lambda_2)$ we can equivalently compute their restrictions to $[\{0\}/\gm]\times [\{0\}/\gm]$.
		
		Consider the composition
		\[ \{0\}\times\{0\} \arr[\{0\}/\gm]\times [\{0\}/\gm] \hooklongrightarrow \Mtilde_{1,1}\times [\AA^1/\gm] \overset{c}{\larr} \Mtilde_2  \]
		where the first arrow is the universal $\gm\times\gm$-torsor. The pullback of the Hodge bundle on $\Mtilde_2$ along this composition, i.e. the fiber of the Hodge bundle on the closed point $\{0\}\times\{0\}\arr\Mtilde_2$, can be regarded as a $\gm\times\gm$-representation of rank two. The equivariant Chern classes of this representation coincide by construction with $c^*(\lambda_1)$ and $c^*(\lambda_2)$.
		
		Let $\widehat{C}$ denote the elliptic stable $A_2$-curve represented by the closed point $\{0\}\times\{0\}$: this is the genus $2$ curve obtained by gluing together the planar curve $C=\{Y^2Z=X^3\}$ with a projective line pinched in $[1,0]$. The points that we glue together are respectively $[0,1,0]$ and $[0,1]$. Here it follows a more detailed construction.
		
		Recall that the family of stable $A_2$-curves of genus two over $\Mtilde_{1,1}\times[\AA^1/\gm]$ inducing the map to $\Mtilde_2$ is constructed as follows: let $\widetilde{C}_{1,1}\subset\PP(V_{-2,-3,0})\times V_{-4,-6}$ be the $\gm$-invariant family of elliptic curves of equation
		\[ Y^2Z=X^3+aXZ^2+bZ^3, \]
		with section given by $p:(a,b)\mapsto [0,1,0],(a,b)$. The quotient stack $[\widetilde{C}_{1,1}/\gm]$ is isomorphic to $\Ctilde_{1,1}$, the universal family over $\Mtilde_{1,1}$.
		Consider the blow up of $\widetilde{C}_{1,1}\times V_1$ along $p(\AA^2)\times \{0\}$, which naturally inherits a $\gm\times\gm$ action. If we pinch this blow up along the proper transform of $p(\AA^2)\times \AA^1$, we get a family of stable $A_2$-curves of genus two, whose $\gm\times\gm$-quotient is the family on $\Mtilde_{1,1}\times[\AA^1/\gm]$ we were looking for.
		
		In particular, the fiber over $[\{0\}/\gm]\times [\{0\}/\gm]$ of this family is given by the $\gm\times\gm$-quotient of $\{Y^2Z=X^3\}$ and the projectivization of the tangent space of $\{[0,1,0]\}\times\{0\}$ in $\{Y^2Z=X^3\}\times \AA^1$. If $u$ is a coordinate for $\AA^1$, a basis for this vector space is given by $e_0=u^{\vee}$ and $e_1=\frac{X}{Y}^{\vee}$, on which $\gm\times\gm$ acts as
		\[ (t,s)\cdot (e_0,e_1) = (se_0,te_1), \]
		from which follows that $\gm\times\gm$ acts on the projective line as
		$$ (t,s)\cdot[e_0,e_1] = [ se_0, te_1].$$
		
		The fiber of the Hodge bundle on $\{0\}\times \{0\}$ is isomorphic to $\rm{H}^0(\widehat{C},\omega_{\widehat{C}})$. We need to compute the inherited action of $\gm\times\gm$ on this vector space. A straightforward computation shows that
		\[ H^0(\widehat{C},\omega_{\widehat{C}})=H^0(C,\omega_C) \oplus H^0(\PP^1,\omega_{\PP^1}(2[1,0])),\]
		hence
		\[c_1^{\gm^2}(H^0(\widehat{C},\omega_{\widehat{C}}))= -T + c_1^{\gm^2}(H^0(\PP^1,\omega_{\PP^1}(2[1,0]))\]
		\[c_2^{\gm^2}(H^0(\widehat{C},\omega_{\widehat{C}}))= -T\cdot  c_1^{\gm^2}(H^0(\PP^1,\omega_{\PP^1}(2[1,0])).\]
		
		It remains to compute the action of $\gm^2$ over $\rm{H}^0(\PP^1,\omega_{\PP^1}(2 [1,0]))$.
		The vector space $\rm{H}^0(\PP^1,\omega_{\PP^1}(2[1,0]))$ is generated by $(x_0/x_1)^2d(x_1/x_0)$, where $x_0=u=e_0^{\vee}$ and $x_1=\frac{X}{Y}=e_1^{\vee}$. Therefore, the action is
		\[ (t,s)\cdot(x_0/x_1)^2d(x_1/x_0) = s^{-1}t (x_0/x_1)^2d(x_1/x_0) \]
		and the first Chern class of the representation is $T-S$. Putting all together, we get the desired conclusion. We are left with the computation of $f^*\psi_1$.
		
		Let $(C\to S,\sigma)$ be a stable $A_2$-curve of genus two, with $\im{\sigma}$ contained in the cuspidal locus of $C\to S$. Suppose moreover that the fibers of $C$ have no separating nodes. Let $\overline{C}\to S$ be the partial normalization of $C$ along $\im{\sigma}$, and let $p:S\to \overline{C}$ be the preimage of $\sigma$, so that we have
		\begin{equation*}
			\begin{tikzcd}
				\overline{C} \ar[rr,"\nu"] \ar[dr] & & C \ar[dl] \\
				& S \ar[ul, bend left, "p"] \ar[ur, bend right, "\sigma"']&
			\end{tikzcd}
		\end{equation*}
		In particular $\sigma^*\omega_{C/S}\simeq p^*\nu^*\omega_{C/S}\simeq p^*\omega_{\overline{C}/S}(2p)$, where $\omega_{\bullet/S}$ denotes the dualizing sheaf.
		As $\psi_1$ is the first Chern class of the line bundle on $\Ctilde_{2}$ defined as
		\[ (C\to S,\sigma)\longmapsto\sigma^*\omega_{C/S} \]
		we deduce that $f^*\psi_1$, when restricted to the open substack $\Mtilde_{1,1}\times[(\AA^1\smallsetminus\{0\})/\gm]\simeq\Mtilde_{1,1}$, is equal to the first Chern class of
		\[ (\overline{C}\to S,p) \longmapsto p^*\omega_{\overline{C}/S}(2p). \]
		This is equal to $T$, because $p^*\cO(-p)\simeq p^*\omega_{\overline{C}/S}$. We have shown in this way that $f^*\psi_1= T + nS$ for some integer $n$, as claimed.
	\end{proof}
	
	From \Cref{prop:cusp rel M2tilde} we deduce that the image of $f_*$ is generated as an ideal in $\ch(\Ctilde_2)$ by $f_*1=[\Ctilde_2^{\rm c}]$: this follows from a straightforward application of the projection formula. In this way we have proved \Cref{prop:1-cusp}.
	
	\begin{proposition}\label{prop:exceptional}
		Let $E$ be the exceptional divisor of ${\rm{Bl}}(\Ctilde_{1,1}\times [\AA^1/\gm])$ and let $\rho:E\to\Mtilde_{1,1}\times\cB\gm$. Then the image of
		\[ c_*'':\ch(E) \longrightarrow \ch(\Ctilde_2) \]
		is contained in the ideal $(J,[\Ctilde_2^{\rm c}],[\Ctilde_2^{\rm E}],c_*''(\rho^*T))$, where $T$ is the first Chern class of the dual of the Hodge line bundle on $\Mtilde_{1,1}\times\cB\gm$.
	\end{proposition}
	
	\begin{proof}
		By construction the exceptional divisor $E$ is the projectivization of the normal bundle of $\Mtilde_{1,1}\times\cB\gm$ in $\Ctilde_{1,1}\times [\AA^1/\gm]$, hence its Chow ring is generated as a $\ch(\Mtilde_{1,1}\times\cB\gm)$-module by the powers of the hyperplane class.
		
		There is a commutative diagram
		\begin{equation*}
			\begin{tikzcd}
				E \ar[r, "\cP"] \ar[dr, "\rho"] \ar[rr, bend left, "c''"] & c^*\Ctilde_2 \ar[r, "c'"] \ar[d, "\pi'"] & \Ctilde_2 \ar[d, "\pi"] \\
				& \Mtilde_{1,1}\times\cB\gm \ar[r, "c"]    & \Mtilde_2
			\end{tikzcd}
		\end{equation*}
		which is basically the same as the diagram (\ref{eq:fund diag Z1}) but restricted to $\Mtilde_{1,1}\times\cB\gm$. In particular, the square on the right is cartesian.
		
		The pullback of $\psi_1$ along $c''=c'\circ\cP$ is equal to $nh + m\rho^*\ell$, where $h$ is the hyperplane class and $\ell$ is the class of a divisor in $\Mtilde_{1,1}\times\cB\gm$. We claim that $n=1$. If so, we can conclude that $\im{c'_*\cP_*}$ is contained in the ideal $(J,[\Ctilde_2^{\rm E}],c''_*(\rho^*T))$, where $[\Ctilde_2^{\rm E}]=c''_*(1)$, as follows: first observe that
		\begin{align*}
			c''_*(h^i \rho^*\zeta) &= c''_*(({c''}^*\psi_1 - m\rho^*\ell)^i \rho^*\zeta) = \sum_{j=0}^{i} \psi_1^j \cdot c''_*\rho^*(\eta_j) 
		\end{align*}
		where $\eta_j$ is some class in $\ch(\Mtilde_{1,1}\times \cB\gm)$ for $j=0,\dots,n$. This shows that the image of $c''_*$ is generated as an ideal by elements of the form $c''_*\rho^*\eta$, where $\eta$ is in $\ch(\Mtilde_{1,1}\times \cB\gm)$.
		
		Second, observe that $\ch(\Mtilde_{1,1}\times \cB\gm)$ is generated as a $\ch(\Mtilde_2)$-module by $1$ and $T$: this follows from \Cref{prop:cusp rel M2tilde}. Therefore, we can write every element $\rho^*\eta$ as
		\[\rho^*\eta = (\rho^*T)\cdot (\rho^*c^*\zeta) + (\rho^*c^*\zeta') = (\rho^*T)\cdot ({c''}^*\pi^*\zeta) + ({c''}^*\pi^*\zeta'). \]
		Applying the projection formula, we get
		\[ c''_*(\rho^*\eta)=(c''_*(\rho^*T))\cdot \pi^*\zeta + c''_*1\cdot \pi^*\zeta', \]
		hence $\im{c''_*}$ is generated as an ideal by $c''_*1$ and $c''_*\rho^*T$, as claimed.
		
		To check that $n=1$, first observe that $\psi_1$ can also be regarded as the first Chern class of the relative dualizing sheaf $\omega_{\pi}$ of $\pi\colon\Ctilde_2\to\Mtilde_2$. Pulling back this line bundle to $E$ amounts to the following: first we restrict $\omega_{\pi}$ to $\Ctilde_2^{\rm E}$ inside $\Ctilde_2$, then we pull it back to the relative normalization of $\Ctilde_2^{\rm E}\to \pi(\Ctilde_2^{\rm E})$. Finally, we restrict to the irreducible component of the normalization that is equal to $E$.
		
		The usual formulas for the pullback of the relative dualizing sheaf along a normalization morphism tells us that the pullback of $\omega_{\pi}$ is equal to $\omega_{E/\Mtilde_{1,1}\times\cB\gm}(2\bfq + \bfr)$, where $\bfq$ is the preimage in $E$ of the cusp and $\bfr$ is the preimage of the separating node.
		
		As $\rho:E\to\Mtilde_{1,1}\times\cB\gm$ is a projective bundle with fibers of dimension one, the class of the relative dualizing sheaf is equal to $-2h+\rho^*\ell'$. The class of $\cO(2\bfq + \bfr)$ is equal to $3h + \rho^*\ell'$, hence the pullback of $\psi_1$ is equal to $h+\rho^*(\ell' + \ell'')$. This concludes the proof.
	\end{proof}
	Putting together \Cref{prop:1-cusp} and \Cref{prop:exceptional}, we deduce that the image of
	\begin{equation}\label{eq:push blow-up}
		\begin{tikzcd}
			\ch(\Ctilde_{1,1}\times \left[\AA^1/\gm\right])\oplus\ch(E) \ar[rrr, "c'_*\cP_*\rho^* + c'_*\cP_*i_*"] & & & \ch(\Ctilde_2).
		\end{tikzcd}
	\end{equation}
	is contained in the ideal $(J,[\Ctilde_2^{\rm c}],[\Ctilde_2^{\rm E}],c''_*\rho^*T)$. This coincides with the image of the map from the Chow ring of the blow-up of $\Ctilde_{1,1}\times \ag$, hence we have proved \Cref{prop:1-cusp whole}.
	
	\subsection{Abstract characterization of $\ch(\Mbar_{2,1})$}
	The map from the blow-up of $\Ctilde_{1,1}\times[\AA^1/\gm]$ to $\Ctilde_2$ is one-to-one onto the locus of curves with exactly one cusp.
	
	Therefore, the ideal in $\ch(\Ctilde_2)$ formed by the cycles coming from the locus of cuspidal curves is equal to the sum of the ideal of cycles coming from the locus of curves with two cusps plus the image of (\ref{eq:push blow-up}). By \Cref{prop:relations two cusps} and \Cref{prop:1-cusp whole}, both these ideals are contained in $(J,[\Ctilde_{2}^{\rm c}],[\Ctilde_2^{\rm E}],c''_*\rho^*T)$, thus the latter ideal coincides with the whole ideal of cycles coming from the cuspidal locus. 
	
	In this way we have proved the following abstract characterization of the integral Chow ring of $\Mbar_{2,1}$.
	\begin{theorem}\label{thm:chow Mbar21 abs}
		Suppose that the ground field has characteristic not $2$. Then
		\[ \ch(\Mbar_{2,1}) \simeq \ch(\Ctilde_2)/(J,[\Ctilde_2^{\rm c}],[\Ctilde_2^{\rm E}],c''_*\rho^*T), \]
		where $J$ is the pullback of the ideal of cuspidal relations in $\Mtilde_2$.
	\end{theorem}
	
	\section{The Chow ring of $\Mbar_{2,1}$, concrete computations}\label{sec:concrete}
	In this final Section we conclude the computation of $\ch(\Mbar_{2,1})$ (\Cref{thm:chow Mbar21}), by determining the last missing pieces: the fundamental class of $\Ctilde_2^{\rm c}$ (\Cref{prop:class Ctilde2 c}), the fundamental class of $\Ctilde_2^{\rm E}$ and $c''_*\rho^*T$ (\Cref{prop:class Ctilde2 E}).
	
	\subsection{Fundamental class of $\Ctilde_2^{\rm c}$ in $\ch(\Ctilde_2)$}
	Recall that $\Ctilde_2^{\rm c}$ is the cuspidal locus of $\Ctilde_2$: in other words, the closed substack $\Ctilde_2^{\rm c}$ is the stack of cuspidal, stable $A_2$-curves of genus two together with a section that lands in the cuspidal locus.
	\begin{remark}
		The stack $\Ctilde_2^{\rm c}$ is the normalization of $\Mtilde_2^{\rm c}$, the stack of cuspidal stable $A_2$-curves of genus two.
	\end{remark}
	We want to compute the fundamental class of $\Ctilde_2^{\rm c}$. First of all, we determine its restriction to the open stratum $\Ctilde_2 \setminus \ThTilde_1$.
	
	\begin{lemma}\label{lm:cusp in Ctilde2 minus ThTilde1}
		$$ [\Ctilde_2^{\rm c}]\vert_{\Ctilde_2 \setminus \ThTilde_1} = 2\lambda_1\psi_1(7\psi_1-\lambda_1) - 24 \psi_1^3 \in \ch(\Ctilde_2 \setminus \ThTilde_1).$$
	\end{lemma}
	
	\begin{proof}
		Recall that $\Ctilde_2 \setminus \ThTilde_1$ is an open substack of $[\widetilde{\AA}(6)/\rm B_2]$: the points in $\widetilde{\AA}(6)$ can be though as pairs $(f,s)$ where $f$ is a binary form of degree $6$ and $s$ is a scalar such that $f(0,1)=s^2$. The action of $\rm B_2$ is described in Section \ref{c2-d1} .
		
		Consider the invariant closed subscheme $Z \subset \widetilde{\AA}(6)$ parametrizing those pairs $(f,s)$ where $(0,1)$ is a root of multiplicity at least three. The restriction of $[Z/\rm B_2]$ to the appropriate open substack of $[\widetilde{\AA}(6)/\rm B_2]$ coincides with $\Ctilde_2^{\rm c}$, thus all we have to do is to compute the $\rm B_2$-equivariant fundamental class of $Z$. Actually, if $\gm^2\subset \rm B_2$ is the maximal subtorus of diagonal matrices, we can equivalently compute the $\gm^2$-equivariant class of $Z$, because the pullback along $[\widetilde{\AA}(6)/\gm^2]\arr[\widetilde{\AA}(6)/\rm B_2]$ induces an isomorphism of Chow rings (\cite{Per}*{Remark 3.1}).
		
		If we write a form $f$ as $f=a_0x_0^6+a_1x_0^5x_1...+a_6x_1^6$, where $s^2=a_6$, we get that $(f,s)$ belongs to $Z$ if and only if $s=a_4=a_5=0$. In other words, the subscheme $Z$ is a complete intersection of the three divisors $\{s=0\}$, $\{a_4=0\}$ and $\{ a_5=0\}$, hence its fundamental class is the product of the fundamental classes of those three divisors. 
		
		Applying the formula for the equivariant fundamental classes of divisors (see \cite{EF}*{Lemma 2.4}) we obtain
		$$ [Z]=-2\psi_1(\lambda_1-3\psi_1)(\lambda_1-4\psi_1),$$
		and this concludes the proof.
	\end{proof}
	
	Consider now the restriction of $\Ctilde^{\rm c}_2$ to $\Ctilde_2 \setminus \ThTilde_2$. Clearly, we have that 
	$$ [\Ctilde_2^{\rm c}]\vert_{\Ctilde_2 \setminus \ThTilde_2}=2\lambda_1\psi_1(7\psi_1-\lambda_1) - 24 \psi_1^3 + \vartheta_1p_2(\lambda_1,\psi_1) $$ 
	where $\vartheta_1$ is the fundamental class of $\ThTilde_1$ restricted to $\Ctilde_2 \setminus \ThTilde_2$.
	and $p_2(\lambda_1,\psi_1)$ is a homogeneous polynomial of degree 2 (notice that $p_2$ does not depend on $\vartheta_1$ because of the relation $\vartheta_1(\vartheta_1+\lambda_1)$ in $\ch(\Ctilde_2 \setminus \ThTilde_2)$, see \Cref{prop:chow C minus ThTilde2}).
	
	\begin{lemma}\label{lm:cusp in Ctilde2 minus ThTilde2}
		The intersection $\Ctilde_2^{\rm c} \cap (\ThTilde_1 \setminus \ThTilde_2)$ is transversal and we have the following:
		$$[\Ctilde_2^{\rm c} \cap (\ThTilde_1 \setminus \ThTilde_2)] = -24 \psi_1^3 \in \ch(\ThTilde_1 \setminus \ThTilde_2). $$
		Furthermore, we get 
		$$[\Ctilde_2^{\rm c}]\vert_{\Ctilde_2 \setminus \ThTilde_2}=2\psi_1(\lambda_1+\vartheta_1)(7\psi_1-\lambda_1) - 24 \psi_1^3$$ 
		where the equality holds in $\ch(\Ctilde_2 \setminus \ThTilde_2)$.
	\end{lemma}
	
	\begin{proof}
		Let $\Ctilde_{1,1}$ be the universal elliptic stable $A_2$-curve, and let $\bfp\colon\Mtilde_{1,1}\arr\Ctilde_{1,1}$ be the universal section.
		Recall from \Cref{prop:chow ThTilde1 minus ThTilde2} that $\ThTilde_1\smallsetminus\ThTilde_2$ is isomorphic to the product $(\Ctilde_{1,1}\setminus \im{\bfp}) \times \Mtilde_{1,1}$, where the isomorphism is given by taking an elliptic stable $A_2$-curve with a second marking and gluing it to another elliptic stable $A_2$-curve along the first marking.
		
		The fibered product of stacks $\Ctilde_2^{\rm c} \times_{\Ctilde_2} (\ThTilde_1 \setminus \ThTilde_2)$ is the substack of curves with a separating node and a marked cusp, hence
		$$ (\Ctilde_2^{\rm c} \times_{\Ctilde_2} (\ThTilde_1 \setminus \ThTilde_2))_{\rm red} \simeq \cB\gm \times \Mtilde_{1,1}.$$
		
		From this we deduce that the dimension of the intersection is the expected one and the intersection is proper. Given a geometric point in $\Ctilde_2^{\rm c} \times_{\Ctilde_2} (\ThTilde_1 \setminus \ThTilde_2$), it is enough to prove the transversality of the intersection in a versal neighborhood of the point. Because both $\Ctilde_2^{\rm c}$ and $\ThTilde_1 \setminus \ThTilde_2$ are smooth stacks, we just need to prove that the natural morphism of vector spaces
		\begin{equation}\label{eq:trans} T_p \Ctilde_2^{\rm c} \oplus T_p (\ThTilde_1 \setminus \ThTilde_2) \longrightarrow T_p \Ctilde_2 \end{equation}
		is surjective for every geometric point $p$ of $\Ctilde_2^{\rm c} \times_{\Ctilde_2}(\ThTilde_1 \setminus \ThTilde_2)$ (here $T_p(-)$ denotes the tangent space in the versal neighborhood).
		
		Let us recall some information about the tangent space of $\Ctilde_2$ at a geometric point $(C,\sigma) \in \Ctilde_2(k)$. In a versal neighborhood of $(C,\sigma)$ we have:
		\begin{equation}\label{eq:formula tangent}
			T_{(C,\sigma)}\Ctilde_2\simeq\im{d\pi}\oplus(\frkm_{\sigma}/\frkm_{\sigma}^2)^\vee, 
		\end{equation}
		where $\pi\colon\Ctilde_2 \rightarrow \Mtilde_2$ is the forgetful morphism and $d\pi:T_{(C,\sigma)}\Ctilde_2 \rightarrow T_C\Mtilde_2$ is the induced morphism of tangent spaces (in a versal neighborhood).
		
		Recall that if $\sigma$ is a smooth point of $C$, then $d\pi$ is surjective and $\dim (\frkm_{\sigma}/\frkm_{\sigma}^2)=1$. On the other hand, if $\sigma$ is a node or a cusp we have that $\dim (T_{C}\Mtilde_2/\im{d\pi}) = 1$ and $\dim (\frkm_{\sigma}/\frkm_{\sigma}^2)=2$ (the formula (\ref{eq:formula tangent}) actually holds true for every double point singularity).
		
		Let $p=(C,\sigma)$ be a geometric point in the intersection. The vector space $T_p\Ctilde_2^{\rm c}$ parametrizes first order deformations of $(C,\sigma)$ that preserve both the cusp and the section: its image in $T_p \Ctilde_2\simeq \im{d\pi}\oplus(\frkm_{\sigma}/\frkm_{\sigma}^2)^\vee $ can be regarded as a the subspace of $\im{d\pi}$ classifying first order deformations of $C$ that preserve the cusp, i.e. the tangent space $T_C\Mtilde_2^{\rm c}$.
		We also have the image of $T_{(C,\sigma)}(\ThTilde_1 \setminus \ThTilde_2)$ in $T_p\Ctilde_2$ can be identified with $(\frkm_{\sigma}/\frkm_{\sigma}^2)^\vee \oplus (T_C\Dtilde_1\cap \im{d\pi})$. From this we deduce that (\ref{eq:trans}) is surjective, because the sum of the vector subspaces $T_{C}\Mtilde_2^{\rm c}$  (deformations preserving the cusp) and $T_C (\Dtilde_1\cap\im{d\pi})$ (deformations preserving the node) generates $\im{d\pi}$.
		
		We have proved that 
		$$\Ctilde_2^{\rm c} \times_{\Ctilde_2} (\ThTilde_1 \setminus \ThTilde_2) \simeq \cB\gm \times \Mtilde_{1,1},$$
		where the latter can be regarded as a closed substack of $ (\Ctilde_{1,1}\setminus \im{\bfp}) \times \Mtilde_{1,1}$.
		To complete the statement, we only need the compute the class $[\cB\gm]$ inside the Chow ring of $\Ctilde_{1,1} \setminus \im{\bfp}$.
		
		Recall that $\Ctilde_{1,1} \setminus \im{\bfp} \simeq [ W/\gm] $, where
		$$W=\{((\alpha,\beta),(x,y))  \in \AA^2 \times \AA^2 | y^2=x^3+\alpha x+ \beta\}$$
		is the universal affine Weierstrass curve and the $\gm$-action is described by the formula $$t\cdot(\alpha,\beta,x,y)=(t^{-4}\alpha,t^{-6}\beta,t^{-2}x,t^{-3}y).$$ The inclusion $\cB\gm \subset \Ctilde_{1,1} \setminus \im{\bfp} $ corresponds to the closed embedding 
		$$ [ (0,0,0,0)/\gm] \hooklongrightarrow [ W/\gm]$$
		through the identification above. 
		The $\gm$-equivariant class of the origin in $\ch_{\gm}(W)$ can be computed as follows: if $\pr_1,\pr_2:W\arr\AA^2$ are the projections respectively on the first and on the second factor of $\AA^2\times\AA^2$, then we have
		\[ (0,0,0,0) = \pr_1^{-1}\{\alpha=0\}\cap \pr_2^{-1}\{x=y=0\}. \]
		The equivariant class of the subrepresentation $\{\alpha=0\}\subset \AA^2$ is $-4\psi_!$, because $\gm$ acts on $\alpha$ by multiplication by $t^{-4}$. With the same argument we deduce that the class of $\{x=0\}$ is $-2\psi_1$ and the class of $\{y=0\}$ is $-3\psi_1$. Putting all together, we conclude that $[\cB\gm]=-24\psi_1^3$. 
		
		Finally, write $$[\Ctilde_2^{\rm c}]\vert_{\Ctilde_2\setminus \ThTilde_2}= p_3(\lambda_1,\psi_1)+\vartheta_1p_2(\lambda_1,\psi_1),$$ where $p_3$ and $p_2$ are homogeneous polynomials of degree $3$ and $2$ respectively. \Cref{lm:cusp in Ctilde2 minus ThTilde1} gives us a formula for the restriction of the cycle above to $\Ctilde_2 \setminus \ThTilde_1$, which implies that $p_3=2\lambda_1\psi_1(7\psi_1-\lambda_1) - 24 \psi_1^3$. 
		On the other hand, we have just computed the restriction of $[\Ctilde_2^{\rm c}]$ to $\ch(\ThTilde_1\smallsetminus\ThTilde_2)$, which is $-24\psi_1^3$. From \Cref{prop:chow C minus ThTilde2} we know that the restriction map sends
		\[ \lambda_i\longmapsto \lambda_i,\quad \psi_1\longmapsto\psi_1,\quad \vartheta_1\longmapsto -\lambda_1.  \]
		Putting these information together, we get $p_2=2\psi_1(7\psi_1-\lambda_1)$, and this finishes the computation.
	\end{proof}
	
	Finally we get the description of the fundamental class of $\Ctilde_2^{\rm c}$ inside the Chow ring of $\Ctilde_2$.
	
	\begin{proposition}\label{prop:class Ctilde2 c}
		$$ [\Ctilde_2^{\rm c}] = 2\psi_1(\lambda_1 +\vartheta_1)(7\psi_1-\lambda_1) - 24\psi_1^3 \in \ch(\Ctilde_2). $$
	\end{proposition}
	
	\begin{proof}
		First observe that the intersection of $\Ctilde_2^{\rm c}$ with $\ThTilde_2$ is empty, hence the restriction of the fundamental class of the first stack to the second is zero.
		
		\Cref{lm:cusp in Ctilde2 minus ThTilde2} gives an explicit expression for the restriction of $[\Ctilde_2^{\rm c}]$ to $\Ctilde_2\smallsetminus\ThTilde_2$, from which we deduce that
		$$ [\Ctilde_2^{\rm c}]= 2\psi_1(\lambda_1+\vartheta_1)(7\psi_1-\lambda_1)- 24\psi_1^3 + \vartheta_2 p_1(\lambda_1,\psi_1) \in \ch(\Ctilde_2),$$
		where $p_1$ is a homogeneous polynomial of degree $1$. Notice that $p_1$ does not depend on $\vartheta_1$ because of the relation $\vartheta_2(\vartheta_1 -\lambda_1 + \psi_1)$ inside $\ch(\Ctilde_2)$ (see \Cref{prop:chow Ctilde2}). Pulling everything back to $\ThTilde_2$ (see formulas \eqref{eq:pullback theta2} and \eqref{eq:pullback psi1}) and using the explicit presentation of $\ch(\ThTilde_2)$ given in \Cref{prop:chow ThTilde2}, we get the equation 
		$$ 0=[\Ctilde_2^{\rm c}]\vert_{\ThTilde_2}=\lambda_2 p_1(\lambda_1,\xi_1) \in \ch(\ThTilde_2),$$
		which implies $p_1=0$ and concludes the proof.
	\end{proof}
	
	\subsection{Relations coming from $\Ctilde_2^{\rm E}$}
	Recall that $\Ctilde_2^{\rm E}$ is the closed substack of $\Ctilde_2$ whose points are families of $1$-pointed, almost stable, genus two curves $(C,\sigma)$ with a separating node satisfying the following property: at least one irreducible component of $C$ is cuspidal and the image of $\sigma$ belongs to a component with a cusp.
	
	Let us recall the following diagram from the proof of \Cref{prop:exceptional}: let $\rm{Bl}(\Ctilde_{1,1}\times [\AA^1/\gm])$ be the blow-up of $\Ctilde_{1,1}\times [\AA^1/\gm]$ along $\im{\bfp\times 0}\simeq \Mtilde_{1,1}\times\cB\gm$, and let $E$ be the exceptional divisor. We can form a diagram
	\[
	\begin{tikzcd}
	    E\simeq\PP(N_{\bfp\times 0}) \ar[r, "c''"] \ar[d, "\rho"] & \Ctilde_2 \ar[d, "\pi"] \\
	    \Mtilde_{1,1}\times\cB\gm \ar[r, "c"] & \Mtilde_2.
	\end{tikzcd}
	\]
	We can identify the Chow ring of $\Mtilde_{1,1}\times\cB\gm$ with $\ZZ[T,S]$, where $T$ is the first Chern class of the dual of the Hodge line bundle.
	The aim of this Subsection is to compute $c''_*[\Ctilde_2^{\rm E}]$ and $c''_*(\rho^*T)$.
	
	\begin{remark}
		Let $\pi\colon\Ctilde_2\arr\Mtilde_2$ be the forgetful morphism: the preimage $\pi^{-1}(\Mtilde_2^{\rm c})$ is the substack of $1$-pointed stable $A_2$-curves of genus two with a cusp. This stack is not irreducible, and one irreducible component is given by $\Ctilde_2^{\rm E}$.
	\end{remark}
	It is clear that $\Ctilde_2^{\rm E}$ has codimension 3 inside $\Ctilde_2$, therefore we have the following description 
	\begin{align}\label{eq:class of C2tildeE} [\Ctilde_2^{\rm E}] &= p_3(\lambda_1,\psi_1) + \vartheta_1p_2(\lambda_1,\psi_1) + \vartheta_2 p_1(\lambda_1,\psi_1) \in \ch (\Ctilde_2)\\
	c''_*(\rho^*T)&= q_4(\lambda_1,\psi_1) + \vartheta_1q_3(\lambda_1,\psi_1) + \vartheta_2 q_2(\lambda_1,\psi_1) \in \ch (\Ctilde_2) \nonumber\end{align} 
	where the $p_i$ and $q_i$ are homogeneous polynomials of degree $i$.
	As the intersection $(\Ctilde_2\setminus \ThTilde_1) \cap \Ctilde_2^{\rm E} $ is empty, we get $p_3=q_4=0$.
	
	We will adopt the following notation: if $\cV\subset \cX \subset \cY$ are all closed embeddings of quotient stacks, we denote $[\cV]_{\cX}$ (respectively $[\cV]_{\cY}$) the fundamental class of $\cV$ in $\ch(\cX)$ (respectively in $\ch(\cY)$).
	\begin{lemma}\label{lm:C2tildeE in the open}
		\begin{align*}
		[\Ctilde_2^{\rm E}]|_{\Ctilde_2\smallsetminus\ThTilde_2} &=24\vartheta_1\psi_1^2 \in \ch(\Ctilde_2 \setminus \ThTilde_2), \\
		(c''_*\rho^*T)|_{\Ctilde_2\smallsetminus\ThTilde_2} &= 24\vartheta_1\psi_1^2(\lambda_1-\psi_1) \in \ch(\Ctilde_2 \setminus \ThTilde_2).
		\end{align*} 
	\end{lemma}
	
	\begin{proof}
		From (\ref{eq:class of C2tildeE}) we know that
		\[[\Ctilde_2^{\rm E}]|_{\Ctilde_2\smallsetminus\ThTilde_2}=\vartheta_1p_2(\lambda_1,\psi_1), \quad (c''_*\rho^*T)|_{\Ctilde_2\smallsetminus\ThTilde_2} = \vartheta_1q_3(\lambda_1,\psi_1)\]
		We need to determine $p_2$ and $q_3$.
		
		Using the description $\ThTilde_1 \setminus \ThTilde_2$ as the product $(\Ctilde_{1,1}\setminus \im{\bfp}) \times \Mtilde_{1,1}$, we get that $\Ctilde_2^{\rm E} \simeq \cV \times \Mtilde_{1,1}$ where $\cV \subset (\Ctilde_{1,1}\setminus \im{\bfp})$ is the closed substack classifying cuspidal curves. 
		
		Recall that $\Ctilde_{1,1}\setminus \im{\bfp}\simeq [W/\gm]$, where
		\[ W=\{(\alpha,\beta,x,y)\text{ }|\text{ }y^2=x^3+\alpha x + \beta \} \subset V_{-4,-6}\times V_{-2,-3}. \]
		As before, we use the notation $V_{i,j}$ to indicate the rank two $\gm$-representation of weight $i$ and $j$.
		
		Hence we have that $\cV\simeq [V/\gm]$ where $V\subset W$ is the $\gm$-invariant closed subscheme defined by 
		$$V=\{ (\alpha,\beta,x,y) \in W\text{ }\vert\text{ }\alpha=0,\beta=0\}$$
		which is a complete intersection inside $W$. As $\gm$ acts on $\alpha$ with weight $-4$ and on $\beta$ with weight $-6$, we deduce that the equivariant class of $V$ is $24\psi_1^2$.
		
		The codimension $2$ closed embedding 
		\[ e:\Ctilde_2^{\rm E} \setminus \ThTilde_2 \hooklongrightarrow \ThTilde_1 \setminus \ThTilde_2\]
		is regular, thus the excess intersection formula gives us the following description:
		$$j^*[\Ctilde_2^{\rm E}]_{\Ctilde_2 \setminus \ThTilde_2} = c_1(N_j) \cdot [\Ctilde_2^{\rm E}]_{\ThTilde_1 \setminus \ThTilde_2}= c_1(N_j)\cdot 24\psi_1^2$$
		where $j$ is the regular closed embedding $\ThTilde_1 \setminus \ThTilde_2 \hookrightarrow \Ctilde_2 \setminus \ThTilde_2$ and $N_j$ is the associated normal bundle.
		
		Combining this with (\ref{eq:class of C2tildeE}) and the fact that $c_1(N_j)$ is not a zero divisor (see \Cref{prop:chow C minus ThTilde2}), we get that $p_2(\lambda_1,\psi_1) = 24\psi_1^2$.
		Observe that $c'':E\to\Ctilde_2$, once restricted over $\Ctilde_2\smallsetminus\ThTilde_2$, becomes a closed embedding, and we have a commutative diagram 
		\[
		\begin{tikzcd}
		    E|_{\Ctilde_2\smallsetminus\ThTilde_2} \ar[r, "e"] \ar[rr, bend left, "c''"] \ar[d, "\rho"] & \ThTilde_1\smallsetminus\ThTilde_2 \simeq \Mtilde_{1,1}\times(\Ctilde_{1,1}\setminus \im{\bfp})\ar[r, "j"] \ar[d, "\pr_2"] & \Ctilde_2\smallsetminus\ThTilde_2 \\
		    \Mtilde_{1,1}\times \cB\gm \ar[r, "\pr_1"] & \Mtilde_{1,1}. & 
		\end{tikzcd}
		\]
		In particular, we have that $T\in \ch(\Mtilde_{1,1}\times\cB\gm)$ is equal to $\pr_1^*(T)$, hence we have
		\[ c''_*\rho^*T=c''_*\rho^*\pr_1^*T=c''_*e^*\pr_2^*T=j_*(e_*1 \cdot \pr_2^*T) \]
		and therefore $j^*c''_*(\rho^*T) = (e_*1\cdot\pr_2^*T)\cdot c_1(N_j)$. We have already computed that $e_*1=24\psi_1^2$. In the proof of \Cref{prop:chow ThTilde1 minus ThTilde2} we showed that $\pr_2^*T=s=\lambda_1-\psi_1$, from which we conclude that $j^*c''_*(\rho^*T)= c_1(N_j)24\psi_1^2(\lambda_1-\psi_1)$. This immediately implies our conclusion.
	\end{proof}
	
	The restriction of $\Ctilde_2^{\rm E}$ to $\ThTilde_2$ requires to be handled with care.
	Consider the cartesian diagram
	$$\begin{tikzcd}
		\Ctilde_2^{\rm E} \cap \ThTilde_2 \arrow[rr, "i'", hook] \arrow[dd, "\codim{2}", hook] &  & \Ctilde_2^{\rm E} \arrow[dd, "\codim{3}", hook] \\
		&  &                                   \\
		\ThTilde_2 \arrow[rr, "\codim{2}", hook]                                &  & \Ctilde_2                        
	\end{tikzcd}$$
	where $i'$ is a $\codim{1}$ closed embedding. It can be easily showed that $i'$ is not a regular embedding. In fact, it is a regular embedding away from the locus parametrizing curves with two cusps.
	
	The substack of cuspidal curves in $\Ctilde_2$ is the image of the pinching morphism (see Subsection \ref{sub:pinching}) $$\cP:{\rm Bl}_{\im{\bfp}\times\cB\gm}(\Ctilde_{1,1}\times [\AA^1/\gm]) \longrightarrow \Ctilde_2$$
	where $\bfp: \Mtilde_{1,1} \rightarrow \Ctilde_{1,1}$ is the universal section and $\cB\gm=[\{0\}/\gm]$. 
	With this picture in mind, $\Ctilde_2^{\rm E}$ can be regarded as the image of the exceptional divisor of the blow-up, i.e. the image of $\PP(N_{\bfp \times 0})$ via the pinching morphism.
	
	The proper transform of $\im{\bfp} \times [\AA^1/\gm]$ in the blow-up induces a section 
	$$ i'': \Mtilde_{1,1} \times \cB\gm \hooklongrightarrow \PP(N_{\bfp \times 0})$$
	of the morphism $\PP(N_{\bfp \times 0}) \rightarrow \Mtilde_{1,1}\times \cB\gm$.
	
	\begin{lemma}\label{lm:formula for C2tildeE in ThTilde2}
		The following commutative diagram
		\begin{equation}\label{eq:excess}\begin{tikzcd}
				{\Mtilde_{1,1}\times \cB\gm} \arrow[d, "i''", hook] \arrow[r, "r"] &  \ThTilde_2 \arrow[d, "i", hook] \\
				\PP(N_{\bfp \times 0}) \arrow[r, "c''"] & \Ctilde_2                               
		\end{tikzcd}\end{equation}
		is cartesian and 
		\begin{align*}
		    i^*([\Ctilde_2^{\rm E}]) &= r_*(r^*c_1(N_{i})-c_1(N_{i''})), \\
		    i^*(c''_*(\rho^*T)) &= r_*(T\cdot (r^*c_1(N_{i})-c_1(N_{i''}))).
		\end{align*}
		where $N_i$ (respectively $N_{i''}$) is the normal bundle associated with the closed regular embedding $i$ (respectively $i''$).
	\end{lemma}
	\begin{proof}
		The cartesianity is clear from the construction of the pinching morphism (see Subsection \ref{sub:pinching}).
		The compatibility of the Gysin homomorphism with the pushforward tells us that for any cycle $\zeta$ we have
		\[ i^*c''_*\zeta = r_*(i^{!}\zeta). \]
		We can apply the excess intersection formula, obtaining
		\[ i^{!}\zeta = c_1(r^*N_{i}/N_{i''}) \cdot {i''}^*\zeta.  \]
		Applying these formulas to $\zeta=1$, $\rho^*T$ together with the fact that ${i''}^*\rho^*={\rm id}$, we obtain the desired conclusion.
    \end{proof}
	
	We are almost ready to compute explicitly the pullback of $[\Ctilde^{\rm E}_2]$ and $c''_*\rho^*T$.
	To proceed, we need the following technical result.
	\begin{lemma}\label{lm:blowup formula}
		Suppose $X,Y,Z$ are three smooth algebraic stacks with closed embeddings $i:X \hookrightarrow Y$ and $j: Y \hookrightarrow Z$ such that $\dim Z= \dim Y +1= \dim X +2$. Let $Z':= {\rm Bl}_{X}Z$ be the blow-up of $Z$ along $X$. Then we have a section  $\sigma:X \hookrightarrow E$ of the natural morphism $E\rightarrow X$ where $E$ is the exceptional divisor, and the following formula holds:
		$$ N_{X\vert E} = i^*N_{Y\vert Z}\otimes N_{X\vert Y}^\vee. $$
	\end{lemma}
	\begin{proof}
	    By \cite[B.6.10]{Ful} there is a lifting $j':Y \hookrightarrow Z'$ of $j$, i.e. we have a diagram
		$$\begin{tikzcd}
			Y \arrow[rr, "j'", hook] \arrow[rd, "j", hook] &   & Z' \arrow[ld, "\pi"] \\
			& Z &                     
		\end{tikzcd}$$ 
		where both $j'$ and $j$ are regular closed embedding of codimension $1$, and $ \pi^*(\cO_Z(Y))= \cO_{Z'}(Y) \otimes \cO_{Z'}(E) $. 
		
		If we consider the two cartesian diagrams
		$$\begin{tikzcd}
			X=X\times_Z Y \arrow[r, "\sigma", hook] \arrow[d, "i", hook] & E \arrow[r] \arrow[d, hook] & X \arrow[d, "j\circ i", hook] \\
			Y \arrow[r, "j'", hook]                                      & Z' \arrow[r,"\pi"]                  & Z                            
		\end{tikzcd}$$
		we get 
		$$ N_{X\vert E}= i^*N_{Y\vert Z'}= i^*j'^*\cO_{Z'}(Y)= i^*j'^*\pi^*\cO_Z(Y) \otimes i^*j'^*\cO_{Z'}(E)^\vee=i^*N_{Y\vert Z} \otimes N_{X\vert Y}^\vee$$
		which concludes the proof.
	\end{proof}
	
	\begin{lemma}\label{lm:C2tilde in the closed}
		Let $i\colon\ThTilde_2\hookrightarrow \Ctilde_2$ be the closed embedding. Then
		\begin{align*}
		    i^*[\Ctilde_2^{\rm E}] &= -24\lambda_1\lambda_2, \\
		    i^*(c''_*\rho^*T) &= 48\lambda_2^2.
		\end{align*} 
	\end{lemma}
	\begin{proof}
		We want to apply the formula given by \Cref{lm:formula for C2tildeE in ThTilde2}. For this, we first need to compute $c_1(N_{i''})$ in the Chow ring of $\Mtilde_{1,1}\times\cB\gm$ (same notation of \Cref{lm:formula for C2tildeE in ThTilde2}). We write
		\[ \ch(\Mtilde_{1,1}\times\cB\gm) \simeq \ZZ[T,S] \]
		where $T$ is the dual of the Hodge line bundle on $\Mtilde_{1,1}$ and $S$ is the universal line bundle of weight $1$.
		
		To compute $c_1(N_{i''})$ we use \Cref{lm:blowup formula}, where the role of $X$ is played by $\Mtilde_{1,1}\times\cB\gm$, the role of $Y$ by $\Ctilde_{1,1}\times\cB\gm$ and the role of $Z$ by $\Ctilde_2\times[\AA^1/\gm]$.
		
		The normal bundle of $\Ctilde_{1,1}\times\cB\gm$ in $\Ctilde_2\times[\AA^1/\gm]$ is the pullback of the normal bundle of $\cB\gm$ in $[\AA^1/\gm]$, whose class is $S$. The conormal bundle of $\Mtilde_{1,1}\times\cB\gm$ in $\Ctilde_{1,1}\times\cB\gm$ is the pullback of the conormal bundle of $\Mtilde_{1,1}$ embedded in $\Ctilde_{1,1}$ via the universal section: this is well known to be isomorphic to the Hodge line bundle, hence it is equal to $-T$.
		
		This said, applying \Cref{lm:blowup formula}, we get
		\[ c_1(N_{i''}) = S - T. \]
		We computed the normal bundle of $\ThTilde_2\hookrightarrow\Ctilde_2$ in the proof of \Cref{prop:chow Ctilde2}: this coincided with the Hodge line bundle twisted by the $2$-torsion line bundle on $\ThTilde_2$, hence its first Chern class is equal to $\xi_1-\lambda_1$. 
		
		Observe that the pullback of $\lambda_1$ along the two maps above is equal to $-S$ (see \Cref{prop:cusp rel M2tilde}): we deduce that $c_1(r^*N_{i})=g^*(\xi_1-\lambda_1)=S$, hence
		\[c_1(r^*N_{i})-c_1(N_{i''})=T. \]
		We can apply \Cref{lm:formula for C2tildeE in ThTilde2}, which tells us
		\[ i^*[\Ctilde^{\rm E}_2] = r_*(T),\quad i^*(c''_*\rho^*T) = r_*(T^2). \]
		
		We need to compute the pushforward of this element to $\ThTilde_2$, which we identify with $\Dtilde_1$ just as we have done in the proof of \Cref{prop:chow ThTilde2}. Observe that there is a natural factorization
		\[ \Mtilde_{1,1}\times\cB\gm \overset{g}{\longrightarrow} \Mtilde_{1,1}\times\Mtilde_{1,1} \overset{f}{\longrightarrow} (\Mtilde_{1,1}\times\Mtilde_{1,1})/\bC_2. \]
		Set
		\[ \ch(\Mtilde_{1,1}\times\Mtilde_{1,1}) = \ZZ[T,U] \]
		where $U$ is the pullback of the dual of the Hodge line bundle from the second copy of $\Mtilde_{1,1}$.
		To compute the pushforward of $T$, first observe that $g^*T=T$, hence $g_*(T)=T\cdot g_*(1)$. The element $g_*(1)$ is the pullback of the fundamental class of $\cB\gm$ in the second copy of $\Mtilde_{1,1}$, which is equal to $24U^2$. Similarly, we have $g_*(T^2)=24U^2T^2$
		
		Hence we are reduced to compute $f_*(24U^2T)$ and $f_*(24(UT)^2)$. This is equal to the pushforward of these cycles through the map
		\[ \cB(\gm\times\gm) \longrightarrow \cB(\gm^{\times 2}\rtimes \bC_2), \]
		which has been explicitly determined in \cite{Lars}*{Lemma 7.3} and \cite{DLV}*{Corollary 3.2}. We get
		\[ f_*(24U^2T) = -24\lambda_2\lambda_1, \quad f_*(24(UT)^2)=48\lambda_2^2, \]
		which concludes the proof.
	\end{proof}
	
	\begin{proposition}\label{prop:class Ctilde2 E}
		\begin{align*}
		    [\Ctilde_2^{\rm E}]&= 24(\psi_1^2\vartheta_1-\lambda_1\vartheta_2) \in \ch (\Ctilde_2),\\
		    c''_*\rho^*T &= 24(\psi_1^2(\lambda_1-\psi_1)\vartheta_1 + 2\lambda_2\vartheta_2) \in \ch (\Ctilde_2).\\
		\end{align*}
	\end{proposition}
	\begin{proof}
		We know from \Cref{lm:C2tildeE in the open} that $[\Ctilde_2^{\rm E}] = 24\vartheta_1\psi_1^2 + \vartheta_2 p_1$. This, combined together with \Cref{lm:C2tilde in the closed}, gives us
		\[ -24\lambda_1\lambda_2 = i^*[\Ctilde_2^{\rm E}] = 24 (\xi_1-\lambda_1)\xi_1^2 + \lambda_2p_1 \]
		from which we deduce that $p_1=-24\lambda_1$.
		
		From \Cref{lm:C2tildeE in the open} we also know that $c''_*\rho^*T = 24\vartheta_1\psi_1^2(\lambda_1-\psi_1) + \vartheta_2 q_2$. This, combined with \Cref{lm:C2tilde in the closed}, gives us 
		\[ 48\lambda_2^2 = i^*(c''_*\rho^*T) = \lambda_2q_2 \]
		from which we deduce $q_2=48\lambda_2$. This concludes the proof.
	\end{proof}

	\subsection{Final results}
	From \Cref{thm:chow Mbar21 abs} we know that
	\[\ch(\Mbar_{2,1})\simeq\ch(\Ctilde_2)/(J,[\Ctilde_{2}^{\rm c}],[\Ctilde_2^{\rm E}],c''_*\rho^*T) \]
	where $J$ is the ideal generated by the pullback of the relations in $\ch(\Mbar_2)$. From \Cref{prop:chow Ctilde2} we know an explicit presentation of $\ch(\Ctilde_2)$ in terms of generators and relations. The fundamental classes of $\Ctilde_2^{\rm c}$ and $\Ctilde_2^{\rm E}$ together with the cycle $c''_*\rho^*T$ have been computed respectively in \Cref{prop:class Ctilde2 c} and \Cref{prop:class Ctilde2 E}. Putting all together, we derive the following presentation of the integral Chow ring of $\Mbar_{2,1}$ in terms of generators and relations.
	\begin{theorem}\label{thm:chow Mbar21}
		Suppose that the characteristic of the base field is not $2$ or $3$. Then we have
		\[ \ch(\Mbar_{2,1})\simeq \ZZ[\lambda_1,\psi_1,\vartheta_1,\lambda_2, \vartheta_2]/(\alpha_{2,1},\alpha_{2,2},\alpha_{2,3},\beta_{3,1},\beta_{3,2},\beta_{3,3},\beta_{3,4}) \]
		where the $\alpha_{2,i}$ have degree $2$, the $\beta_{3,j}$ have degree $3$ and their explicit expressions are
		\begin{align*}
			\alpha_{2,1}&=\lambda_2-\vartheta_2-\psi_1(\lambda_1-\psi_1),\\
			\alpha_{2,2}&=24\lambda_1^2-48\lambda_2,\\
			\alpha_{2,3}&=\vartheta_1(\lambda_1+\vartheta_1),\\
			\beta_{3,1}&=20\lambda_1\lambda_2-4\lambda_2\vartheta_1,\\
			\beta_{3,2}&=2\psi_1\vartheta_2,\\
			\beta_{3,3}&=\vartheta_2(\vartheta_1+\lambda_1-\psi_1),\\
			\beta_{3,4}&=2\psi_1(\lambda_1 +\vartheta_1)(7\psi_1-\lambda_1) - 24\psi_1^3.\\
		\end{align*}
	\end{theorem}
	\begin{proof}
		The generating relations in the Theorem are obtained from the relations in the Chow ring of $\Ctilde_2$ together with the relations in the Chow ring of $\Mbar_2$, the fundamental classes of $\Ctilde_2^{\rm c}$ and $\Ctilde_2^{\rm E}$ and the cycle $c''_*\rho^*T$: a straightforward computation with Macaulay2 shows that all the relations of degree four as well as $[\Ctilde_2^{\rm E}]$ are redundant, thus giving us the list above.
	\end{proof}
	We could have avoided to include either $\lambda_2$ or $\vartheta_2$ among the generators of the ring, but keeping them both allowed us to simplify the explicit expressions of the relations.
	
	As a Corollary of the Theorem above we get the following description of the rational Chow ring of the coarse moduli space $\overline{M}_{2,1}$, whose computation over $\mathbb{C}$ has been done by Faber in his thesis \cite{Fab}.
	\begin{corollary}\label{cor:rational chow}
		Suppose that the characteristic of the base field is not $2$ or $3$. Then we have
		\[ \ch(\overline{M}_{2,1})_{\mathbb{Q}}\simeq \mathbb{Q}[\lambda_1,\psi_1,\vartheta_1]/(\alpha_{2,3},\beta'_{3,1},\beta'_{3,2},\beta'_{3,3},\beta'_{3,4}) \]
		where $\alpha_{2,3}$ has degree $2$, the $\beta_{3,j}$ have degree $3$ and their explicit expressions are
		\begin{align*}
			\alpha_{2,3}&=\vartheta_1(\lambda_1+\vartheta_1),\\
			\beta'_{3,1}&=10\lambda_1^3-2\lambda_1^2\vartheta_1,\\
			\beta'_{3,2}&=\psi_1\lambda_1^2-2\psi^2(\lambda_1-\psi_1)\\
			\beta'_{3,3}&=(\frac{1}{2}\lambda_1^2-\psi_1(\lambda_1-\psi_1))(\vartheta_1+\lambda_1-\psi_1),\\
			\beta'_{3,4}&=2\psi_1(\lambda_1 +\vartheta_1)(7\psi_1-\lambda_1) - 24\psi_1^3.\\
		\end{align*}
	\end{corollary}
	\begin{proof}
		One can express $\vartheta_2$ in terms of the other generators using the relation $\alpha_1$, and we obtain $\vartheta_2=\lambda_2-\psi_1(\lambda_1-\psi_1)$. We also have the relation $\frac{1}{24}\alpha_2=\frac{1}{2}\lambda_1-\lambda_2$, hence $\lambda_1$, $\psi_1$ and $\vartheta_1$ are enough to generate the rational Chow ring. The relations $\beta'_{j}$ are obtained from $\beta_j$ by substituting $2\lambda_2$ with $\lambda_1$ and $\vartheta_2$ with $\frac{1}{2}\lambda_1^2-\psi_1(\lambda_1-\psi_1)$.
	\end{proof}
	
	\subsection{Comparison with Faber's computation}
	In his thesis \cite{Fab} Faber computed the rational Chow ring of $\overline{M}_{2,1}$ over $\mathbb{C}$. Using his notation, let $[(c)]_Q$ be the fundamental class of the locus of banana curves with a component of genus $1$ and a marked component of genus $0$, and let $[(d)]_Q$ be the fundamental class of the locus of curves with two elliptic tails connected by a marked $\PP^1$. 
	
	Denote $\delta_0$ the fundamental class of marked curves with a non-separating node, and let $\delta_1$ be the fundamental class of marked curves with a separating node (this would be $\vartheta_1$ in our notation).
	\begin{theorem}[Faber]
		Suppose that the ground field has characteristic zero. Then the rational Chow ring of $\overline{M}_{2,1}$ is given by
		\[\ch(\overline{M}_{2,1})_{\mathbb{Q}} \simeq \mathbb{Q}[\psi_1,\delta_0,\delta_1]/(\gamma_{2,1},\gamma_{3,1},\gamma_{3,2},\gamma_{3,3},\gamma_{3,4}) \]
		where the $\gamma_{i,j}$ have degree $i$ and they are given by
		\begin{align*}
			\gamma_{2,1}&=(\delta_0 + 12 \delta_1) \delta_1, \\
			\gamma_{3,1}&=3\delta_0^3 + 11\delta_0^2\delta_1, \\
			\gamma_{3,2}&=\delta_1 ( \delta_1^2 + 2 \delta_1 \psi_1 + 4 \psi_1^2 ), \\
			\gamma_{3,3}&=\psi_1[(c)]_Q, \\
			\gamma_{3,4}&=\psi_1[(d)]_Q.
		\end{align*}
	\end{theorem}
	Two of the three generators picked by Faber coincide with the ones chosen by us, namely $\psi_1$ and $\delta_1=\vartheta_1$. The last generator $\delta_0$ is well known to be equal to $10\lambda_1-2\vartheta_1$. A straightforward computation with Macaulay2 shows that the ring computed by Faber and the one given in \Cref{cor:rational chow} are isomorphic.
	
	Regarding the relations, we have $\gamma_{2,1}=10\alpha_{2,1}$: this follows by simply substituting $\delta_0$ with $10\lambda_1-2\vartheta_1$ (and of course $\delta_1$ with $\vartheta_1$).
	
	We also have that $\gamma_{3,4}$ and $\frac{1}{2}\beta'_{3,2}$ coincide. Indeed, we can rewrite $\frac{1}{2}\beta'_{3,2}$ as $\frac{1}{2}\beta_{3,2}=\psi_1\vartheta_2$, and $\vartheta_2$ is by definition the fundamental class of the locus of curves with a marked separating node, if we regard $\overline{M}_{2,1}$ as the coarse space of the universal curve over $\Mbar_2$. On the other hand, we can also regard $\overline{M}_{2,1}$ as the coarse moduli space of stable marked curves of genus two: then $\vartheta_2$ coincides precisely with $[(d)]_Q$, hence $\psi_1\vartheta_2=\psi_1[(d)]_Q$.
	
	Moreover, we have $\gamma_{3,3}=\frac{1}{2}\beta'_{3,4}$. Indeed $\gamma_{3,3}=\psi_1[(c)]_Q$ and one can compute the class of $[(c)]$ explicitly as the class of the locus of curves with a marked non-separating node, obtaining
	\[ [(c)]_Q=(\lambda_1 +\vartheta_1)(7\psi_1-\lambda_1) - 12\psi_1^2, \]
	from which the equality above follows.
	
	For what concerns the last two relations found by Faber, they can be expressed in terms of our relations as follows:
	\begin{align*}
		\gamma_{3,1}&=20\bigl((\vartheta_1-5\lambda_1)\alpha_{2,3} + 15\beta'_{3,1}\bigr) \\
		\gamma_{3,2}&=(2\psi_1+\vartheta_1-\lambda_1)\alpha_{2,3} - \frac{1}{12}\beta'_{3,1} + \frac{17}{6}\beta'_{3,2} + \frac{5}{6}\beta'_{3,3} + \frac{1}{6}\beta'_{3,4}.
	\end{align*}
	
\end{document}